\documentclass[11pt,reqno]{article}
\usepackage{amssymb,amsmath}
\usepackage{amsfonts}
\usepackage{amsbsy}
\usepackage{amsopn} 
\usepackage{amsthm}
\usepackage{amscd}
\usepackage{verbatim}

\usepackage{color}

\numberwithin{equation}{section}
\newtheorem{theorem}{Theorem}[section]
\newtheorem{remark}[theorem]{Remark}
\newtheorem{lemma}[theorem]{Lemma}
\newtheorem{cor}[theorem]{Corollary}
\newtheorem{prop}[theorem]{Proposition}
\newtheorem{definition}[theorem]{Definition}

\newtheorem{assumption}[theorem]{Assumption}

\newcommand{\eps}{\varepsilon }

\newcommand{\rd}{{\mathbb{R}}^{d}}

\newcommand{\ind}[1]{{1\! \! 1}_{#1}  }

\newcommand{\curl}{\text{\rm curl}\,}

\newcommand{\supp}{\text{\rm supp}\,}

\newcommand{\lhs}{{\mathcal{L}}_{(2)}}

\newcommand{\pcal}{\mathcal{P}}

\newcommand{\tcal}{\mathcal{T}}

\newcommand{\vcal}{\mathcal{V}}
\newcommand{\xcal}{\mathcal{X}}
\newcommand{\ycal}{\mathcal{Y}}
\newcommand{\zcal}{\mathcal{Z}}

\newcommand{\Bbold}{\mathbf{B}}

\newcommand{\Cbold}{\mathbf{C}}

\newcommand{\Gbold}{\mathbf{G}}

\newcommand{\ubold}{\mathbf{u}}

\newcommand{\vbold}{\mathbf{v}}

\newcommand{\p}{\mathbb{P}}

\newcommand{\e}{\mathbb{E}}

\newcommand{\norm}[3]{{\|  #1 \| }_{#2}^{#3}}
\newcommand{\Norm}[3]{{\Bigl\|  #1 \Bigr\| }_{#2}^{#3}}
\newcommand{\ilsk}[3]{{( #1 | #2 )}_{#3}}

\newcommand{\dual}[3]{{\langle #1 | #2 \rangle}_{#3}}
\newcommand{\Dual}[3]{{\Bigl< #1 \bigl| #2 \Bigr>}_{#3}}
\newcommand{\dirilsk}[3]{{( \! ( #1 | #2 ) \! ) }_{#3}}

\newcommand{\nnorm}[3]{{|  #1 |}_{#2}^{#3}}
\newcommand{\Nnorm}[3]{{\Bigl|  #1 {\Bigr| }_{#2}^{#3}}}

\newcommand{\lb}{\langle}
\newcommand{\rb}{\rangle}
\newcommand{\ddual}[4]{{}_{#1}\lb #2 |#3 {\rb }_{#4}}

\newcommand{\Hn}{{\Hmath }_{n}}
\newcommand{\Fn}{{F}_{n}}

\newcommand{\diver}{\text{\rm div\,}}

\newcommand{\Pn}{{P}_{n}}
\newcommand{\un}{{u}_{n}}
\newcommand{\Jn}[1]{{J}_{n}^{#1}}

\newcommand{\ball}{\mathbb{B}}

\newcommand{\taun}{{\tau}_{n}}

\newcommand{\mhd}{\widetilde{b}}
\newcommand{\MHD}{\widetilde{B}}
\newcommand{\Hall}{\mathcal{H}}
\newcommand{\tHall}{\widetilde{\Hall}}
\newcommand{\hall}{\mathfrak{h}}
\newcommand{\thall}{\widetilde{\mathfrak{h}}}

\newcommand{\X}{X}
\newcommand{\Xn}{X_n}

\newcommand{\bXn}{{\bar{X}}_{n}}
\newcommand{\bX}{\bar{X}}
\newcommand{\Xast}{{\X }_{\ast }}

\newcommand{\taunR}{{\tau }_{R}^{n}}

\newcommand{\Law}{\mathrm{Law}}

\newcommand{\ARiesz}{{A}_{n}}
\newcommand{\BRiesz}{{\MHD }_{n}}
\newcommand{\RRiesz}{{\tHall }_{n}}

\newcommand{\Vast}{{\Vmath }_{\ast}}

\newcommand{\Cmath}{\mathbb{C}}
\newcommand{\Dmath}{\mathbb{D}}
\newcommand{\Fmath}{\mathbb{F}}
\newcommand{\Hmath}{\mathbb{H}}
\newcommand{\Kmath}{\mathbb{K}}
\newcommand{\Lmath}{\mathbb{L}}
\newcommand{\Smath}{\mathbb{S}}
\newcommand{\Umath}{\mathbb{U}}
\newcommand{\Vmath}{\mathbb{V}}

\newcommand{\Uprime}{{\Umath }^{\prime }}
\newcommand{\Vprime}{{\Vmath }^{\prime }}
\newcommand{\Vastprime}{{\Vmath}_{\ast }^{\prime }}

\newcommand{\Vtest}{{\Vmath }_{1,2}}

\newcommand{\Hsol}[1]{{V}_{#1}}

\newcommand{\Sn}{{S}_{n}}
\newcommand{\Ldn}{{L}^{2}_{n}}

\begin{document}
\title{\bf Existence of a martingale solution of the stochastic Hall-MHD  equations perturbed by Poisson type random forces on ${\mathbb{R}}^{3}$  }
\author{El\.zbieta Motyl 
\footnote{Faculty of Mathematics and Computer Science, University of \L\'{o}d\'{z},  Poland; 
email: elzbieta.motyl@wmii.uni.lodz.pl}}

\maketitle

\begin{abstract}
Stochastic Hall-magne\-to\-hydro\-dynamics  equations  on ${\mathbb{R}}^{3}$ with random forces expressed in terms of the  time homogeneous Poisson random measures are considered. We prove the existence of a global martingale solution. The construction of a solution is based on the Fourier truncation method,
stochastic compactness method and a version of the Skorokhod theorem for non-metric spaces adequate for L\'{e}vy type random fields.
\end{abstract}


\medskip  \noindent
\bf Keywords \rm  Hall-MHD equations; Poisson random measure, martingale solution, compactness method

\medskip
\section{Introduction}

\medskip  \noindent
We consider the following stochastic Hall-magnetohydrodynamics system on $[0,T] \times {\mathbb{R} }^{3}$ with Poisson type noises
\begin{align}
 d\ubold  +& \Bigl[ (\ubold \cdot \nabla  ) \ubold + \nabla p 
- s \,(\Bbold \cdot \nabla ) \Bbold +s \, \nabla \Bigl( \frac{{|\Bbold |}^{2}}{2} \Bigr) 
- {\nu }_{1} \, \Delta \ubold \Bigr] \, dt \notag \\
\; &= \;  \int_{Y_1}{F}_{1}(t,\ubold ({t}^{-});y) \, d{\tilde{\eta }}_{1}(dt,dy),
\label{eq:Hall-MHD_u_Poisson} \\
  d \Bbold  +& \Bigl[ (\ubold \cdot \nabla ) \Bbold - (\Bbold \cdot \nabla ) \ubold 
+\eps \, \curl [(\curl \Bbold ) \times \Bbold ]  
- {\nu }_{2} \, \Delta \Bbold \Bigr] \, dt  \notag \\
\; &= \; \int_{Y_2}{F}_{2}(t,\Bbold ({t}^{-});y) \, d{\tilde{\eta }}_{2}(dt,dy) 
\label{eq:Hall-MHD_B_Poisson} 
\end{align}
with the incompressibility conditions
\begin{equation}
\begin{split}
\diver \ubold \; &= \; 0 \quad \mbox{ and } \quad \diver \Bbold \; = \; 0
\end{split}
\label{eq:Hall-MHD_incompressibility}
\end{equation}
and the initial conditions
\begin{equation}
\ubold (0) \; = \; {\ubold }_{0} \quad \mbox{ and } \quad \Bbold (0) \; = \; {\Bbold }_{0}.
\label{eq:Hall-MHD_ini-cond}
\end{equation}
In this problem $\ubold (t,x) = (u_1,u_2,u_3)(t,x)$, $\Bbold (t,x) = (B_1,B_2,B_3)(t,x)$ for $(t,x)\in [0,T] \times {\mathbb{R} }^{3}$, are  three-dimensional vector fields representing velocity and magnetic fields, respectively, and  $p(t,x) $ is a real valued function representing the pressure of the fluid.
Random forces are defined as the integrals $\int_{Y_1}{F}_{1}(t,\ubold ({t}^{-});y) \, d{\tilde{\eta }}_{1}(dt,dy)$ and  $\int_{Y_2}{F}_{2}(t,\Bbold ({t}^{-});y) \, d{\tilde{\eta }}_{2}(dt,dy)$ with respect to the compensated time homogeneous Poisson random measures
 ${\tilde{\eta }}_{1}$ and ${\tilde{\eta }}_{2}$.
The positive constants ${\nu }_{1},{\nu }_{2}$ and $s$ stand for kinematic viscosity, resistivity and the Hartmann number, respectively.
Let us recall that the $\curl $-operator  for a vector field $\phi : {\mathbb{R} }^{3} \to {\mathbb{R} }^{3}$ is defined by
$
\curl \phi  :=  \nabla \times \phi .
$
The expression
\[
\eps \, \curl [(\curl \Bbold ) \times \Bbold ] 
\]
represents the Hall-term with the Hall parameter $\eps >0$. 
For simplicity we will assume that $s=1$ and $\eps =1$.

\medskip  \noindent
Mathematical treatment of Hall-MHD equations, which are important in the physics of plasma, were introduced in \cite{Acheritogaray+Degond'11}. Since then this subject receives an increasing attention. Let us mention several papers where deterministic Hall-MHD model is considered, e.g. \cite{Chae+Degond+Liu'14}, 
\cite{Chae+Lee'14}, \cite{Chae+Schonbek'13}, \cite{Chae+Wan+Wu'15}.
Stochastic Hall-MHD equtions perturbed by Wiener noise were studied in \cite{Yamazaki'17}, \cite{Yamazaki'19a}, \cite{Yamazaki'19} and \cite{EM'21_ArX_Hall-MHD_Gauss}.

\medskip  \noindent
Recently, Yamazaki and Mohan \cite{Yamazaki+Mohan'19} proved the existence and uniqueness of a local in time solution for the stochastic Hall-MHD system forced by the L\'{e}vy noise. The solution is global provided that the initial data is sufficiently small. 

\medskip  \noindent
In the present paper we are interested in the existence of a global martingale solution of problem \eqref{eq:Hall-MHD_u_Poisson}-\eqref{eq:Hall-MHD_ini-cond} (without limitations on the smallness of the initial data). 
Proceeding similarly as in \cite{EM'21_ArX_Hall-MHD_Gauss}, where the Gaussian type noise terms are considered,
we rewrite problem \eqref{eq:Hall-MHD_u_Poisson}-\eqref{eq:Hall-MHD_ini-cond}  in the form of the  equation
\begin{equation}
\begin{split} 
&  \X (t)  + \int_{0}^{t} \bigl[  \mathcal{A} \X (s)  + \MHD  (\X (s)) +\tHall (\X (s)) \bigr] \, ds  \\
& \qquad   \; = \; {\X }_{0}
 + \int_{0}^{t}\int_{Y}  F(s,\X ({s}^{-});y ) \, \tilde{\eta }(ds,dy)  , \qquad t \in [0,T] ,   
\end{split}  \label{eq:Hall-MHD_functional_Intro}
\end{equation} 
where $\X = (\ubold ,\Bbold )$ and ${\X }_{0} := ({\ubold }_{0}, {\Bbold }_{0})$.
Besides,  $\mathcal{A} $, $\MHD $ and $\tHall $ are the maps corresponding to the Stokes-type operators, the MHD-term and the Hall-term, respectively, defined in Section \ref{sec:Hall-MHD_funct-setting}.
In fact,  the deterministic part of problem \eqref{eq:Hall-MHD_functional_Intro} is the same as in \cite{EM'21_ArX_Hall-MHD_Gauss}. Analysis of the stochastic part is different.

\medskip  \noindent
The main result is stated in Theorem \ref{th:mart-sol_existence}.
We prove the existence of a global martingale solution of problem \eqref{eq:Hall-MHD_functional_Intro} understood as a system $(\Omega ,\mathcal{F} , \Fmath ,\p ,\eta ,\X )$, where $(\Omega ,\mathcal{F} , \Fmath , \p )$ is a filtered probability space, $\eta $ is a time-homogeneous Poisson random measure and $\X = {\{ \X (t)\} }_{t \in [0,T]}$ is a stochastic process. The trajectories of the process $\X $ belong, with probability $1$, to space of c\`{a}dl\`{a}g functions $\Dmath ([0,T];{\Hmath }_{w})$, where the subscript $w$ indicates the weak topology in $\Hmath $, see Definition \ref{def:mart-sol}. 
Besides, we prove that the stochastic process $\X $ satisfies, in particular, the energy  
inequality
\[
\e \Bigl[ \sup_{t\in [0,T]} \nnorm{\X (t)}{\Hmath }{2} + \int_{0}^{T} \norm{\X (t)}{\Vmath }{2} \Bigr] < \infty 
\] 
The spaces $\Hmath \subset {\Lmath }^{2}$ and $\Vmath \subset {\Hmath }^{1}$, defined in Section \ref{sec:Hall-MHD_funct-setting}, are products of appropriate spaces of divergence free vector fields corresponding to the incompressibility conditions from \eqref{eq:Hall-MHD_incompressibility}.  The problem of uniqueness  is still open.

\medskip  \noindent
Similarly to \cite{EM'21_ArX_Hall-MHD_Gauss}, the construction of a solution is based on the approximation which uses the Fourier transform. Problem \eqref{eq:Hall-MHD_functional_Intro} is approximated by a sequence of the  truncated SPDEs perturbed by appropriate Poisson type random forces, in the infinite dimensional Hilbert spaces ${\Hmath }_{n}$, $n\in \mathbb{N} $. The solutions $\Xn $ of the truncated equations generate a tight sequence of probability measures $\{ \Law (\Xn ), n \in \mathbb{N}  \} $ on appropriate functional space $\zcal $ defined by \eqref{eq:Z_cadlag} in Section \ref{sec:comp-tight}. In comparison to \cite{EM'21_ArX_Hall-MHD_Gauss} in the definition of the space $\zcal $, the spaces of continuous functions are replaced by the spaces of c\`{a}dl\`{a}g functions with the Skorokhod topology, see Section \ref{sec:comp-tight} and Appendix \ref{sec:Cadlag_functions}. Besides, the construction of the new stochastic basis and new stochastic process is different from that used in \cite{EM'21_ArX_Hall-MHD_Gauss}. It is based on the version of the Skorokhod theorem  adequate for Poisson type noises. It is closely related to the techniques used in \cite{EM'13} in the context of Navier-Stokes equation and in \cite{EM'14} for other hydrodynamic equations and uses Corollary \ref{cor:Skorokhod_J,B,H} stated in Appendix \ref{sec:Skorokhod_th}. Because of the Hall term, the framework considered in  \cite{EM'14} does not cover directly the Hall-MHD system. 

\medskip  \noindent
In Section \ref{sec:general} we consider also the Hall-MHD equation with more general L\'{e}vy noise, i.e. 
\begin{equation*}
\begin{split} 
&  \X (t)  +  \int_{0}^{t} \bigl[  \mathcal{A} \X (s)  + \MHD  (\X (s)) +\tHall (\X (s)) \bigr] \, ds  \\
&   \; = \; {\X }_{0} + \int_{0}^{t}  f (s) \, ds + \int_{0}^{t}\int_{Y}  F(s,\X ({s}^{-});y ) \, \tilde{\eta }(ds,dy)
 + \int_{0}^{t} G (s,\X (s) ) \, dW(s)  , \quad t \in [0,T] ,
\end{split} 
\end{equation*}
where $f$ represents deterministic forces and $W$ is a Wiener process, and state the theorem about the existence of a martingale solution. 

\medskip  \noindent
The Fourier truncation method is very useful in the study of differential equations. 
Let us recall the paper \cite{Feff+McCorm+Rob+Rod'2014}, where Fefferman and co-authors use this method in deterministic MHD equations on ${\mathbb{R}}^{d}$, $d=2,3$. The same idea, referred to as the Friedrichs method, is used by Bahouri, Chemin and Danchin \cite[Section 4]{Bahouri+Chemin+Danchin'11} to study other class of deterministic differential equations.

\medskip  \noindent
In the context of SPDEs, Mohan and Sritharan \cite{Mohan+Sritharan'16} apply the Fourier truncation method to study  the stochastic Euler equation with L\'{evy} noise. This method is also used by Manna, Mohan and Sritharan \cite{Manna+Mohan+Srith'2017} in the stochastic  MHD equations with L\'{e}vy noise and 
by Yamazaki and Mohan in \cite{Yamazaki+Mohan'19}.
Moreover, Brze\'{z}niak and Dhariwal \cite{Brze+Dha'20} apply the truncated approximation in the stochastic tamed Navier-Stokes equations.

\medskip  \noindent
The paper is organized  as follows. In Section \ref{sec:funct-setting} we recall the functional setting of the Hall-magnetohydrodynamics equations. The main result is stated in Section \ref{sec:statement}. 
In Section \ref{sec:truncated_eq} we consider approximate equations and prove a priori estimates. 
Compactness and tightness theorems used in the proof of the main theorem are presented in Section \ref{sec:comp-tight}. Section \ref{sec:existence} is devoted to the proof of the existence of a global martingale solution. Version of the Skorokhod theorem used in the proof of Theorem \ref{th:mart-sol_existence} is recalled in Appendix 
\ref{sec:Skorokhod_th}.
Some auxiliary results related to the Fourier analysis are contained in Appendix \ref{sec:Fourier_truncation}.

\medskip
\section{Functional setting of the Hall-magnetohydrodynamics equations} \label{sec:funct-setting}

\medskip  \noindent
Let ${\mathcal{C} }^{\infty }_{c} = {\mathcal{C} }^{\infty }_{c}({\mathbb{R} }^{3},{\mathbb{R} }^{3})$ denote the space of all ${\mathbb{R} }^{3}$-valued functions of class ${\mathcal{C} }^{\infty }$ with compact supports in ${\mathbb{R} }^{3}$, and let
\begin{align}
&\mbox{ $\vcal := \{ u \in {\mathcal{C} }^{\infty }_{c} ({\mathbb{R} }^{3},{\mathbb{R} }^{3}): \; \diver u =0  \} $ },
\label{eq:vcal} \\
& \mbox{ $H := $ the closure of $\vcal $ in $L^2({\mathbb{R} }^{3},{\mathbb{R} }^{3})$}, 
\label{eq:H}\\
&\mbox{ $V := $ the closure of $\vcal $ in $H^1({\mathbb{R} }^{3},{\mathbb{R} }^{3})$}.
\label{eq:V}
\end{align}
In the space $H$ we consider the inner product and the norm inherited from $L^2({\mathbb{R} }^{3},{\mathbb{R} }^{3})$
and denote them by $\ilsk{\cdot }{\cdot }{H}$ and $\nnorm{\cdot }{H}{}$, respectively, i.e.
\[
\ilsk{u}{v}{H} := \ilsk{u}{v}{L^2} , \qquad \nnorm{u}{H}{} := \nnorm{u}{L^2}{}, \qquad u,v \in H .
\]
In the space $V$ we consider the inner product inherited from ${H}^{1}({\mathbb{R} }^{3},{\mathbb{R} }^{3})$, i.e.
\begin{equation*}
\ilsk{u}{v}{V} := \ilsk{u}{v}{H}  + \ilsk{\nabla u}{\nabla v}{L^2} , \qquad u,v \in V ,
\end{equation*}
where 
\begin{equation*}
 \ilsk{\nabla u}{\nabla v}{L^2} 
= \sum_{i=1}^{3} \int_{{\mathbb{R} }^{3}} \frac{\partial u}{\partial x_i} \cdot \frac{\partial v}{\partial x_i} \, dx ,  
\end{equation*}
and the norm induced by $\ilsk{\cdot }{\cdot }{V}$, i.e.
\begin{equation*}
\norm{u}{V}{} \; := \; \bigl( \nnorm{u}{H}{2} + \nnorm{\nabla u}{L^2}{2} {\bigr) }^{\frac{1}{2}} .
\end{equation*}
For any $m\ge 0$ consider the following  of Hilbert space
\begin{equation}
{V}_{m} \; := \;  \mbox{the closure of $\vcal $ in ${H}^{m}({\mathbb{R} }^{3} , {\mathbb{R} }^{3} )$} 
\label{eq:V_m}
\end{equation}
with the inner product inherited from the space ${H}^{m}({\mathbb{R} }^{3} , {\mathbb{R} }^{3})$.
Note that, ${V}_{0}=H$ and ${V}_{1}=V$.

\medskip
\noindent
\bf Notations. \rm 
Let $(X,\norm{\cdot }{X}{}) $, $(Y,\norm{\cdot }{Y}{}) $ be two normed spaces. By $\mathcal{L} (X,Y)$ we denote the space of all bounded linear operators from $X$ to $Y$. If $Y=\mathbb{R} $, the $X':= \mathcal{L} (X, \mathbb{R} )$ is called the dual space  of $X$. The standard duality pairing is denoted by $\ddual{X'}{\cdot }{\cdot }{X}$. If no confusion seems likely we omit the subscripts $X'$ and $X$ and write $\dual{\cdot }{\cdot }{}$.

\medskip
\subsection{Spaces used in the Hall-MHD equations} \label{sec:Hall-MHD_funct-setting}

\medskip  \noindent
Spaces used in the theory of Hall-magnetohydrodynamics equations are products of the spaces $H$ and $V$ defined by \eqref{eq:H} and \eqref{eq:V}, respectively. Namely, 
\begin{equation}
\Hmath \; := \; H \times H , \qquad \Vmath := V \times V 
, \qquad   {\Vmath }^{\prime } \; := \;  \text{the dual space of $\Vmath $}
\label{eq:Hall-MHD_H-V}
\end{equation}
The spaces $\Hmath $ and $\Vmath $ are Hilbert spaces with the following inner products
\begin{align*}
\ilsk{\phi }{\psi }{\Hmath }  
\; &:= \;  \ilsk{\ubold }{\vbold }{{L}^{2}} + \ilsk{\Bbold }{\Cbold }{{L}^{2}} 
\end{align*}
for all $ \phi =(\ubold ,\Bbold ), \, \, \;  \psi =(\vbold ,\Cbold ) \in \Hmath $ ,
and
\begin{equation*}
\begin{split}
\ilsk{\phi }{\psi }{\Vmath }
\; &:= \;   \ilsk{\phi }{\psi }{\Hmath }   + \dirilsk{\phi }{\psi }{}   
\end{split}
\end{equation*} 
for all $ \phi  =(\ubold ,\Bbold ) , \;  \psi =(\vbold ,\Cbold )  \in \Vmath $,
where
\begin{equation}
\dirilsk{\phi }{\psi }{} \; := \; {\nu }_{1} \, \ilsk{\nabla \ubold }{\nabla \vbold }{L^2} 
+{\nu }_{2} \, \ilsk{\nabla \Bbold }{\nabla \Cbold }{L^2} .
\label{eq:il-sk_Dirichlet+curl}  
\end{equation}
In the spaces $\Hmath $ and $\Vmath $ we consider the norms induced by the inner products $\ilsk{\cdot }{\cdot }{\Hmath }$ and $\ilsk{\cdot }{\cdot }{\Vmath }$, respectively, i.e.
$\nnorm{\phi }{\Hmath }{2}:= \ilsk{\phi }{\phi }{\Hmath } $ for $\phi \in \Hmath $, and
\begin{equation}
\norm{\phi }{\Vmath }{2} \; = \; \nnorm{\phi }{\Hmath }{2} + \norm{\phi }{}{2},
\label{eq:Hall-MHD_V-norm} 
\end{equation}
where
\begin{equation*}
\norm{\phi }{}{2} \; := \; \dirilsk{\phi }{\phi }{}, \qquad \phi \in \Vmath . 
\end{equation*}

\medskip  
\noindent
Let $\mathcal{A} $  be the operator  defined by 
\begin{equation}
\dual{\mathcal{A} \phi }{\psi }{} \; = \;  \dirilsk{\phi }{\psi }{}, \qquad \phi ,\psi \in \Vmath , 
\label{eq:A_acal_rel} 
\end{equation}
where $\dirilsk{\cdot  }{\cdot }{}$ is given by \eqref{eq:il-sk_Dirichlet+curl} .

\medskip  
\begin{remark}  \label{rem:Acal-term_properties} 
It is clear that  $\mathcal{A} \in \mathcal{L} (\Vmath ,\Vmath ')$
and
\begin{equation}
\nnorm{\mathcal{A} \phi }{\Vmath '}{} \; \le \; \norm{\phi }{}{}, \qquad \phi \in \Vmath .
\label{eq:Acal_norm}
\end{equation}
\end{remark}

\medskip  \noindent
For $m_1, m_2 \ge 0 $ let us define
\begin{equation}
{\Vmath }_{m_1,m_2} \; := \; {V}_{m_1} \times {V}_{m_2},
\label{eq:Vmath_m1,m2}
\end{equation}
where ${V}_{m_1}, {V}_{m_2}$ are the spaces defined by \eqref{eq:V_m}. In ${\Vmath }_{m_1,m_2}$ we consider the product norm
\begin{equation}
\norm{\phi }{{m_1,m_2} }{2}  \; := \; 
\norm{\ubold }{V_{m_1}}{2} + \norm{\Bbold }{V_{m_2}}{2}
\label{eq:Vmath_m1,m2-norm} 
\end{equation}
for all $\phi =(\ubold , \Bbold ) \in {\Vmath }_{m_1,m_2}$.
In the case when $m_1=m_2=:m$ we denote
\begin{equation}
{\Vmath }_{m} \; := \; V_m \times V_m
\quad \mbox{ and } \quad \norm{\cdot }{m}{} \; := \; \norm{\cdot }{m,m}{} .
\label{eq:Vmath_m}
\end{equation}
Note that if $m=1$, then ${\Vmath }_{1}=\Vmath $ and $\norm{\cdot }{1}{} = \norm{\cdot }{\Vmath }{}$. 

\medskip  \noindent
By ${L}^{2}_{loc}({\mathbb{R} }^{3},{\mathbb{R} }^{3})$, we denote the space of all Lebesgue measurable functions $v:{\mathbb{R} }^{3} \to {\mathbb{R} }^{3}$ such that $\int_{K} {|v(x)|}^{2} \, dx < \infty $
for every compact subset $K \subset {\mathbb{R} }^{3}$,  equipped with the Fr\'{e}chet topology generated by the family of seminorms
\begin{equation}
{p}_{R}(v) \; := \; \Bigl( \int_{{\mathcal{O} }_{R}} {|v(x)|}^{2} \, dx {\Bigr) }^{\frac{1}{2}}, \qquad R \in \mathbb{N} ,
\label{eq:p_R-seminorms}
\end{equation} 
where ${({\mathcal{O} }_{R})}_{R \in \mathbb{N} }$ is an increasing sequence of open bounded subsets of ${\mathbb{R} }^{3}$ with smooth boundaries and such that $\bigcup_{R\in \mathbb{N} } {\mathcal{O} }_{R} = {\mathbb{R} }^{3}$.

\medskip \noindent
By ${H}_{loc}$, we denote the space $H$ defined by \eqref{eq:H},  endowed with the Fr\'{e}chet topology inherited from the space
${L}^{2}_{loc}({\mathbb{R} }^{3},{\mathbb{R} }^{3})$.

\medskip  \noindent
Let us fix $T>0$. By ${L}^{2}(0,T;{H}_{loc})$ we denote the space of measurable functions 
 $ u :[0,T] \to H   $ such that for all $ R \in \mathbb{N} $
\begin{equation}
{p}_{T,R}(u) \; = \; \Bigl(  \int_{0}^{T} \int_{{\mathcal{O} }_{R}}  {|u  (t,x)|}^{2}  \, dxdt {\Bigr) }^{\frac{1}{2}} < \infty  .
\label{eq:seminorms-L^2(0,T;H_loc)} 
\end{equation}
In ${L}^{2}(0,T;{H}_{loc})$ we consider the topology   generated by the seminorms 
$({p}_{T,R}{)}_{R\in \mathbb{N} } $ defined by \eqref{eq:seminorms-L^2(0,T;H_loc)}.  

\medskip  \noindent
By ${L}^{2}(0,T;{\Hmath }_{loc})$ we denote the space of measurable functions 
 $ \phi :[0,T] \to \Hmath    $ such that for all $ R \in \mathbb{N} $
\begin{equation}
{p}_{T,R}(\phi) \; := \; \Bigl(  \int_{0}^{T} \int_{{\mathcal{O} }_{R}} 
\bigl[ {|\ubold (t,x)|}^{2}  + {|\Bbold (t,x)|}^{2} \bigr] \, dxdt {\Bigr) }^{\frac{1}{2}} < \infty  .
\label{eq:seminorms} 
\end{equation}
where $\phi =(\ubold ,\Bbold )$.
In the space ${L}^{2}(0,T;{\Hmath }_{loc})$ we consider the topology   generated by the seminorms  $({p}_{T,R}{)}_{R\in \mathbb{N} } $ defined by \eqref{eq:seminorms}.   

\medskip  \noindent
\subsection{The form $\mhd$ and the operator $\MHD$}

\medskip
We recall the standard maps used in the theory of magnetohydrodynamic equation, see 
Sermange and Temam \cite{Sermange+Temam'83} and Sango \cite{Sango'10}.
Let us consider the following tri-linear form
\begin{equation}  
b(u,w,v ) = \int_{{\mathbb{R} }^{3}}\bigl( u \cdot \nabla w \bigr) v \, dx .
\label{eq:b-form}
\end{equation}
For the properties of the form $b$, see Temam \cite{Temam'79}.
Using the form $b$  defined by \eqref{eq:b-form} we will consider the tri-linear  form  $\mhd$ on $\Vmath \times \Vmath \times \Vmath $, where $\Vmath $ is defined by \eqref{eq:Hall-MHD_H-V},  Namely,
\begin{align*}
\mhd ({\phi }^{(1)},{\phi }^{(2)},{\phi }^{(3)})
\; := \;  & b({\ubold }^{(1)}, {\ubold }^{(2)}, {\ubold }^{(3)})
- b({\Bbold }^{(1)}, {\Bbold }^{(2)}, {\ubold }^{(3)})  \\
& + b({\ubold }^{(1)}, {\Bbold }^{(2)}, {\Bbold }^{(3)})
 -  b({\Bbold }^{(1)}, {\ubold }^{(2)}, {\Bbold }^{(3)}) ,
\end{align*}
where ${\phi }^{(i)} = ({\ubold }^{(i)},{\Bbold }^{(i)}) \in \Vmath $, $i=1,2,3$. 
Since the form $b$ is continuous on $V\times V \times V$,   the form $\mhd$ is continuous on $\Vmath \times \Vmath \times \Vmath $.
Moreover,  the form $\mhd$ has the following properties
\begin{align*} 
&  \mhd ({\phi }^{(1)},{\phi }^{(2)},{\phi }^{(3)})
\; = \; -  \mhd ({\phi }^{(1)},{\phi }^{(3)},{\phi }^{(2)}) , \qquad 
{\phi }^{(i)} \in \Vmath , \quad i=1,2,3  
\end{align*}
and in particular
\begin{align*}  
& \mhd ({\phi }^{(1)},{\phi }^{(2)},{\phi }^{(2)}) \; = \; 0 , \qquad {\phi }^{(1)}, {\phi }^{(2)} \in \Vmath .
\end{align*}        
Now, let us define a bilinear map $\MHD$ by 
\begin{equation}   
\MHD (\phi ,\psi ) \; := \;  \mhd (\phi ,\psi ,\cdot )  ,  \qquad \qquad \phi ,\psi  \in \Vmath . 
\label{eq:MHD-map}
\end{equation}

\medskip  \noindent
We will also use the notation $ \MHD (\phi ) :=  \MHD (\phi ,\phi ) $.

\medskip
\noindent
Let us recall   properties of the map $\MHD$ stated in \cite{EM'14}.

\medskip
\begin{lemma} \label{lem:MHD-term_properties}
\rm (See \cite[Lemma 6.4]{EM'14}) \it 
\begin{description}
\item[(i) ] There exists a constant ${c}_{\MHD } >0 $ such that 
\[
\nnorm{ \MHD(\phi , \psi )}{\Vmath '}{} 
\;  \le \;  {c}_{\MHD }\norm{\phi }{\Vmath }{} \norm{\psi }{\Vmath }{}  , \qquad \phi ,\psi  \in \Vmath . 
\]
In particular, the map $\MHD: \Vmath \times \Vmath \to \Vmath '$ is bilinear and continuous.
Moreover, 
\[
\dual{\MHD(\phi ,\psi )}{\theta  }{} \; = \; - \dual{\MHD(\phi ,\theta  )}{\psi  }, 
\qquad \phi ,\psi , \theta  \in \Vmath , 
\]
and, in particular,
\begin{equation}
\dual{\MHD (\phi )}{\phi }{} \; = \; 0 , 
\qquad \phi  \in \Vmath  . 
\label{eq:MHD-map_perp}
\end{equation}
\item[(ii) ] The mapping $\MHD$ is locally Lipschitz continuous on the space $\Vmath $, i.e.
for every $r>0 $ there exists a constant ${L}_{r}>0$ such that 
\[
\nnorm{ \MHD(\phi )- \MHD (\tilde{\phi } )}{\Vmath '}{} 
\; \le \; {L}_{r}  \norm{\phi -\tilde{\phi } ) }{\Vmath }{} , \qquad \phi ,\tilde{\phi } \in \Vmath ,
  \quad \norm{\phi }{\Vmath }{} , \norm{\tilde{\phi }}{\Vmath }{} \le r .
\]   
\item[(iii) ] If $m> \frac{5}{2}$, then $\MHD$ can be extended to the bilinear mapping from $\Hmath \times \Hmath $ to ${\Vmath }_{m}'$ (denoted still by $\MHD$) such that
\begin{equation}
\nnorm{ \MHD(\phi , \psi )}{{\Vmath }_{m}'}{} 
\; \le \; {c}_{\MHD }(m) \nnorm{\phi }{\Hmath }{} \nnorm{\psi }{\Hmath }{} , \qquad \phi ,\psi  \in \Hmath  ,
\label{eq:MHD-map_est-H-H} 
\end{equation}
where ${c}_{\MHD }(m)$ is a positive constant.
\end{description}
\end{lemma}

\medskip  \noindent
We will use also the following convergence result for the nonlinear term $\MHD $.

\medskip
\begin{cor} \label{cor:MHD-map_conv-aux} (See \cite[Corollary 2.8]{EM'21_ArX_Hall-MHD_Gauss}.)
 Let $\phi ,\psi  \in {L}^{2}(0,T;\Hmath )$ and let $({\phi }_n) , ({\psi }_n) \subset {L}^{2}(0,T;\Hmath )$ be two sequence bounded in ${L}^{2}(0,T;\Hmath )$ and  such that 
\[
{\phi }_{n} \; \to \;  \phi   \quad \mbox{ and } \quad {\psi }_{n} \to \psi   \quad   \mbox{ in } \quad {L}^{2}(0,T;{\Hmath }_{loc}).
\] 
If $m > \frac{5}{2}$, then for all $t \in [0,T]$ and all 
$\varphi  \in {\Vmath }_{m} $:
\[
\lim_{n \to \infty } \int_{0}^{t} \dual{\MHD ({\phi }_{n}(s),{\psi }_{n}(s))}{\varphi }{} \, ds 
\; = \; \int_{0}^{t} \dual{\MHD (\phi (s),\psi (s))}{\varphi }{} \, ds ,
\] 
where ${\Vmath }_{m}$ is the space defined by \eqref{eq:Vmath_m}.
\end{cor}

\medskip
\subsection{The form $\thall $ and the map $\tHall $}

\medskip  \noindent
Now we analyze the Hall term. By the integration by parts formula for the  $\curl $-operator, we obtain
\[
\int_{{\mathbb{R} }^{3}} \curl [u \times \curl w ] \cdot v \, dx 
\; = \; \int_{{\mathbb{R} }^{3}} [u \times \curl w ] \cdot \curl v \, dx
\]
for $u,w,v \in \vcal $.
We will use the following tri-linear form defined by 
\begin{equation}
\hall (u,w,v ) \; := \; - \int_{{\mathbb{R} }^{3}} [u \times \curl w ] \cdot \curl v \, dx 
\label{eq:hall-form}
\end{equation}
for $u,w,v \in \vcal $, see \cite[Section 2.3]{EM'21_ArX_Hall-MHD_Gauss}.

\medskip \noindent
Since $(a\times b) \cdot c = - (a\times c) \cdot b $  for $a,b,c \in {\mathbb{R} }^{3}$, we infer that  
\begin{equation}
\hall (u,v,w) \; = \; - \hall (u,w,v). 
\label{eq:hall-form_antisym}
\end{equation} 
In particular,
\begin{equation}
\hall (u,v,v) \; = \; 0 .
\label{eq:hall-form_perp}
\end{equation}
Using the form $\hall $  defined by \eqref{eq:hall-form} we define the tri-linear  form  $\thall $ on $\Vmath \times \Vmath \times \Vtest   $  by
\[
\thall ({\phi }^{(1)},{\phi }^{(2)},{\phi }^{(3)}) 
\; := \; \hall ({\Bbold }^{(1)},{\Bbold }^{(2)},{\Bbold }^{(3)}),
\]
where ${\phi }^{(i)} = ({\ubold }^{(i)},{\Bbold }^{(i)}) \in \Vmath $ for $i=1,2$, and ${\phi }^{(3)} = ({\ubold }^{(3)},{\Bbold }^{(3)}) \in \Vtest   $. Due to \eqref{eq:Vmath_m1,m2}, ${\Vmath }_{1,2} := \Hsol{1} \times \Hsol{2}  $.  
The form $\thall $ is continuous on $\Vmath \times \Vmath \times \Vtest   $.
Moreover, by \eqref{eq:hall-form_antisym} and \eqref{eq:hall-form_perp}
the form $\thall $ has the following properties
\begin{align*} 
&  \thall ({\phi }^{(1)},{\phi }^{(2)},{\phi }^{(3)})
\; = \; -  \thall  ({\phi }^{(1)},{\phi }^{(3)},{\phi }^{(2)}) , \qquad 
{\phi }^{(1)} \in \Vmath , \quad {\phi }^{(i)} \in \Vtest  , \quad i=2,3 ,  
\end{align*}
and in particular
\begin{align*}  
& \thall ({\phi }^{(1)},{\phi }^{(2)},{\phi }^{(2)}) \; = \; 0 , 
\qquad {\phi }^{(1)} \in \Vmath , \quad {\phi }^{(2)} \in \Vtest .
\end{align*}        
Now, let us define a bilinear map $\tHall $ by 
\begin{equation} 
\tHall (\phi ,\psi ) \; := \;  \thall (\phi ,\psi ,\cdot )  ,  \qquad \qquad \phi ,\psi  \in \Vmath . 
\label{eq:tHall_map} 
\end{equation}
We will also use the notation $ \tHall (\phi ) :=  \tHall  (\phi ,\phi ) $.

\medskip  \noindent
In the following lemma we state basic properties of the map  $\tHall $ proved in \cite{EM'21_ArX_Hall-MHD_Gauss}.

\medskip
\begin{lemma} \label{lem:tHall-term_properties} (See \cite[Lemma 2.9]{EM'21_ArX_Hall-MHD_Gauss}.)
\begin{description}
\item[(i) ] There exists a constant ${c}_{\tHall } >0 $ such that 
\begin{equation}
\nnorm{ \tHall (\phi , \psi )}{{\Vmath }_{1,2}'}{} 
\;  \le \;  {c}_{\tHall }\norm{\phi }{\Vmath }{} \norm{\psi }{\Vmath }{}  , \qquad \phi ,\psi  \in \Vmath . 
\end{equation}
In particular, the map $\tHall : \Vmath \times \Vmath \to {\Vmath }_{1,2}' $ is well-defined bilinear and continuous.
Moreover, 
\[
\dual{\tHall (\phi ,\psi )}{\Theta  }{} \; = \; - \dual{\tHall (\phi ,\theta  )}{\psi  }, 
\qquad \phi \in \Vmath , \quad \psi , \theta \in {\Vmath }_{1,2} ,
\]
and, in particular,
\begin{equation}
\dual{\tHall (\phi )}{\phi }{} \; = \; 0 ,  \qquad \phi  \in  {\Vmath }_{1,2}  . 
\label{eq:tHall-map_perp}
\end{equation}
\item[(ii) ] The map $\tHall $ is locally Lipschitz continuous on the space $\Vmath $, i.e.
for every $r>0 $ there exists a constant ${L}_{r}>0$ such that 
\[
\nnorm{ \tHall (\phi )- \tHall  (\tilde{\phi } )}{ {\Vmath }_{1,2}' }{} 
\; \le \; {L}_{r}  \norm{\phi -\tilde{\phi } ) }{\Vmath }{} , \qquad \phi ,\tilde{\phi } \in \Vmath ,
  \quad \norm{\phi }{\Vmath }{} , \norm{\tilde{\phi }}{\Vmath }{} \le r .
\]   
\item[(iii) ] If $s\ge 0$ and $m> \frac{5}{2}$, then $\tHall $ can be extended to the bilinear mapping from $\Hmath \times \Vmath $ to $ {\Vmath }_{s,m}' $ (denoted still by $\tHall $) such that
\[
\nnorm{ \tHall (\phi , \psi )}{{\Vmath }_{s,m}'}{} 
\; \le \; {c}_{\tHall }(s,m) \nnorm{\phi }{\Hmath }{} \nnorm{\psi }{\Vmath }{} , \qquad \phi   \in \Hmath  , \quad \psi \in \Vmath ,
\]
where ${c}_{\tHall }(s,m)$ is a positive constant.

\medskip \noindent
In particular, if $m> \frac{5}{2} $, then  $\tHall $ can be extended to the bilinear mapping from $\Hmath \times \Vmath $ to ${\Vmath }_{m}'$ (denoted still by $\tHall $) such that
\begin{equation}
\nnorm{ \tHall (\phi , \psi )}{{\Vmath }_{m}'}{} 
\; \le \; {c}_{\tHall }(m) \nnorm{\phi }{\Hmath }{} \norm{\psi }{\Vmath }{} , \qquad \phi   \in \Hmath  ,
\quad  \psi \in \Vmath ,
\label{eq:tHall-map_est-H-V}
\end{equation}
where ${c}_{\tHall }(m):={c}_{\tHall }(m,m)$ and ${\Vmath }_{m}$ is the space defined by \eqref{eq:Vmath_m}.
\end{description}
\end{lemma}

\medskip
\noindent
We will use also the following convergence result for the map $\tHall $ proved in Lemma 2.5 and Corollary 2.10 in \cite{EM'21_ArX_Hall-MHD_Gauss}.

\medskip
\begin{cor} \label{cor:tHall-term_conv_general} (See \cite[Corollary 2.10]{EM'21_ArX_Hall-MHD_Gauss}.)
Let $\phi \in {L}^{2}(0,T;\Hmath )$ and $\psi  \in {L}^{2}(0,T;\Vmath )$  and let
$({\phi }_{n} ) \subset {L}^{2}(0,T;\Hmath )$
and $({\psi }_{n} ) \subset {L}^{2}(0,T;\Vmath )$ be two sequences  such that
\begin{itemize}
\item $({\phi }_{n})$ is bounded in ${L}^{2}(0,T;\Hmath )$ and ${\psi }_{n} \to \psi $ weakly in ${L}^{2}(0,T;\Vmath )$,
\item  ${\phi }_{n} \to \phi $ and ${\psi }_{n} \to \psi  $  in ${L}^{2}(0,T; {\Hmath }_{loc}) $.
\end{itemize}
If $s\ge 0 $ and $m > \frac{5}{2}$, then for all $t \in [0,T] $ and all
 $ \varphi  \in {\Vmath }_{s,m}  $:
\begin{equation*}
\lim_{n\to \infty } \int_{0}^{t}\dual{\tHall ({\phi }_{n} (s),{\psi }_{n} (s))}{ \varphi  }{} \, ds  
\; = \; \int_{0}^{t}\dual{\tHall (\phi (s),\psi (s))}{\varphi  }{} \, ds ,
\end{equation*}
where ${\Vmath }_{s,m}  $ is the the space defined by \eqref{eq:Vmath_m1,m2}-\eqref{eq:Vmath_m1,m2-norm}.
\end{cor}

\medskip
\subsection{Probabilistic preliminaries. Time homogeneous Poisson random measure}  \label{sec:Poisson_random_measure}

\medskip
\noindent
We follow the approach due to \cite{Applebaum'2009}, \cite{Brze+Hausenblas'2009}, 
\cite{Brze+Haus+Raza'18}, see also \cite{Ikeda+Watanabe'81}  and \cite{Peszat+Zabczyk'2007}. 
Let us denote
 $\mathbb{N} :=\{ 0,1,2,... \} , \, \,  \overline{\mathbb{N} }:= \mathbb{N} \cup \{ \infty  \} , \, \, 
 {\mathbb{R} }_{+}:=[0,\infty ) $.
Let  $(S, \mathcal{S} )$ be a measurable space and let 
 ${M}_{\overline{\mathbb{N} }}(S) $ be the set of all $\overline{\mathbb{N} }$ valued measures on $(S, \mathcal{S} )$.
On the set  ${M}_{\overline{\mathbb{N} }}(S)$ we consider the $\sigma $-field
 $  {\mathcal{M} }_{\overline{\mathbb{N} }}(S) $ defined as
the smallest $\sigma $-field  such that for all $ B \in \mathcal{S} $: the map
$
{i}_{B} : {M}_{\overline{\mathbb{N} }}(S) \ni \mu \mapsto \mu (B) \in \overline{\mathbb{N} }
$ 
is measurable.

\medskip  \noindent
Let $(\Omega , \mathcal{F} ,\p  )$ be a complete probability space with filtration $\mathbb{F}:=({\mathcal{F} }_{t}{)}_{t\ge 0}$ satisfying the usual conditions.

\medskip 
\begin{definition} \rm (See \cite{Applebaum'2009} and \cite{Brze+Hausenblas'2009}). Let $(Y, \ycal )$ be a measurable space. A
\it  time homogeneous Poisson random measure $\eta $  \rm on $(Y, \ycal )$ over $(\! \Omega ,\mathcal{F} , \Fmath ,\p )$ is a measurable function 
\[
\eta : (\Omega , \mathcal{F} ) \to \bigl( {M}_{\overline{\mathbb{N} }} ({\mathbb{R} }_{+}\times Y),{\mathcal{M} }_{\overline{\mathbb{N} }} ({\mathbb{R} }_{+}\times Y) \bigr) 
\]
such that 
\begin{itemize}
\item[(i) ] for all $ B \in \mathcal{B} ({\mathbb{R} }_{+}) \otimes \ycal $, $\eta (B):= {i}_{B} \circ \eta : \Omega \to \overline{\mathbb{N} }$ is a Poisson random variable with parameter $\e [\eta (B)]$;
\item[(ii) ] $\eta $ is independently scattered, i.e. if the sets ${B}_{j}\in \mathcal{B} ({\mathbb{R} }_{+}) \otimes \ycal $, $j=1,...,n$, are disjoint then the random variables $\eta ({B}_{j})$, $j=1,...,n$, are independent;
\item[(iii) ]  for all $ U \in \ycal $  the $\overline{\mathbb{N} }$-valued process 
$\bigl( N(t,U) {\bigr) }_{t \ge 0} $ defined by 
\[
N(t,U) \; := \; \eta ((0,t]\times U) , \qquad  t \ge 0 
\]
is $\Fmath $-adapted and its increments are independent of the past, i.e. if $t>s\ge 0$, then 
$N(t,U)-N(s,U)= \eta ((s,t]\times U)$ is independent of ${\mathcal{F} }_{s}$.
\end{itemize} 
\end{definition}
\noindent
If $\eta $ is a time homogeneous Poisson random measure then the formula
\[
\mu (A) \; := \; \e [\eta ((0,1]\times A )] , \qquad A \in \ycal 
\]
defines a measure on $(Y, \ycal )$ called an \it intensity measure \rm of $\eta $.
Moreover, for all $T<\infty $ and all $A\in \ycal $ such that $\e \bigl[ \eta ((0,T]\times A) \bigr] <\infty $, the 
$\mathbb{R} $-valued process $\{ \tilde{N} (t,A) {\} \! }_{t \in (0,T]\! }$ defined by 
\[
\tilde{N} (t,A) \; := \; \eta ((0,t]\times A)  - t \mu (A) , \qquad t \in (0,T],
\] 
is an integrable martingale on $(\Omega ,\mathcal{F} , \Fmath ,\p )$.
The random measure $\ell \otimes \mu $ on $\mathcal{B} ({\mathbb{R} }_{+}) $ $\otimes $ $\ycal $, where $\mathcal{B} ({\mathbb{R} }_{+}) $ denotes the Borel $\sigma $-field and $\ell $ stands for the Lebesgue measure, is called a \it compensator \rm of $\eta $, and 
the difference between a time homogeneous Poisson random measure $\eta $ and its compensator, i.e. 
\[
\tilde{\eta } \; := \; \eta - \ell \otimes \mu   ,
\]
is called a \it compensated time homogeneous Poisson random measure.\rm 

\medskip  \noindent
Let us also recall basic properties of the stochastic integral with respect to  $\tilde{\eta }$,
see \cite{Brze+Hausenblas'2009}, \cite{Ikeda+Watanabe'81} and \cite{Peszat+Zabczyk'2007} for details.
Let $E $ be a separable Hilbert space and let $\pcal $ be a predictable $\sigma $-field on $[0,T] \times \Omega $.
Let ${\mathfrak{L}}^{2}_{\mu ,T} (\pcal \otimes \ycal ,l \otimes \p \otimes \mu ;E)$ be a space of all $E$-valued, $\pcal \otimes \ycal $-measurable processes such that
\[
\e \Bigl[ \int_{0}^{T}\int_{Y} \norm{\xi (s, \cdot ,y)}{E}{2} \,  ds d\mu (y) \Bigr] \; < \; \infty .
\] 
If $\xi \in {\mathfrak{L}}^{2}_{\mu ,T} (\pcal \otimes \ycal ,l \otimes \p \otimes \mu ;E )$
then the integral process $\int_{0}^{t} \int_{Y} \xi (s, \cdot ,y) \, \tilde{\eta }(ds,dy)$, 
$t\in [0,T]$, is a \it c\`{a}dl\`{a}g \rm ${L}^{2}$-integrable martingale. Moreover, the following isometry formula holds
\begin{equation} \label{eq:isometry}
\e \biggl[ \Norm{\int_{0}^{t} \int_{Y} \xi (s, \cdot ,y)  \tilde{\eta }(ds,dy) }{E}{2} \biggr]
\; = \; \e \Bigl[ \int_{0}^{t}\int_{Y} \norm{\xi (s, \cdot ,y)}{E}{2}   ds d\mu (y) \Bigr] ,
   \, \,  t \in [0,T].
\end{equation}

\medskip
\section{Martingale solutions of the Hall-MHD equations.} \label{sec:statement}

\medskip  \noindent
Let  $H$ and ${L}^{2}(0,T;{H}_{loc})$ be the spaces defined by \eqref{eq:H} and \eqref{eq:seminorms-L^2(0,T;H_loc)}, respectively.
Let $(\Omega , \mathcal{F} ,\p  )$ be a complete probability space with filtration 
$\mathbb{F}:=\{ {\mathcal{F} }_{t}{\} }_{t\ge 0}$ satisfying the usual conditions and let us denote
$\mathfrak{A}:=(\Omega , \mathcal{F} , \Fmath ,\p  )$.

\medskip  \noindent
\begin{assumption}  \label{ass:data+noise_Poisson} \
\begin{description}
\item[(F.1)] Let $(Y_i, {\ycal }_{i})$, $i=1,2$, be measurable spaces and let  ${\mu }_{i}$, $i=1,2 $, be  $\sigma $-finite measures on $(Y_i, {\ycal }_{i})$. 
Assume that ${\eta }_{i}$, $i=1,2$, are  time homogeneous Poisson random measures on  $({Y}_{i}, {\ycal }_{i})$  over
$\mathfrak{A}$ with the (jump) intensity measures ${\mu }_{i}$.
\item[(F.2)] Assume that $F_i:[0,T]\times H \times {Y}_{i} \to H $, where $i=1,2$, are two measurable maps such that 
 there exist  constants $L_i$, $i=1,2$ such that
\begin{equation}
\int_{{Y}_{i}} \nnorm{F_i(t,\phi ;y)- F_i(t,\psi ;y) }{H}{2}  \, {\mu }_{i}(dy) 
\; \le \;  L_i \nnorm{\phi -\psi }{H}{2} 
   , \quad \phi , \psi \in H  , \, \,  t \in [0,T] \label{eq:Fi_Lipschitz_cond} ,
\end{equation}
and for each $q\ge 1 $  there exists a constant ${K}_{iq}$ such that
\begin{equation}   
\int_{{Y}_{i}} \nnorm{F_i(t,\phi ;y)}{H}{q} \, {\mu }_{i} (dy) 
\; \le \;  {K}_{iq} (1 + \nnorm{\phi }{H}{q}), \qquad  \phi  \in H  , \quad t \in [0,T].  \label{eq:Fi_linear_growth}
\end{equation}
\item[(F.3)] Moreover, for all $\varphi  \in H $ the maps ${ \tilde{F} }_{i,\varphi }$, $i=1,2 $, defined
for $\phi : [0,T] \to \Hmath $ by 
\begin{equation} \label{eq:Fi**}
\bigl( {\tilde{F}}_{i,\varphi }(\phi )\bigr) (t,y)
\; := \; \ilsk{F_i(t,\phi ({t}^{-});y)}{\varphi }{H}, 
    \quad (t,y) \in [0,T] \times {Y}_{i} 
\end{equation}
are continuous from ${L}^{2}(0,T;{H}_{loc}) $ into $ {L}^{2}([0,T]\times {Y}_{i}, d\ell \otimes {\mu }_{i}; \mathbb{R} ) $.

\medskip  \noindent
Here $\ell $ denotes the Lebesgue measure on the interval $[0,T]$. 
\end{description}
\end{assumption}

\medskip  \noindent
In the context of Assumption \ref{ass:data+noise_Poisson} let 
\begin{equation*}
Y:= {Y}_{1}\times {Y}_{2}, \qquad \ycal := {\ycal }_{1} \otimes {\ycal }_{2}, \qquad \mu := {\mu }_{1} \otimes {\mu }_{2}, \qquad \eta := ({\eta }_{1},{\eta }_{2}).
\end{equation*}

\medskip  
\begin{remark} \label{rem:F_properties}
\bf (The map $F$ and its properties.)  \rm 
Let ${F}_{1}$ and ${F}_{2}$ be the maps given in Assumption \ref{ass:data+noise_Poisson}. 
Let us define the following map
\begin{equation}
F(t,\Phi ;y) \; := \;  (F_1(t,\ubold ;y_1),F_2(t,\Bbold ;y_2)) , 
\label{eq:F_map}
\end{equation}
where $t\in [0,T]$, $\Phi = (\ubold ,\Bbold ) \in \Hmath $, $y = (y_1,y_2)\in Y $.
\begin{description} 
\item[(i) ] Then
$
F:[0,T]\times \Hmath \times Y \to \Hmath
$
is a measurable map and there exists a constant $L>0$ such that
\begin{equation}
\int_{Y} \nnorm{F(t,\Phi ;y)- F(t,\Psi ;y) }{\Hmath }{2}  \mu (dy) 
\; \le \;  L \nnorm{\Phi -\Psi }{\Hmath }{2} 
   , \quad \Phi ,  \Psi  \in \Hmath  , \, \,  t \in [0,T] \label{eq:F_Lipschitz_cond} ,
\end{equation}
and for each $q\ge 1 $  there exists a constant ${K}_{q}$ such that
\begin{equation}   
\int_{Y} \nnorm{F(t,\Phi ;y)}{\Hmath }{q} \, \mu (dy) 
\; \le \;  {K}_{q}(1 + \nnorm{\Phi }{\Hmath }{q}), \qquad  \Phi  \in \Hmath  , \quad t \in [0,T].  \label{eq:F_linear_growth}
\end{equation}
\item[(ii) ]
Moreover, for every $\Psi  \in \Hmath $ the map ${\tilde{F}}_{\Psi }$ defined
for $\Phi :[0,T] \to \Hmath $ by 
\begin{equation} \label{eq:F**}
\bigl( {\tilde{F}}_{\Psi }(\Phi )\bigr) (t,y)
\; := \; \ilsk{F(t,\Phi ({t}^{-});y)}{\Psi }{\Hmath }, 
      \quad (t,y) \in [0,T] \times Y 
\end{equation}
is  continuous from ${L}^{2}(0,T;{\Hmath }_{loc}) $ into $ {L}^{2}([0,T]\times Y, d\ell \otimes \mu ; \mathbb{R} ) $.
\end{description}
\end{remark}

\noindent
Let us recall that the spaces $\Hmath $ and ${L}^{2}(0,T;{\Hmath }_{loc})$ are defined in Section \ref{sec:Hall-MHD_funct-setting}.

\medskip  
\noindent
Taking into account the above notations, and the maps  $\mathcal{A} $, $\MHD $, $\tHall $ and $F$ defined respectively  by \eqref{eq:A_acal_rel},  \eqref{eq:MHD-map}, \eqref{eq:tHall_map} and \eqref{eq:F_map},  problem \eqref{eq:Hall-MHD_u_Poisson}-\eqref{eq:Hall-MHD_ini-cond} can be rewritten in the form of the following equation
\begin{equation}
\begin{split} 
&  \X (t)  + \int_{0}^{t} \bigl[  \mathcal{A} \X (s)  + \MHD  (\X (s)) +\tHall (\X (s)) \bigr] \, ds  \\
& \qquad   \; = \; {\X }_{0}
 + \int_{0}^{t}\int_{Y}  F(s,\X ({s}^{-});y ) \, \tilde{\eta }(ds,dy)  , \qquad t \in [0,T] ,   
\end{split}  \label{eq:Hall-MHD_functional}
\end{equation} 
where ${\X }_{0} := ({\ubold }_{0}, {\Bbold }_{0})$.

\medskip
\begin{definition}  \rm  \label{def:mart-sol}
Let Assumption \ref{ass:data+noise_Poisson} be satisfied and let ${\X }_{0} \in \Hmath $. 
We say that there exists a \bf martingale solution \rm of  problem \eqref{eq:Hall-MHD_functional} 
iff there exist
\begin{itemize}
\item[$\bullet $] a stochastic basis $\bar{\mathfrak{A}}:= \bigl( \bar{\Omega }, \bar{\mathcal{F} },  \bar{\Fmath } ,\bar{\p }  \bigr) $ with a  filtration $\bar{\Fmath } = \{ {\bar{\mathcal{F} }_{t}}{\} }_{t \in [0,T]} $ satisfying the usual conditions, 
\item[$\bullet $]  a time homogeneous Poisson random measure $\bar{\eta }$ on $(Y, \ycal )$ over
$\bar{\mathfrak{A}} $ with the intensity measure $\mu $,
\item[$\bullet $] and an $\bar{\Fmath }$- progressively measurable process $\bX : [0,T] \times \Omega \to \Hmath $ with  paths satisfying
\begin{equation}
\bX (\cdot , \omega ) \; \in \; \Dmath ([0,T]; {\Hmath }_{w}) \cap {L}^{2}(0,T;\Vmath ) ,
\label{eq:mart-sol_regularity}
\end{equation}
for $\bar{\p } $-a.e. $\omega \in \bar{\Omega }$,
and such that 
for all $ t \in [0,T] $ and  $\phi \in \Vtest  $ the following identity holds $\bar{\p }$ - a.s.
\begin{equation} 
\begin{split}
&\ilsk{\bX (t)}{\phi}{\Hmath }  + \int_{0}^{t} \dual{\mathcal{A} \bX (s)}{\phi}{}  ds
+ \int_{0}^{t} \dual{\MHD (\bX (s))}{\phi }{}  ds 
+ \int_{0}^{t} \dual{\tHall (\bX (s))}{\phi }{}  ds
    \\
\; & = \; \ilsk{{\X }_{0}}{\phi}{\Hmath } 
 + \int_{0}^{t} \int_{Y}  \ilsk{F(s,\bX ({s}^{-});y ) }{\phi }{\Hmath } \, \tilde{\bar{\eta }}(ds,dy),
\end{split}  \label{eq:mart-sol_int-identity}
\end{equation}
where  ${\Vmath }_{1,2}$ is the space is defined by \eqref{eq:Vmath_m1,m2}.
\end{itemize}
If all the above conditions are satisfied, then the system
$ \bigl( \bar{\mathfrak{A}}, \bar{\eta } ,\bX \bigr)  $
is called a \bf martingale solution \rm of problem \eqref{eq:Hall-MHD_functional}.
\end{definition}

\medskip  \noindent
Here, ${\Hmath }_{w}$ denotes the Hilbert space $\Hmath $ endowed with the weak topology.
The fact that $\Psi  \in  \Dmath ([0,T]; {\Hmath }_{w}) $ means that for every $h\in \Hmath $ 
\[
[0,T] \; \ni \; t \; \mapsto \; \ilsk{\Psi (t)}{h}{\Hmath } \in \; \mathbb{R} \;  
\] 
is a real-valued c\`{a}dl\`{a}g function, i.e. it is right continuous and has left limits at every $t\in [0,T]$, see also Appendix \ref{sec:Cadlag_functions}.   

\medskip  \noindent
The main result of this paper asserts the existence of a martingale solution.

\medskip
\begin{theorem} \label{th:mart-sol_existence}
Let Assumption  \ref{ass:data+noise_Poisson} be satisfied and let $p\in [2,\infty )$.
Then for every ${\X }_{0} \in \Hmath $ there exists a  martingale solution  of  problem \eqref{eq:Hall-MHD_functional} such that for every $q \in [2,p]$
\begin{equation} 
\bar{\e } \Bigl[ \sup_{t \in [0,T]} \nnorm{\bX (t)}{\Hmath }{q} 
 + \int_{0}^{T} \norm{\bX (t)}{\Vmath }{2} \, dt\Bigr] \; < \; \infty .
\label{eq:mart-sol_energy-ineq_exist}
\end{equation}   
\end{theorem}

\medskip
\section{Approximate SPDEs } \label{sec:truncated_eq}

\medskip  \noindent
In this section will construct a sequence of stochastic equations approximating problem \eqref{eq:Hall-MHD_functional} and prove a priori estimates. This construction is based on the Fourier analysis. The deterministic background is the same as in \cite{EM'21_ArX_Hall-MHD_Gauss}. Since in the present paper we consider the Hall-MHD equation with the Poisson type noise terms, analysis of the stochastic part is different from \cite{EM'21_ArX_Hall-MHD_Gauss}. Approximation of this type is closely related to the Littlewood-Paley decomposition, see \cite{Bahouri+Chemin+Danchin'11}, and has also been used, e.g., in  \cite{Feff+McCorm+Rob+Rod'2014}, \cite{Mohan+Sritharan'16}, \cite{Manna+Mohan+Srith'2017} and \cite{Brze+Dha'20}.

\medskip
\subsection{The subspaces $\Hn $ and the  operators $\Pn $}

\medskip  \noindent
First we recall the construction of the sequence of subspaces ${(\Hn )}_{n\in \mathbb{N}}$ and associated sequence of operators ${(\Pn )}_{n \in \mathbb{N} }$.
Let 
\begin{equation*}
{\bar{B}}_{n} \; := \; \{ \xi \in {\mathbb{R} }^{3} : \; \nnorm{\xi }{}{} \le n \}  , \qquad n \in \mathbb{N} 
\end{equation*}
and let
\begin{equation*}
{H}_{n} \; := \; \{ v \in H: \; \; \supp \widehat{v} \in {\bar{B}}_{n}  \} .
\end{equation*}
In the subspace ${H}_{n}$ we consider the norm inherited from the  the space $H$ defined by \eqref{eq:H}.
For each $n\in \mathbb{N} $ let us define a map ${\pi }_{n}$ by
\begin{equation*}
{\pi }_{n} v \; := \; {\mathcal{F} }^{-1} (\ind{{\bar{B}}_{n}} \widehat{v}) , \qquad v \in H ,
\end{equation*}
where ${\mathcal{F} }^{-1}$ denotes  the inverse of the Fourier transform, see Appendix \ref{sec:Fourier_truncation}. 
By Remark \ref{rem:S_n-projection},  the map ${\pi }_{n}:H \to H_n$ is the orthogonal projection onto $H_n$.

\medskip  \noindent
Let
\begin{equation*}
{\bar{\ball }}_{n} \; := \; {\bar{B}}_{n} \times {\bar{B}}_{n}
\end{equation*}
and
\begin{equation}
 \Hn \; := H_n \times H_n. 
\label{eq:H_n} 
\end{equation} 
In the subspace $\Hn $ we consider the norm inherited from the space $\Hmath  $ defined by \eqref{eq:Hall-MHD_H-V}.
Let us define the  operator
\begin{equation}
\Pn \; := \;  {\pi }_{n} \times {\pi }_{n} : \Hmath \; \to \;  \Hn   .
\label{eq:P_n}
\end{equation}
Explicitly, for $\Phi = (\ubold ,\Bbold ) \in \Hmath $
\begin{equation*}
\Pn (\ubold ,\Bbold ) \; = \; ({\pi }_{n}\ubold , {\pi }_{n}\Bbold )
\; = \; ({\mathcal{F} }^{-1} (\ind{{\bar{B}}_{n}} \widehat{\ubold }), {\mathcal{F} }^{-1} (\ind{{\bar{B}}_{n}} \widehat{\Bbold })) .
\end{equation*}
Since the map ${\pi }_{n}:H \to H_n$ is the orthogonal projection onto $H_n$, we infer that 
\begin{equation}
\Pn : \Hmath \to \Hn 
\label{eq:P_n_proj}
\end{equation}
is the orthogonal projection onto $\Hn $. 

\medskip  \noindent
The following lemma states that the subspaces $\Hn $ are embedded in the spaces 
${\Vmath }_{m_1,m_2}$ for $m_1,m_2\ge 0 $, defined by \eqref{eq:Vmath_m1,m2} with the equivalence of appropriate norms.

\medskip
\begin{lemma} \label{lem:H_n-V_m1,m2-relation}  
(See \cite[Lemma 4.1]{EM'21_ArX_Hall-MHD_Gauss}.)
Let $n \in \mathbb{N} $ and  ${m}_{1},m_2 \ge 0 $.
Then
\begin{equation*}
\Hn \; \hookrightarrow  \; {\Vmath }_{m_1,m_2},   
\end{equation*}
and   for all  $  u \in \Hn :$
\begin{equation*}
\norm{u}{{m_1,m_2}}{2} \; \le \;  {(1+{n}^{2})}^{m} \, \nnorm{u}{\Hn }{2} ,  
\end{equation*}
where $m=\max \{ m_1,m_2\} $.
\end{lemma}

\medskip
\begin{cor} \label{cor:H_n-V_m1,m2-norm_equiv} 
(See \cite[Corollary 4.2]{EM'21_ArX_Hall-MHD_Gauss}.)
On the subspace $\Hn $ the norm $\nnorm{\cdot}{\Hn }{}$ and the norms  $\norm{\cdot }{{m_1,m_2}}{}$, for $m_1,m_2\ge 0 $, inherited from the spaces ${\Vmath }_{m_1,m_2}$ are equivalent (with appropriate constants depending on $m_1,m_2$ and $n$).   
\end{cor}

\medskip 
\noindent
Let us recall properties of the operators $\Pn $ in the spaces ${\Vmath }_{m_1,m_2}$ defined by \eqref{eq:Vmath_m1,m2}

\medskip
\begin{lemma} \label{lem:P_n-pointwise_conv}  
(See \cite[Lemma 4.3]{EM'21_ArX_Hall-MHD_Gauss}.)
Let us fix $m_1 ,m_2 \ge 0 $.   
Then for all $n \in \mathbb{N} $:
\[
\Pn  : {\Vmath }_{m_1,m_2} \; \to \; {\Vmath }_{m_1,m_2}
\]
is well defined linear and bounded. Moreover,  for every $ u \in {\Vmath }_{m_1,m_2}:$
\begin{align*}
\lim_{n \to \infty } \norm{\Pn u-u}{{m_1,m_2}}{}  \; = \; 0 . 
\end{align*} 
\end{lemma}

\medskip  \noindent
From Lemma \ref{lem:P_n-pointwise_conv} we obtain the following corollary.

\medskip
\begin{cor} \label{cor:P_n-pointwise_conv} 
(See \cite[Corollary 4.4]{EM'21_ArX_Hall-MHD_Gauss}.)
\begin{description}
\item[(i) ]  $\Pn \in \mathcal{L} (\Vmath , \Vmath )$, and for all $u \in \Vmath $
\[
\lim_{n\to \infty } \norm{ \Pn u -u }{\Vmath }{} \; = \; 0,
\]
\item[(ii) ] For every $m \ge 0 $,  $\Pn \in \mathcal{L} ({\Vmath }_{m} , {\Vmath }_{m})$   and for all $u \in {\Vmath }_{m} $
\[
\lim_{n\to \infty } \norm{ \Pn u -u }{{\Vmath }_{m} }{} \; = \; 0,
\]
\item[(iii) ] For every $m \ge 1$, $\Pn \in \mathcal{L} ( {\Vmath }_{m} , \Vmath )$  and for all $u \in {\Vmath }_{m} $
\[
\lim_{n\to \infty } \norm{ \Pn u -u }{\Vmath }{} \; = \; 0.
\]
\end{description}
The spaces $\Vmath $ and  ${\Vmath }_{m}$ is defined by \eqref{eq:Hall-MHD_H-V} and \eqref{eq:Vmath_m}, respectively.
\end{cor}

\medskip
\subsection{The truncated SPDEs}

\medskip  \noindent
Let us consider the following approximation in the space $\Hn $ defined by \eqref{eq:H_n}. 

\medskip
\begin{definition} \label{def:trunctated_weak} \rm
Let ${\X }_{0} \in \Hmath $.
By a \bf truncated approximation \rm  of equation 
\eqref{eq:Hall-MHD_functional}  we mean an $\Hn $-valued  c\`{a}dl\`{a}g, $\Fmath $-adapted  process  ${\{ \Xn (t) \} }_{t \in [0,T]}$  such that for all $t\in [0,T]$ and $\phi \in \Hn $ the following identity holds $\p $ - a.s. 
\begin{equation}
\begin{split}
&\ilsk{\Xn (t)}{\phi}{\Hmath }
+ \int_{0}^{t} \dual{ \mathcal{A}  \Xn (s)}{ \phi }{} \, ds
+ \int_{0}^{t} \dual{ \MHD (\Xn (s))}{\phi}{}  ds
+ \int_{0}^{t} \dual{ \tHall (\Xn (s))}{ \phi }{} \, ds
   \\
 &= \,\, \ilsk{{\X }_{0}}{\phi}{\Hmath }
+ \int_{0}^{t} \int_{Y}  \ilsk{F(s,\Xn ({s}^{-});y ) }{\phi }{\Hmath } \, \tilde{\eta }(ds,dy).
\end{split}   \label{eq:Hall-MHD_truncated_weak}
\end{equation}
\end{definition}

\medskip
\noindent
Using Lemma \ref{lem:H_n-V_m1,m2-relation}, Corollary \ref{cor:H_n-V_m1,m2-norm_equiv} and
 the Riesz representation theorem for continuous linear functionals on $\Hn$, 
we will rewrite identity \eqref{eq:Hall-MHD_truncated_weak} as a stochastic equation in the space $\Hn $ 
(with the $\ilsk{\cdot }{\cdot }{\Hmath }$-inner product).

\medskip  
\begin{remark} \label{rem:truncated_funct_Riesz}  \rm 
(See \cite[Remark 4.7]{EM'21_ArX_Hall-MHD_Gauss}.)
 For fixed $v \in \Vmath $, the maps
\begin{align}
&\Hn \ni \varphi \,\, \mapsto \,\, \ddual{\Vprime }{ \mathcal{A} v}{ \varphi }{\Vmath } \in \mathbb{R} ,  
\label{eq:Acal_Hn}
\\
&\Hn \ni \varphi \; \mapsto \; \ddual{\Vprime }{ \MHD (v)}{ \varphi }{\Vmath } \in \mathbb{R} ,  
\label{eq:Bn_Hn}
\\
&\Hn \ni \varphi \; \mapsto \; \ddual{\Vprime }{ \tHall (v)}{ \varphi }{\Vmath } \in \mathbb{R} .  
\label{eq:R_Hn}
\end{align}
are continuous linear functionals on $\Hn $.
Let $\ARiesz (v), \BRiesz (v), \RRiesz (v)  \in \Hn$ denote the Riesz representations in $\Hn$ of 
the functionals  \eqref{eq:Acal_Hn}, \eqref{eq:Bn_Hn_Riesz}, \eqref{eq:R_Hn}, respectively. Then we have for every $\varphi \in \Hn $
\begin{align}
\ddual{\Vprime }{ \mathcal{A} v}{ \varphi }{\Vmath }  
\; &= \; \ilsk{\ARiesz(v)}{\varphi }{\Hmath } , 
\label{eq:Acal_Hn_Riesz}
\\
\ddual{\Vprime }{ \MHD (v)}{ \varphi }{\Vmath }  
\; &= \; \ilsk{\BRiesz(v)}{\varphi }{\Hmath } , 
\label{eq:Bn_Hn_Riesz}
\\
\ddual{{\Vmath }_{1,2}' }{ \tHall (v)}{ \varphi }{{\Vmath }_{1,2}}  
\; &= \; \ilsk{\RRiesz(v)}{\varphi }{\Hmath }  .
\label{eq:R_Hn_Riesz}
\end{align}
Since $\mathcal{A} $ is linear, the map $\Vmath \ni v \mapsto \ARiesz(v) \in \Hn  $ is linear as well.

\medskip  \noindent
Since $\Pn : \Hmath \to \Hn $ is the $\ilsk{\cdot }{\cdot }{\Hmath }$-orthogonal projection,  we have
\begin{equation*}
\ilsk{v}{\Pn \varphi }{\Hmath } \; = \; \ilsk{\Pn v}{\varphi }{\Hmath } \quad \mbox{ for all } \quad  \varphi \in \Hn . 
\label{eq:Pn_Hn_Riesz}
\end{equation*}
In particular, for a fixed $v \in \Hmath $ the Riesz representation of the linear functional
$
\Hn \ni \varphi  \,\, \mapsto \,\, \ilsk{v}{\varphi }{\Hmath } \in \mathbb{R} 
$
is equal $\Pn v$.
\end{remark}

\medskip  \noindent
Using Remark \ref{rem:truncated_funct_Riesz} we can write the integral identity \eqref{eq:Hall-MHD_truncated_weak} (tested by the functions from the subspace $\Hn $) in terms of appropriate Riesz representations.
In fact, 
integral identity \eqref{eq:Hall-MHD_truncated_weak}
is equivalent to the following stochastic equation in $\Hn$
\begin{equation}
\begin{split}
&  \Xn (t)  + \int_{0}^{t}\bigl[ \ARiesz( \Xn (s)) + \BRiesz  (\Xn (s)) +\RRiesz ( \Xn (s)) \bigr] \, ds \\
\,\, &= \,\,  \Pn {\X }_{0}  + \int_{0}^{t} \int_{Y}  \Fn (s,\Xn ({s}^{-});y )  \, \tilde{\eta }(ds,dy) ,
\quad  t \in [0,T]  ,   
\end{split} 
\label{eq:Hall-MHD_truncated}
\end{equation}
where $\Fn $ is a map defined by
\begin{equation}
\Fn :[0,T]\times \Hmath \times Y \ni (s,X,y)  \to \Pn F(t,X;y ) \in \Hn \subset \Hmath .
\label{eq:F_n}
\end{equation}

\medskip
\begin{prop} \label{prop:Hall-MHD_truncated_existence}
Let Assumption \ref{ass:data+noise_Poisson}  be satisfied and let ${\X }_{0} \in \Hmath $.
For each $n \in \mathbb{N} $, there exists a unique global solution ${(\Xn (t))}_{t\in [0,T]}$ of equation \eqref{eq:Hall-MHD_truncated} with $\Hn $-valued c\`{a}dl\`{a}g trajectories. 
\end{prop}

\medskip
\begin{proof} 
By Lemmas \ref{lem:MHD-term_properties} and \ref{lem:tHall-term_properties}, for every $n\in \mathbb{N} $ the nonlinear terms  $\BRiesz $   and ${\tHall }_{n} $ are locally Lipschitz (with the Lipschitz constants dependent on $n$).
Moreover, by \eqref{eq:Bn_Hn_Riesz}, \eqref{eq:MHD-map_perp}, \eqref{eq:R_Hn_Riesz} and \eqref{eq:tHall-map_perp} for all $\Xn \in \Hn  $
\[
\dual{\BRiesz (\Xn ) + {\tHall }_{n} (\Xn )}{\Xn }{} \; = \; 0 .
\]
Now, the assertion follows from \cite[Theorem 3.1]{Albeverio+Brzezniak+Wu'2010}.
\end{proof}

\medskip
\subsection{A priori estimates} \label{sec:a_priori_est-truncated}

\medskip  \noindent
In this section we will prove some uniform estimates for the solutions $\{ \Xn , n \in \mathbb{N} \}$ of the truncated equation \eqref{eq:Hall-MHD_truncated}. 
These estimates will be used to prove the tightness of the family of laws of $\Xn $, $n \in \mathbb{N} $, on the functional space $\zcal $ defined by   \eqref{eq:Z_cadlag}. In fact, in the proof of tightness the estimates for $q=2$ from the following Lemma \ref{lem:Hall-MHD_truncated_estimates} are sufficient.
Higher order estimates will be used in the proof of the convergence in Section \ref{sec:main_th-proof}.

\medskip
\begin{lemma} \label{lem:Hall-MHD_truncated_estimates}
Let Assumption \ref{ass:data+noise_Poisson}  be satisfied,  
let ${\X }_{0} \in \Hmath $ and let ${(\Xn )}_{n\in \mathbb{N} }$ be the solutions of equations \eqref{eq:Hall-MHD_truncated}.
Then for every $q\ge 2 $ there exist positive constants ${C}_{1}(q)$ and ${C}_{2}(q)$
such that
\begin{equation} \label{eq:H_estimate_truncated_p}
\sup_{n \in \mathbb{N} }\mathbb{E} \Bigl[ \sup_{s\in [0,T] } \nnorm{\Xn (s)}{{\Hmath }}{q} \Bigr] 
\; \le \; {C}_{1}(q)
\end{equation}
and
\begin{equation} \label{eq:HV_estimate_truncated}
\sup_{n \in \mathbb{N} } \mathbb{E} \Bigl[ \int_{0}^{T} \nnorm{\Xn (s)}{{\Hmath }}{q-2} \norm{  \Xn (s)}{}{2} \, ds \Bigr] \; \le \;  {C}_{2}(q)  .
\end{equation}
In particular, there exists a positive constant ${C}_{2}$ such that 
\begin{equation} \label{eq:V_estimate_truncated}
\sup_{n \in \mathbb{N} } \mathbb{E} \Bigl[ \int_{0}^{T} \norm{\Xn (s)}{\Vmath }{2} \, ds \Bigr]
\; \le \;  {C}_{2}.
\end{equation}
\end{lemma}

\medskip \noindent
In the proof of Lemma \ref{lem:Hall-MHD_truncated_estimates} we will use the It\^{o} formula and the following version of the Gronwall lemma.

\medskip
\begin{lemma} \label{lem:A.1_Chueshov+Millet'10}
(See Lemma A.1 in \cite{Chueshov+Millet'10}.)
Let $X,Y,I$ and $\varphi $ be non-negative processes and $Z$ be a non-negative integrable random variable. Assume that $I$ is a non-decreasing and there exist non-negative constants $C,\alpha ,\beta ,\gamma ,\delta $ with the following properties
\begin{equation}
\int_{0}^{T} \varphi (s) \, ds \; \le \; C, \quad a.s., \qquad 2\beta {e}^{C} \; \le \; 1 ,
\qquad 2 \delta {e}^{C} \; \le \; \alpha ,
\label{eq:A.3_Chueshov+Millet'10}
\end{equation}
and such that for $0\le t \le T $,
\[
\begin{split}
& X(t) + \alpha Y(t) \; \le \; Z + \int_{0}^{t} \varphi (r)X(r) \, dr + I(t), \qquad a.s. , 
\\
& \e [I(t)] \; \le \; \beta \, \e [X(t)] + \gamma \, \int_{0}^{t} \e [X(s)] \, ds + \delta \e [Y(t)] + \tilde{C}, 
\end{split}
\]
where $\tilde{C}>0 $ is a constant. If $X \in {L}^{\infty } ([0,T]\times \Omega )$, then we have
\begin{equation} 
\e [X(t) + \alpha Y(t)] \; \le \; 2 \, \exp \bigl( C+ 2t \gamma {e}^{C} \bigr)
   \cdot \bigl( \e [Z] + \tilde{C} \bigr) , \qquad t \in [0,T].
\label{eq:A.4_Chueshov+Millet'10}
\end{equation}
\end{lemma}

\medskip
\begin{proof}[Proof of Lemma \ref{lem:Hall-MHD_truncated_estimates}]
We apply the It\^{o} formula 
to the function $f$ defined by
\[ 
f(x) \; := \;  \nnorm{x}{{\Hmath }}{q}, \qquad x \in \Hmath .
\]
In the sequel we will often omit the subscript ${\Hmath }$ and write $|\cdot |:= \nnorm{\cdot }{{\Hmath }}{}$. The Fr\'{e}chet derivative of $f$ is given by
\begin{align*}
{f}^{\prime}(x) (h) \; & = \; d_x f (h) \; = \; q \cdot \nnorm{{x}}{}{q-2} \cdot \dual{x}{h}{\Hmath } , \quad h \in \Hmath . 
\end{align*}
By the It\^{o} formula, see \cite{Gyongy+Wu'21}, we have for every $t\in [0,T]$:
\begin{equation*}
\begin{split}
&\nnorm{{\Xn } (t )}{}{q} -  \nnorm{\Pn {\X }_{0}}{}{q}   
\; = \; \int_{0}^{t} q \, \nnorm{\Xn (s)}{}{q-2} \dual{\Xn  (s)}{-\ARiesz \Xn  (s)
- {\MHD }_{n} (\Xn (s)) -{\tHall }_{n}( \Xn (s)) }{} \, ds \\
&+ \int_{0}^{t} \int_{Y}  \Big\{ 
\nnorm{ \Xn ({s}^{-})  +  F_n(s,\Xn ({s}^{-});y)}{\Hmath }{q} 
-\nnorm{\Xn ({s}^{-})}{\Hmath }{q}   \\
&\qquad \qquad \quad  -q \, \nnorm{\Xn ({s}^{-})}{\Hmath }{q-2} 
 \ilsk{\Xn ({s}^{-})}{ F_n(s,\Xn ({s}^{-});y)}{\Hmath } \Bigr\} \, \eta (ds, dy)   \\
& +\int_{0}^{t} \int_{Y} 
q \, \nnorm{\Xn ({s}^{-})}{\Hmath }{q-2} \ilsk{\Xn ({s}^{-})}{F_n(s,\Xn ({s}^{-});y)}{\Hmath} \, \tilde{\eta}(ds,dy)  .
\end{split}
\end{equation*}
Note that by \eqref{eq:Acal_Hn_Riesz} and \eqref{eq:A_acal_rel} we have  
\[
\ilsk{\ARiesz \Xn}{\Xn }{} \; = \;  \dual{\mathcal{A} \Xn }{\Xn }{} \; = \; \norm{\Xn }{}{2}.
\]
By  \eqref{eq:Bn_Hn_Riesz}, \eqref{eq:MHD-map_perp}, \eqref{eq:R_Hn_Riesz}  and \eqref{eq:tHall-map_perp} we infer that
\[
\ilsk{{\MHD }_{n}(\Xn )}{\Xn }{} = \dual{\MHD (\Xn )}{\Xn }{} =0
\quad \mbox{ and } \quad 
\ilsk{{\tHall }_{n}(\Xn )}{\Xn }{} = \dual{\tHall (\Xn )}{\Xn }{} =0.
\]
Thus 
for every $t\in [0,T]$
\begin{equation}
\begin{split}
& \nnorm{{\Xn } (t)}{}{q } +q\, \int_{0}^{t} \nnorm{\Xn (s)}{}{q-2} \norm{\Xn (s)}{}{2} \, ds \\ 
&= \; \nnorm{\Pn {\X }_{0}}{}{q } 
+ \int_{0}^{t} \int_{Y}  \Big\{ 
\nnorm{ \Xn ({s}^{})  +  F_n(s,\Xn ({s}^{});y)}{\Hmath }{q} 
-\nnorm{\Xn ({s}^{})}{\Hmath }{q}   \\
&\qquad \qquad \qquad \qquad \qquad  -q \, \nnorm{\Xn ({s}^{})}{\Hmath }{q-2} 
 \ilsk{\Xn ({s}^{})}{F_n(s,\Xn ({s}^{});y)}{\Hmath } \Bigr\} \, \eta (ds, dy)   \\
& \qquad  +\int_{0}^{t} \int_{Y} 
q \, \nnorm{\Xn ({s}^{-})}{\Hmath }{q-2} \ilsk{\Xn ({s}^{-})}{F_n(s,\Xn ({s}^{-});y)}{\Hmath} \, \tilde{\eta}(ds,dy) 
.
\end{split} \label{ineq-01_truncated}
\end{equation} 

\medskip  \noindent
For any $R>0 $ let us define the stopping time
\begin{equation}
\taunR \; := \;  \inf \{ t \in [0,T]: \nnorm{\Xn (t)}{\Hmath }{} > R \} \wedge T.
\end{equation}
Since $\{ \Xn (t) , \, t \in [0,T]\} $ is an $\Hmath $-valued $\Fmath $-adapted and right-continuous process, $\taunR $ is a stopping time.

\medskip
\noindent
Let us fix $t\in [0,T]$. 
From \eqref{ineq-01_truncated}, we infer that 
\begin{equation}
\begin{split}
&\sup_{s\in [0,t\wedge \taunR ]} \nnorm{{\Xn } (s)}{}{q } 
+ \sup_{s\in [0,t\wedge \taunR ]} q\, \int_{0}^{s} \nnorm{\Xn (r)}{}{q-2} \norm{\Xn (r)}{}{2} \, dr 
\\ 
\;  &\le  \; \nnorm{ {\X }_{0}}{}{q } 
+ \sup_{s\in [0,t\wedge \taunR ]} \biggl| \int_{0}^{s} \int_{Y}  \Big\{ 
\nnorm{ \Xn ({r}^{})  + F_n(r,\Xn ({r}^{});y)}{\Hmath }{q} 
-\nnorm{\Xn ({r}^{})}{\Hmath }{q}   \\
&\qquad \qquad \qquad \qquad \qquad  -q \, \nnorm{\Xn ({r}^{})}{\Hmath }{q-2} 
 \ilsk{\Xn ({r}^{})}{F_n(r,\Xn ({r}^{});y)}{\Hmath } \Bigr\} \, \eta (dr, dy)   \\
& \qquad  +\int_{0}^{s} \int_{Y} 
q \, \nnorm{\Xn ({r}^{-})}{\Hmath }{q-2} \ilsk{\Xn ({r}^{-})}{\Pn F(r,\Xn ({r}^{-});y)}{\Hmath} \, \tilde{\eta}(dr,dy) \biggr|
.
\end{split} 
\label{eq:Millet_aux}
\end{equation} 

\medskip  
\noindent
Let us denote
\begin{equation}
\begin{split}
{\mathcal{U} }_{n,R}(t) \; &:= \; \sup_{s \in [0,t\wedge \taunR ]} \nnorm{\Xn (s)}{\Hmath }{q}
\\
{\vcal }_{n,R}(t) \; &:= \; \sup_{s \in [0,t\wedge \taunR ]} q\, \int_{0}^{s} \nnorm{\Xn (r)}{}{q-2} \norm{\Xn (r)}{}{2} \, dr
\; = \;  q\, \int_{0}^{t\wedge \taunR} \nnorm{\Xn (r)}{}{q-2} \norm{\Xn (r)}{}{2} \, dr 
\\
{\mathcal{I} }_{n,R}(t) \; &:= \; \sup_{s \in [0,t\wedge \taunR ]} \biggl|  
\int_{0}^{s} \int_{Y}  \Big\{ 
\nnorm{ \Xn ({r}^{})  + F_n(r,\Xn ({r}^{});y)}{\Hmath }{q} 
-\nnorm{\Xn ({r}^{})}{\Hmath }{q}   \\
&\qquad \qquad \qquad \qquad \qquad  -q \, \nnorm{\Xn ({r}^{})}{\Hmath }{q-2} 
 \ilsk{\Xn ({r}^{})}{F_n(r,\Xn ({r}^{});y)}{\Hmath } \Bigr\} \, \eta (dr, dy)   \\
& \qquad \qquad \qquad  +\int_{0}^{s} \int_{Y} 
q \, \nnorm{\Xn ({r}^{-})}{\Hmath }{q-2} \ilsk{\Xn ({r}^{-})}{F_n (r,\Xn ({r}^{-});y)}{\Hmath} \, \tilde{\eta}(dr,dy) ,
\biggr| 
\end{split}
\label{eq:app_Gronwall_1}
\end{equation}
where $t\in [0,T]$.

\medskip  \noindent
Using the notations introduced in \eqref{eq:app_Gronwall_1} by \eqref{eq:Millet_aux}
we have for every $t\in [0,T]$
\begin{equation}
{\mathcal{U} }_{n,R}(t) + {\vcal }_{n,R}(t) \; \le \; \nnorm{ {\X }_{0}}{}{q } + {\mathcal{I} }_{n,R}(t).
\end{equation}

\medskip  \noindent
We will estimate the term ${\mathcal{I} }_{n,R}(t)$.
To this end, let 
\begin{equation}
\begin{split}
{\mathcal{I} }_{n}(t) \; :=& \;  
\int_{0}^{t} \int_{Y}  \Big\{ 
\nnorm{ \Xn ({r}^{})  + F_n(r,\Xn ({r}^{});y)}{\Hmath }{q} 
-\nnorm{\Xn ({r}^{})}{\Hmath }{q}   \\
&\qquad \qquad    -q \, \nnorm{\Xn ({r}^{})}{\Hmath }{q-2} 
 \ilsk{\Xn ({r}^{})}{F_n(r,\Xn ({r}^{});y)}{\Hmath } \Bigr\} \, \eta (dr, dy)   \\
& +\int_{0}^{t} \int_{Y} 
q \, \nnorm{\Xn ({r}^{-})}{\Hmath }{q-2} \ilsk{\Xn ({r}^{-})}{F_n(r,\Xn ({r}^{-});y)}{\Hmath} \, \tilde{\eta}(dr,dy) 
\end{split}
\label{eq:I_n}
\end{equation}
We decompose ${\mathcal{I} }_{n}(t) $  into two terms
\begin{equation}
{\mathcal{I} }_{n}(t) \; = \; {\mathcal{J} }_{n}(t) + {\mathcal{M} }_{n}(t), \qquad t \in [0,T], 
\label{eq:I_n_decomp}
\end{equation}
where
\begin{equation}
\begin{split}
{\mathcal{J} }_{n}(t) \; &:= \; \int_{0}^{t} \int_{Y}  \big\{ 
\nnorm{\Xn ({r}^{}) + F_n(r,\Xn ({r}^{});y)}{\Hmath }{q} -\nnorm{\Xn ({r}^{})}{\Hmath }{q}  \\
&  \qquad \qquad \quad - q\, \nnorm{\Xn ({r}^{})}{\Hmath }{q-2} 
\ilsk{\Xn ({r}^{})}{F_n(r,\Xn ({r}^{});y)}{\Hmath } \bigr\} \, \eta (dr, dy) , \quad t\in [0,T] ,  
\label{eq:J_n}
\end{split}
\end{equation}
and 
\begin{equation}
{\mathcal{M} }_{n}(t) \; := \; \int_{0}^{t } \int_{Y} 
q \, \nnorm{\Xn ({r}^{-})}{\Hmath }{q-2} \ilsk{\Xn ({r}^{-})}{F_n(r,\Xn ({r}^{-});y)}{\Hmath} \, \tilde{\eta}(dr,dy)  ,
\quad t \in [0,T].
\label{eq:Mcal_n}
\end{equation}
We will estimate separately the terms ${\mathcal{J} }_{n} $ and ${\mathcal{M} }_{n}$.

\medskip  \noindent
Let us consider first the term ${\mathcal{J} }_{n} $ defined by \eqref{eq:J_n}. 
From the Taylor formula, it follows that for every $q\ge 2 $ there exists a positive constant ${c}_{q}>0$ such that for all $x,h \in \Hmath $ the following inequality holds
\begin{equation}
\bigl| \nnorm{x+h}{\Hmath }{q} -\nnorm{x}{\Hmath }{q}- q \nnorm{x}{\Hmath }{q-2} \ilsk{x}{h}{\Hmath } \bigr| 
\; \le \; {c}_{q} (\nnorm{x}{\Hmath }{q-2} +\nnorm{h}{\Hmath }{q-2} ) \, \nnorm{h}{\Hmath }{2} .  
 \label{eq:app_Taylor_I}
\end{equation}
By \eqref{eq:app_Taylor_I},
the fact that $\Pn :\Hmath \to \Hn$ is the $\ilsk{\cdot }{\cdot }{\Hmath }$-projection, 
and inequality \eqref{eq:F_linear_growth}  we obtain the following inequalities
\begin{equation}
\begin{split}
&\e \Bigl[ \sup_{s\in [0,t\wedge \taunR ]} |{\mathcal{J} }_{n}(s)| \Bigr] 
\\  
\; &\le \; \e \biggl[ \sup_{s\in [0,t\wedge \taunR ]}  \int_{0}^{s} \int_{Y} \Bigl|  
\nnorm{\Xn ({r}^{}) + F_n(r,\Xn ({r}^{});y)}{\Hmath }{q} -\nnorm{\Xn ({r}^{})}{\Hmath }{q}  \\
&  \qquad \qquad \qquad - q\, \nnorm{\Xn ({r}^{})}{\Hmath }{q-2} 
\ilsk{\Xn ({r}^{})}{F_n(r,\Xn ({r}^{});y)}{\Hmath }  \Bigr| \, \eta (dr, dy) 
 \biggr] 
 \\ 
\; &\le \; \e \biggl[  \int_{0}^{t\wedge \taunR} \int_{Y} \Bigl|  
\nnorm{\Xn ({r}^{}) + F_n(r,\Xn ({r}^{});y)}{\Hmath }{q} -\nnorm{\Xn ({r}^{})}{\Hmath }{q}  \\
&  \qquad \qquad \qquad - q\, \nnorm{\Xn ({r}^{})}{\Hmath }{q-2} 
\ilsk{\Xn ({r}^{})}{F_n(r,\Xn ({r}^{});y)}{\Hmath }  \Bigr| \, \eta (dr, dy) 
 \biggr]
  \\ 
\; &= \; \e \biggl[  \int_{0}^{t\wedge \taunR} \int_{Y} \Bigl|  
\nnorm{\Xn ({r}^{}) + F_n(r,\Xn ({r}^{});y)}{\Hmath }{q} -\nnorm{\Xn ({r}^{})}{\Hmath }{q}  \\
&  \qquad \qquad \qquad - q\, \nnorm{\Xn ({r}^{})}{\Hmath }{q-2} 
\ilsk{\Xn ({r}^{})}{F_n(r,\Xn ({r}^{});y)}{\Hmath }  \Bigr| \, \mu (dy) dr 
 \biggr]
\\
\; &\le \; {c}_{q} \e \biggl[ \int_{0}^{t\wedge \taunR } \int_{Y}  
\nnorm{F_n(s,\Xn ({s}^{});y)}{\Hmath }{2}
\Bigl\{ \nnorm{ \Xn ({s}^{})}{\Hmath }{q-2} + \nnorm{F_n(s,\Xn ({s}^{});y)}{\Hmath }{q-2} \Bigr\}   \, \mu (dy) ds   \biggr]
\\
\; &\le \;  {c}_{q} \e \biggl[  \int_{0}^{t\wedge \taunR} \bigl\{ 
{K}_{2} \, \nnorm{\Xn ({s}^{})}{\Hmath }{q-2}  \bigl( 1+ \nnorm{\Xn ({s}^{})}{\Hmath }{2}\bigr) 
+ {K}_{q}\bigl( 1+ \nnorm{\Xn ({s}^{})}{\Hmath }{q}\bigr)   \bigr\} \, ds  \biggr]
\\
\; &\le \; {\tilde{c}}_{q} \e \biggl[  \int_{0}^{t\wedge \taunR} \bigl\{ 1+ \nnorm{\Xn ({s}^{})}{\Hmath }{q} \bigr\} \, ds  \biggr]
\; \le \; {\tilde{c}}_{q}t 
+ {\tilde{c}}_{q} \, \e \biggl[  \int_{0}^{t\wedge \taunR} \nnorm{\Xn ({s}^{})}{\Hmath }{q} \, ds \biggr]  , \qquad t \in [0,T] ,
\end{split}
\label{eq:J_n_est}
\end{equation}
where  ${\tilde{c}}_{q}>0$ is a certain constant.

\medskip  \noindent
Let us move to the term ${\mathcal{M} }_{n}$ given by \eqref{eq:Mcal_n}.
By  \eqref{eq:F_linear_growth}, the fact that $\Pn $ is the orthogonal projection in ${\Hmath }$ 
and \eqref{eq:isometry},
we infer that the process ${({\mathcal{M} }_{n}(t\wedge \taunR) )}_{t\in [0,T]}$
is a square integrable martingale.  Indeed, this is a consequence of the following estimates.
By \eqref{eq:isometry}, \eqref{eq:F_linear_growth} and the fact that $\Pn $ is the orthogonal projection in ${\Hmath }$ we infer that for every $t\in [0,T]$, 
\[
\begin{split}
& \e \Bigl[ \int_{0}^{t\wedge \taunR } 
 \int_{Y} \bigl|  q \, \nnorm{\Xn ({r}^{})}{\Hmath }{q-2} \ilsk{\Xn ({r}^{})}{F_n(r,\Xn ({r}^{});y)}{\Hmath}   {\bigr| }^{2}\, \mu (dy) dr  \Bigr]
 \\
\; &\le \; {q }^{2} \, \e \Bigl[ \int_{0}^{t\wedge \taunR } 
\int_{Y}  \nnorm{\Xn ({r}^{})}{\Hmath }{2(q-2)}
\, \nnorm{\Xn ({r}^{})}{\Hmath }{2} \, \nnorm{F_n(r,\Xn ({r}^{});y)}{\Hmath }{2}
\, \mu (dy) dr  \Bigr]
\\
\; &= \; {q }^{2} \, \e \Bigl[ \int_{0}^{t\wedge \taunR } 
 \nnorm{\Xn ({r}^{})}{\Hmath }{2q-2}
 \, \Bigl\{  \int_{Y} \nnorm{F_n(r,\Xn ({r}^{});y)}{\Hmath }{2}
\, \mu (dy)  \Bigr\}  dr  \Bigr]
\\
\; &\le \; {q }^{2} \, \e \Bigl[ \int_{0}^{t\wedge \taunR } 
 \nnorm{\Xn ({r}^{})}{\Hmath }{2q-2} 
{K}_{2}\bigl( 1+\nnorm{\Xn ({r}^{})}{\Hmath }{2} \bigr)
\, dr  \Bigr]
\; \le \; {q }^{2} {K}_{2} T \, {R}^{2q-2} \bigl( 1+{R}^{2} \bigr)
\; < \; \infty .
\end{split}
\]
By the maximal inequality we infer that there exists a positive constant ${\tilde{K}}_{2}$ such that
\begin{equation*}
\begin{split}
& \e \Bigl[ \sup_{s\in [0,t\wedge \taunR ]} |{\mathcal{M} }_{n}(s)| \Bigr] 
\\
\; &= \; \e \Bigl[ \sup_{s\in [0,t\wedge \taunR ]} \Bigl|
\int_{0}^{s} \int_{Y}  
q \, \nnorm{\Xn ({r}^{-})}{\Hmath }{q-2} \ilsk{\Xn ({r}^{-})}{F_n(r,\Xn ({r}^{-});y)}{\Hmath} \, \tilde{\eta}(dr,dy) 
\Bigr|  \Bigr]
\\
\; &\le  \; {\tilde{K}}_{2} q{K}_{2} \, \e \biggl[ \biggl( \int_{0}^{t\wedge \taunR} 
\nnorm{\Xn ({r}^{})}{\Hmath }{2q-2}  \, (1+ \nnorm{\Xn ({r}^{})}{\Hmath }{2}) 
dr   {\biggr) }^{\frac{1}{2}} \biggr]
\\
\; &\le  \; {\tilde{K}}_{2} q{K}_{2} \, \e \biggl[ 
\biggl( \sup_{r \in [0,\taunR ]} \nnorm{\Xn ({r}^{})}{\Hmath }{q} {\biggr) }^{\frac{1}{2}} 
\biggl( \int_{0}^{t\wedge \taunR} 
\nnorm{\Xn ({r}^{})}{\Hmath }{q-2}  \, (1+ \nnorm{\Xn ({r}^{})}{\Hmath }{2}) 
dr   {\biggr) }^{\frac{1}{2}} \biggr] .
\end{split}
\end{equation*}
Using moreover the Young inequality (for numbers) we infer that for every $\eps > 0$ there exists ${C}_{\eps } > 0 $ such that
\begin{equation}
\begin{split}
& \e \Bigl[ \sup_{s\in [0,t\wedge \taunR ]} |{\mathcal{M} }_{n}(s)| \Bigr] 
\;  \le \; \eps \, \e \Bigl[ \sup_{r \in [0,\taunR ]} \nnorm{\Xn ({r}^{})}{\Hmath }{q} \Bigr]
+ {C}_{\eps } \, \e \biggl[ \int_{0}^{t\wedge \taunR} 
\nnorm{\Xn ({r}^{})}{\Hmath }{q-2}  \, (1+ \nnorm{\Xn ({r}^{})}{\Hmath }{2}) 
dr    \biggr] 
\\
\; & \le \; \eps \, \e \Bigl[ \sup_{r \in [0,\taunR ]} \nnorm{\Xn ({r}^{})}{\Hmath }{q} \Bigr]
+ {C}_{\eps } \, \e \biggl[ \int_{0}^{t\wedge \taunR} 
\Bigl\{ \frac{2}{q} + \Bigl( 1-\frac{1}{q}\Bigr) \nnorm{\Xn ({r}^{})}{\Hmath }{q}) \Bigr\} 
dr    \biggr] 
\\
\; & \le \; \eps \, \e \Bigl[ \sup_{r \in [0,\taunR ]} \nnorm{\Xn ({r}^{})}{\Hmath }{q} \Bigr]
+\frac{2}{q} {C}_{\eps } \,t 
+ 2{C}_{\eps }\Bigl( 1-\frac{1}{q}\Bigr) \, \e \biggl[ \int_{0}^{t\wedge \taunR} 
\nnorm{\Xn ({r}^{})}{\Hmath }{q})  dr    \biggr] .
\end{split}
\label{eq:Mcal_n_est}
\end{equation}

\medskip
\noindent
Now we are in a position to estimate the term ${\mathcal{I} }_{n,R}$ defined in \eqref{eq:app_Gronwall_1}.
By \eqref{eq:I_n_decomp}, \eqref{eq:J_n_est} and \eqref{eq:Mcal_n_est} we obtain for every $t\in [0,T]$
\begin{equation}
\begin{split}
&\e \bigl[ {\mathcal{I} }_{n,R}(t)\bigr]  
\; \le \; \e \Bigl[ \sup_{s\in [0,t\wedge \taunR ]} |{\mathcal{J} }_{n}(s)| \Bigr] 
+ \e \Bigl[ \sup_{s\in [0,t\wedge \taunR ]} |{\mathcal{M} }_{n}(s)| \Bigr] 
\\
\; &\le \; \eps \, \e \Bigl[ \sup_{r \in [0,\taunR ]} \nnorm{\Xn ({r}^{})}{\Hmath }{q} \Bigr]
+{K}_{1}(q,\eps ,T) + {K}_{2}(q,\eps )   \, \e \biggl[ \int_{0}^{t\wedge \taunR} 
\nnorm{\Xn ({r}^{})}{\Hmath }{q})  dr    \biggr] , 
\end{split}
\label{eq:I_n(t)_est}
\end{equation}
where  ${K}_{1}(q,\eps ,T)$ and ${K}_{2}(q,\eps )$ are some positive constants.
Note that $\e [ \int_{0}^{t\wedge \taunR} 
\nnorm{\Xn ({r}^{})}{\Hmath }{q})  dr    ]  
\le \e [ \int_{0}^{t} 
\sup_{\sigma \in [0,r\wedge \taunR ]}\nnorm{\Xn ({\sigma }^{})}{\Hmath }{q})  dr  ]  $.
Using the notation \eqref{eq:app_Gronwall_1},
by \eqref{eq:I_n(t)_est} we have
\begin{equation}
\begin{split}
&\e \bigl[ {\mathcal{I} }_{n,R}(t)\bigr]  
\; \le \; \eps \, \e \bigl[ {\mathcal{U} }_{n,R} (t) \bigr]
+{K}_{1}(q,\eps ,T) + {K}_{2}(q,\eps ) \,  \int_{0}^{t} \e \bigl[{\mathcal{U} }_{n,R} (r) \bigr]  \, dr    
, \quad t \in [0,T] .
\end{split}
\label{eq:app_Gronwall_2}
\end{equation}

\medskip  \noindent
Let us choose $\eps \in (0,\frac{1}{2}]$.
By \eqref{eq:app_Gronwall_1}, \eqref{eq:app_Gronwall_2} and  Lemma \ref{lem:A.1_Chueshov+Millet'10} we obtain
\begin{equation}
\e \bigl[ {\mathcal{U} }_{n,R}(t)+ {\vcal }_{n,R}(t)\bigr]  
\; \le \; 2 {e}^{2t {K}_{2}(\eps ,q)} \, \Bigl( \e [\nnorm{{\X }_{0}}{\Hmath }{q}] + {K}_{1}(\eps ,q,T) \Bigr) , \qquad t \in [0,T].
\end{equation}
Thus there exists a constant $C(\eps ,q,T ) >0 $ such that for every $t\in [0,T]$
\begin{equation}
\e \bigl[ {\mathcal{U} }_{n,R}(t)+ {\vcal }_{n,R}(t)\bigr]  
\; \le \; C(\eps ,q,T ) \,  \Bigl( \e [\nnorm{{\X }_{0}}{\Hmath }{q}] + 1 \Bigr) ,
\end{equation}
or explicitly
\begin{equation}
\e \biggl[ \sup_{s\in [0,t\wedge \taunR ]} \nnorm{\Xn (s)}{\Hmath }{q}
+ q \, \int_{0}^{t\wedge \taunR}\nnorm{\Xn (s)}{\Hmath }{q-2} \norm{\Xn (s)}{}{2} \, ds  \biggr]  
\; \le \;  C(\eps ,q,T ) \,  \Bigl( \e [\nnorm{{\X }_{0}}{\Hmath }{q}] + 1 \Bigr) .
\end{equation}
Recall that  $\taunR \uparrow T$ as $R\to \infty $, $\p $-a.s. and $\p \{ \taunR <T\} =0 $.
Using the Fatou lemma we infer that
\begin{equation}
\e \biggl[ \sup_{s\in [0,T]} \nnorm{\Xn (s)}{\Hmath }{q}
+ q \, \int_{0}^{T}\nnorm{\Xn (s)}{\Hmath }{q-2} \norm{\Xn (s)}{}{2} \, ds  \biggr]  
\; \le \;  C(\eps ,q,T ) \,  \bigl( \e [\nnorm{{\X }_{0}}{\Hmath }{q}] + 1 \bigr) .
\end{equation}
From the above inequality we obtain estimates \eqref{eq:H_estimate_truncated_p}, \eqref{eq:HV_estimate_truncated} and \eqref{eq:V_estimate_truncated}.  
This concludes the proof of Lemma \ref{lem:Hall-MHD_truncated_estimates}.
\end{proof}

\medskip
\section{Compactness and tightness results.} \label{sec:comp-tight}

\medskip
\subsection{The space $\Umath $} \label{sec:aux_funct.anal}

\medskip  \noindent
It is well known that in the case when the domain is ${\mathbb{R} }^{3}$, thus unbounded, the standard Sobolev embedding are not compact. To overcome this problem we introduce auxiliary space $\Umath $. 
In the functional setting of the Hall-MHD equations we have the following three basic spaces
\[
\Vtest \; \subset \; \Vmath \; \subset \; \Hmath  ,
\]
see Section \ref{sec:Hall-MHD_funct-setting} and Definition \ref{def:mart-sol}. 
For fixed $m > \frac{5}{2}$,  let us consider the space
\begin{equation}
\Vast \; := \; {\Vmath }_{m}, 
\label{eq:Vast}
\end{equation}
where ${\Vmath }_{m}  $ is defined by \eqref{eq:Vmath_m}.
The choice of the space $\Vast $ corresponds to the properties of nonlinear maps $\MHD $ and $\tHall $, 
see Lemmas \ref{lem:MHD-term_properties}(iii) and \ref{lem:tHall-term_properties}(iii) and Corollaries 
\ref{cor:MHD-map_conv-aux} and \ref{cor:tHall-term_conv_general}.

\medskip  \noindent
Since $\Vast $ is dense in $\Hmath $ and the embedding $\Vast \hookrightarrow \Hmath $ is continuous, 
by Lemma 2.5 from \cite{Holly+Wiciak'1995} (see \cite[Lemma C.1]{Brze+EM'13})
there exists a separable Hilbert space $\Umath $  such that 
$\Umath \subset \Vast $, $\Umath $ is dense in $\Vast $ and  
\begin{equation}
\mbox{the embedding  ${\iota }_{} : \Umath \hookrightarrow \Vast   $ is compact.}
\label{eq:Umath}
\end{equation}
Then we have
\begin{equation}
\Umath \; \hookrightarrow  \; \Vast \; \; \hookrightarrow  \; \Vtest \;
\hookrightarrow \; \Vmath \; \hookrightarrow \; \Hmath .
\label{eq:spaces}
\end{equation}

\medskip
\subsection{The space $\zcal $}

\medskip  \noindent
In this section we define the space $\zcal $ which plays important role in our approach.
By \eqref{eq:Umath} and  \eqref{eq:spaces}, in particular,
we have
\begin{equation}
\Umath  \hookrightarrow \; \Vmath \; \hookrightarrow \; \Hmath 
\; \cong \; {\Hmath }^{\prime } \; \hookrightarrow \; {\Umath }^{\prime },
\end{equation}
the embedding  $ \Umath \hookrightarrow \Vmath   $ being compact.
To define the space $\zcal $ we will need the following four functional spaces
being the counterparts in our framework of the spaces used in \cite{EM'13}, see also \cite{Metivier'88}:
\begin{itemize}
\item $\Dmath ([0,T],{\Umath }^{\prime }) $ := the space of  c\`{a}dl\`{a}g functions 
 $ \phi :[0,T] \to {\Umath }^{\prime } $  with the topology  $ {\tcal }_{1}$
               induced by the Skorohod metric (see Appendix \ref{sec:Cadlag_functions}),
\item ${L}_{w}^{2}(0,T;\Vmath ) $ := the space ${L}^{2} (0,T;\Vmath )$ with the weak topology 
                     $ {\tcal}_{2} $,      
\item ${L}^{2}(0,T;{\Hmath }_{loc})$ := the space of measurable functions 
 $ \phi :[0,T]\to   \Hmath  $ such that for all $ R \in \mathbb{N} $
 \begin{equation*}                 
{p}_{T,R}(\phi):= \Bigl(  \int_{0}^{T} \int_{{\mathcal{O} }_{R}} \bigl[ {|{\phi }_{1} (t,x)|}^{2} +  {|{\phi }_{2} (t,x)|}^{2} \bigr] \, dxdt {\Bigr) }^{\frac{1}{2}}
<\infty , 
\end{equation*}   
where $\phi  = ({\phi }_{1},{\phi }_{2})$,   
with the topology  $ {\tcal }_{3}$  generated by the seminorms 
$({p}_{T,R}{)}_{R\in \mathbb{N} } .$     
\end{itemize}
Let ${\Hmath }_{w}$ denotes the Hilbert space $\Hmath $ endowed with the weak topology. 
Let us consider the fourth space, see \cite{EM'13},
\begin{itemize} 
\item $\Dmath ([0,T];{\Hmath }_{w}) $ : = the space of all function  $\phi :[0,T]\to \Hmath $ such that
for every $h\in \Hmath $ 
\[
[0,T] \; \ni \; t \; \mapsto \; \ilsk{\phi (t)}{h}{\Hmath } \in \; \mathbb{R} \;  
\] 
is a real-valued c\`{a}dl\`{a}g function. In the space $\Dmath ([0,T];{\Hmath }_{w})$ we consider the weakest topology ${\tcal }_{4}$ such that for all $h \in \Hmath  $   the  maps 
\begin{equation}
\Dmath ([0,T];{\Hmath }_{w}) \; \ni \; \phi   \; \mapsto \;  \ilsk{\phi (\cdot )}{h}{\Hmath } \; \in \; \Dmath  ([0,T];\mathbb{R} ) 
\label{eq:D([0,T];H_w)_cadlag}  
\end{equation} 
are continuous.                     
In particular,  
${\phi }_{n} \to \phi  $ in $\Dmath ([0,T];{\Hmath }_{w}) $ iff  for all $ h \in \Hmath  $:
$
\ilsk{{\phi }_{n} (\cdot )}{h}{\Hmath }  \to \ilsk{\phi (\cdot )}{h}{\Hmath }   \mbox{ in the space }  
  \Dmath ([0,T];\mathbb{R} ).
$ 
\end{itemize}

\medskip  
\noindent
Let us consider the ball
\[
\ball \; := \; \{ x \in \Hmath  : \, \, \, \nnorm{x}{\Hmath }{} \le r \} .
\]
Let ${\ball }_{w}$ denote the ball $\ball $ endowed with the weak topology.
It is well-known that  ${\ball }_{w}$ is metrizable, see \cite{Brezis'2011}. 
Let ${q}_{r}$ denote the metric compatible with the weak topology on $\ball $.
Let us denote by 
 $\Dmath ([0,T]; {\ball }_{w}) $    the space of  functions  
 $\phi \in \Dmath ([0,T];{\Hmath }_{w})$  such that 
\begin{equation}  
\sup_{t \in [0,T]} \nnorm{\phi (t)}{\Hmath }{} \; \le \;  r  .  \label{eq:D([0,T];B_w)} 
\end{equation}          
The space $\Dmath ([0,T]; {\ball }_{w})$ is  completely metrizable, as well.
In fact, $\Dmath ([0,T]; {\ball }_{w})$ is metrizable with 
\begin{equation} \label{eq:metric_D([0,T];B_w)}
{\delta }_{T,r}(\phi ,v) \; = \; \inf_{\lambda \in {\Lambda }_{T}} \! \biggl\{ \sup_{t\in [0,T]} \! {q}_{r}(\phi (t),v\circ \lambda (t))\! + \! \sup_{t \in [0,T]} \! |t-\lambda (t)|  + \sup_{s \ne t } \! \Bigl| \log \frac{\lambda (t)-\lambda (s)}{t-s} \Bigr|  \biggr\} .
\end{equation}
Since by the Banach-Alaoglu Theorem  ${\ball }_{w}$ is compact, $(\Dmath ([0,T];{\ball }_{w}),{\delta }_{T,r} )$ is a complete metric space. 

\medskip
\begin{definition}\label{def:space_Z_cadlag}
\rm Let us put
\begin{equation}
\label{eq:Z_cadlag}
\mathcal{Z} \; := \; {L}_{w}^{2}(0,T;\Vmath )  
\cap {L}^{2}(0,T;{\Hmath }_{loc})  \cap \Dmath ([0,T];{\Hmath }_{w}) \cap \Dmath ([0,T]; {\Umath }^{\prime }) 
\end{equation}
and let  $\mathcal{T} $ be   the supremum of the corresponding four topologies, i.e. the smallest topology on $\mathcal{Z}$ such that the four natural embeddings from $\mathcal{Z}$ are continuous.
The space  $\mathcal{Z}$  will  also be considered with the Borel $\sigma $-field, denoted by $\sigma (\mathcal{Z})$, i.e. the smallest $\sigma $-field containing the family $\mathcal{T} $.
\end{definition}

\medskip
\subsection{Deterministic compactness  theorem} 

\medskip  
\noindent
The following lemma says that any bounded sequence $({\phi }_{n}) \subset {L}^{\infty } (0,T;\Hmath )$  convergent in 
$\Dmath ([0,T];{\Umath }^{\prime })$ is  convergent in the space $\Dmath ([0,T];{\ball }_{w})$, as well.
It is closely related to the lemma due to Strauss, see \cite{Strauss'66}, that says:
\begin{equation} \label{eq:Strauss}
{L}^{\infty }(0,T;\Hmath ) \cap \mathcal{C} ([0,T]; {\Umath }_{w}')
\; \subset \; \mathcal{C} ([0,T];{\Hmath }_{w}),
\end{equation}
where $\mathcal{C} ([0,T];{\Umath }_{w}')$ and $\mathcal{C} ([0,T];{\Hmath }_{w})$ denote the space of $\Umath '$ and $\Hmath $-valued, respectively,  weakly continuous functions.

\medskip
\begin{lemma} \label{lem:D(0,T,{hmath }_{w})_conv}
(See \cite[Lemma 2]{EM'13}.) \it
Let ${\phi }_{n}:[0,T] \to \Hmath  $, $n \in \mathbb{N} $, be functions such that
\begin{description}
\item[(i) ] $\sup_{n \in \mathbb{N} } \sup_{s \in [0,T]} \nnorm{{\phi }_{n} (s)}{\Hmath }{} \le r  $,
\item[(ii) ] ${\phi }_{n} \to \phi  $ in $\Dmath ([0,T];{\Umath }^{\prime })$.
\end{description}
Then  $\phi , {\phi }_{n} \in \Dmath ([0,T];{\ball }_{w}) $ and ${\phi }_{n} \to \phi $ in $\Dmath ([0,T];{\ball }_{w})$ as $n \to \infty $.
\end{lemma}

\medskip  \noindent
The compactness criterion contained in the following Theorem \ref{th:Dubinsky_cadlag_unbound} can be seen as the generalization of the classical Dubinsky theorem, see \cite[Theorem IV.4.1]{Vishik+Fursikov'88},  to the case when the embedding $\Vmath \subset \Hmath $ is continuous and possibly not compact and the space of continuous functions is replaced by the space of c\`{a}dl\`{a}g functions $\Dmath ([0,T];\Umath ')$. 
Together with Lemma \ref{lem:D(0,T,{hmath }_{w})_conv}, it is a simple modification of Lemma 4.1 from  \cite{EM'14}. See also \cite[Theorem 2]{EM'13} for the case of the Navier-Stokes equations. 

\medskip  
\noindent
\begin{theorem} \rm  \label{th:Dubinsky_cadlag_unbound} \it
(See \cite[Lemma 4.1]{EM'14}.)
Let
\[
\mathcal{K} \; \subset \; {L}^{\infty }(0,T;\Hmath ) \, \cap \, {L}^{2}(0,T;\Vmath ) \, \cap \, \Dmath ([0,T];{\Umath }^{\prime })
\]
be a set satisfying the following three conditions 
\begin{description}
\item[(a) ] for all $\phi  \in \mathcal{K} $ and  all $t \in [0,T]$, $\phi (t) \in \Hmath   $ and 
$\, \, \sup_{\phi \in \mathcal{K} } \sup_{s \in[0,T]} \nnorm{\phi (s)}{\Hmath }{} < \infty  $, 
\item[(b) ] $ \sup_{\phi \in \mathcal{K} } \int_{0}^{T} \norm{\phi (s)}{\Vmath }{2} \, ds < \infty  $,
  i.e. $\mathcal{K} $ is bounded in ${L}^{2}(0,T;\Vmath )$,
\item[(c) ] $\lim{}_{\delta \to 0 } \sup_{\phi \in \mathcal{K} } {w}_{[0,T],{\Umath }^{\prime }}(\phi ;\delta ) =0 $.
\end{description}
Then $\mathcal{K} \subset {\zcal }_{}$  and $\mathcal{K} $ is $\tcal $-relatively compact in ${\zcal }_{}$ defined by \eqref{eq:Z_cadlag}.
\end{theorem}

\medskip
\subsection{Tightness criterion} 

\medskip \noindent
Let $(\Omega , \mathcal{F} ,\p )$ be a probability space with filtration $\mathbb{F}:=({\mathcal{F} }_{t}{)}_{t \in [0,T]}$ satisfying the usual conditions.
Using Theorem \ref{th:Dubinsky_cadlag_unbound}, we get the corresponding tightness criterion in the
measurable  space  $(\zcal ,\sigma (\zcal ))$.

\medskip
\begin{cor}  \label{cor:tigthness_criterion_cadlag_unbound}
(See \cite[Corollary 1]{EM'13}.)
\it Let $(\Xn {)}_{n \in \mathbb{N} }$ be a sequence of c\`{a}dl\`{a}g $\mathbb{F}$-adapted 
${\Umath }^{\prime }$-valued processes such that
\begin{description}
\item[(a)] there exists a positive constant ${C}_{1}$ such that
\[
 \sup_{n\in \mathbb{N}}\e \bigl[ \sup_{s \in [0,T]} \nnorm{\Xn (s)}{\Hmath }{}  \bigr] \; \le \; {C}_{1} ,
\]
\item[(b)] there exists a positive constant ${C}_{2}$ such that
\[
 \sup_{n\in \mathbb{N}}\e \Bigl[  \int_{0}^{T} \norm{\Xn (s)}{\Vmath }{2} \, ds    \Bigr] \;  \le \; {C}_{2} ,
\]
\item[(c)]  $(\Xn {)}_{n \in \mathbb{N} }$ satisfies the Aldous condition  in ${\Umath }^{\prime }$.
\end{description}
Let ${\tilde{\p }}_{n}$ be the law of $\Xn $ on ${\zcal }_{}$.
Then for every $\eps >0 $ there exists a compact subset ${K}_{\eps }$ of ${\zcal }_{}$ such that
\[
{\tilde{\p }}_{n} ({K}_{\eps }) \; \ge \;  1 - \eps .
\]
\end{cor}

\medskip  \noindent 
For the completeness of presentation we recall the Aldous condition in the form given by  M\'{e}tivier \cite{Metivier'88}.

\medskip  \noindent
\begin{definition} (M. M\'{e}tivier) \label{def:Aldous_U'}
\rm A sequence $({X}_{n}{)}_{n\in \mathbb{N} }$  satisfies the \bf  Aldous condition \rm
in the space ${\Umath }^{\prime }$
iff
\begin{description}
\item[$\mbox{\bf [A]\rm }$ ] for every $\eps >0 $  and $\eta >0 $ there exists $\delta >0 $ such that for every sequence $({{\tau}_{n} } {)}_{n \in \mathbb{N} }$ of $\mathbb{F}$-stopping times with
${\tau }_{n}\le T$ one has
\[
\sup_{n \in \mathbb{N}} \, \sup_{0 \le \theta \le \delta }  \p \bigl\{
\nnorm{ {X}_{n} ({\tau }_{n} +\theta )-{X}_{n} ( {\tau }_{n}  ) }{{\Umath }^{\prime }}{} \ge \eta \bigr\} 
\; \le \;  \eps .
\]
\end{description}
\end{definition}

\medskip  \noindent
In the following lemma we recall a certain condition which guarantees that the sequence $({X}_{n}{)}_{n\in \mathbb{N} }$  satisfies  condition \rm \bf [A]\rm .

\medskip  
\begin{lemma} \label{lem:Aldous_criterion} 
(See \cite[Lemma 9]{EM'13})
 \it
Let $(E,\norm{\cdot }{E}{})$ be a separable Banach space and let $({X}_{n}{)}_{n \in \mathbb{N} }$ be a sequence of $E$-valued random variables such that 
\begin{itemize} 
\item[\mbox{\bf [A']\rm } ] there exist $\alpha ,\beta >0 $ and $C>0$ such that
for every sequence $({{\tau}_{n} } {)}_{n \in \mathbb{N} }$ of $\mathbb{F}$-stopping times with
${\tau }_{n}\le T$ and for every $n \in \mathbb{N} $ and $\theta \ge 0 $ the following condition holds
\begin{equation} \label{eq:Aldous_est}
\e \bigl[ \bigl( \norm{ {X}_{n} ({\tau }_{n} +\theta )-{X}_{n} ( {\tau }_{n}  ) }{E}{\alpha } \bigr] 
\; \le \;  C {\theta }^{\beta } .
\end{equation}
\end{itemize}
Then the sequence $({X}_{n}{)}_{n\in \mathbb{N} }$  satisfies  condition \rm \bf [A] \rm in the space $E$.  \rm
\end{lemma}

\medskip
\section{Existence of a martingale solution} \label{sec:existence}

\medskip  \noindent
We will use the structure introduced in Section \ref{sec:aux_funct.anal}.
By \eqref{eq:spaces} we have the following spaces $\Umath  \subset \Vast \subset \Vtest  \subset \Vmath \subset  \Hmath $, where $\Vast = {\Vmath }_{m}$ for fixed $m>\frac{5}{2}$.
Considering  the dual spaces, and identifying $\Hmath $ with its dual $\Hmath '$, we have the following system 
\begin{equation}
\Umath \; \subset \;  \Vast \; \subset \Vtest  \; \subset \; \Vmath \; \subset \; \Hmath \; \cong \; \Hmath '
\; \to \; \Vprime \; \to \; \Vastprime \; \to \; \Uprime .  
\label{eq:spaces_system} 
\end{equation}
In \eqref{eq:P_n} we have defined  the map 
$ 
  \Pn : \Hmath  \to \Hn
$
being the  $\ilsk{\cdot }{\cdot }{\Hmath }$-orthogonal projection onto $(\Hn ,\ilsk{\cdot }{\cdot }{\Hmath })$. Using further properties of the map $\Pn $  stated in Corollary \ref{cor:P_n-pointwise_conv} we will consider the adjoint operators  $\Pn '$. 
 
\medskip  
\begin{remark} \label{rem:Pn_dual}
\rm \
 Since  $\Pn \in \mathcal{L} ( \Vmath , \Vmath )$,  its adjoint $\Pn ' \in \mathcal{L} (\Vprime ,\Vprime )$ by definition  satisfies
\begin{equation} \label{eq:P_n_V_V_adj}
\ddual{\Vprime}{\xi }{\Pn \varphi}{\Vmath }
\; = \; \ddual{\Vprime}{{\Pn '}\xi }{ \varphi}{\Vmath } , \qquad
\mbox{ for all } \quad \xi \in \Vprime , \quad \varphi \in \Vmath   . 
\end{equation} 
Since  $\Pn \in \mathcal{L} (\Vast , \Vmath) $,  its adjoint $\Pn ' \in \mathcal{L} (\Vprime , \Vastprime )$
satisfies
\begin{equation} \label{eq:P_n_Vast_V_adj}
\ddual{\Vprime}{\xi }{\Pn \varphi}{\Vmath }
\; = \; \ddual{\Vastprime}{{\Pn '}\xi }{ \varphi}{\Vast } , \qquad
\mbox{ for all } \quad \xi \in \Vprime , \quad \varphi \in \Vast   .    
\end{equation}
Since  $\Pn \in \mathcal{L} (\Vast ,\Vast )$,  its adjoint $\Pn ' \in \mathcal{L} (\Vastprime , \Vastprime )$ satisfies
\begin{equation} \label{eq:P_n_Vast_Vast_adj}
\ddual{\Vastprime}{\xi }{\Pn \varphi}{\Vast }
 \,\, = \,\, \ddual{\Vastprime}{{\Pn '}\xi }{ \varphi}{\Vast } , \qquad
\mbox{ for all } \quad \xi \in \Vastprime , \quad \varphi \in \Vast   .    
\end{equation}
\end{remark}

\medskip  \noindent
Remark \ref{rem:Pn_dual} enables us to rewrite the truncated equation  as an equation in the space $\Vastprime $, and by the injection $\Vastprime \to \Uprime $ - also as an equation in $\Uprime $.

\medskip  
\begin{remark} \label{rem:truncated_U'} \rm \
\begin{description}
\item[(i) ] If the $\Hn$-valued process $\Xn $ satisfies identity \eqref{eq:Hall-MHD_truncated_weak}, then
in particular, for all $t \in [0,T]$ and $\varphi \in \Vast   $ we have $\Pn \varphi \in \Hn$ and  
\begin{equation}
\begin{split}
&\ilsk{\Xn (t)}{\Pn \varphi}{\Hmath }
 + \int_{0}^{t} \ddual{\Vprime }{ \mathcal{A}  \Xn (s)}{\Pn  \varphi }{\Vmath } \, ds
+ \int_{0}^{t} \ddual{\Vastprime}{ \MHD  (\Xn (s))}{\Pn \varphi}{\Vast}  ds  \\
& \qquad \quad 
+ \int_{0}^{t} \ddual{\Vastprime }{ \tHall   (\Xn (s))}{\Pn  \varphi }{\Vast } \, ds
  \\
& \quad  \; = \; \ilsk{{\X }_{0}}{\Pn \varphi}{\Hmath }
+ \int_{0}^{t} \int_{Y}  \ilsk{F(s,\Xn ({s}^{-});y ) }{\Pn \varphi  }{\Hmath } \, \tilde{\eta }(ds,dy) .
\end{split} \label{eq:truncated_weak_Pn}
\end{equation}
Since $\Pn : \Hmath \to \Hn $ is an $\ilsk{\cdot }{\cdot }{\Hmath }$-orthogonal projection, we have
$
\ilsk{\Xn (t)}{ \Pn \varphi}{\Hmath } = \ilsk{\Pn \Xn (t)}{ \varphi}{\Hmath } = \ilsk{\Xn (t)}{ \varphi}{\Hmath }.
$
Using the operators $\Pn '$ from Remark \ref{rem:Pn_dual}, identity   \eqref{eq:truncated_weak_Pn} can be written in the form
\begin{equation}
\begin{split}
&\ilsk{\Xn (t)}{ \varphi}{\Hmath }
 + \int_{0}^{t} \ddual{\Vprime }{ \Pn ' \mathcal{A}    \Xn (s)}{  \varphi }{\Vmath  } \, ds
+ \int_{0}^{t} \ddual{\Vastprime }{ \Pn ' \MHD  (\Xn (s))}{ \varphi}{\Vast }  ds \\
& \qquad \quad  + \int_{0}^{t} \ddual{\Vastprime }{\Pn ' \tHall (\Xn (s))}{  \varphi }{\Vast } \, ds
\\
& \quad  \; = \; \ilsk{\Pn {\X }_{0}}{ \varphi}{\Hmath }
+ \int_{0}^{t} \int_{Y}  \ilsk{F_n(s,\Xn ({s}^{-});y ) }{\varphi  }{\Hmath } \, \tilde{\eta }(ds,dy) . \qquad t \in [0,T] ,
\end{split}   \label{eq:truncated_weak_Pn'}
\end{equation}
where $F_n$ is defined by \eqref{eq:F_n}.
\item[(ii) ] By the identification $\Hmath \cong \Hmath '$ and the fact that 
$\Hmath ' \hookrightarrow \Vmath ' \hookrightarrow \Vastprime $,
we identify $\Xn (t)$ with the functional induced by $\Xn (t)$ on the space $\Vast $. From \eqref {eq:truncated_weak_Pn'}, we infer that $\Xn (t)$ satisfies the following equation
\begin{equation} 
\begin{split}
& \Xn (t) + \int_{0}^{t}\bigl[ \Pn ' \mathcal{A} \Xn (s)
+ \Pn ' \MHD   \bigl( \Xn (s) \bigr) + \Pn ' \tHall  (\Xn (s)) \bigr] \, ds  
\\ 
& = \,\,  \Pn {\X }_{0} 
+ \int_{0}^{t} \int_{Y} F_n(s,\Xn ({s}^{-});y )  \, \tilde{\eta }(ds,dy),    \qquad  t \in [0,T]  .  
\end{split} \label{eq:truncated}
\end{equation}  
\end{description}
\end{remark}

\medskip
\subsection{Tightness} \label{sec:tightness}

\medskip  \noindent
Let us consider the sequence ${(\Xn )}_{n\in \mathbb{N} }$ of approximate solutions. We will prove that the sequence of laws is tight in the space $\zcal $ defined by \eqref{eq:Z_cadlag}, i.e.
\begin{equation*}
\mathcal{Z} \; := \; {L}_{w}^{2}(0,T;\Vmath )  
\cap {L}^{2}(0,T;{\Hmath }_{loc})  \cap \Dmath ([0,T];{\Hmath }_{w}) \cap \Dmath ([0,T]; {\Umath }^{\prime }), 
\end{equation*}
equipped with the Borel $\sigma $-field  $\sigma(\tcal )$, see Definition \ref{def:space_Z_cadlag}.

\medskip
\begin{lemma} \label{lem:X_n-tightness}
The set of probability measures $\{ \Law (\Xn ), n \in \mathbb{N}   \} $ is tight on the space $(\zcal ,\sigma(\tcal ))$. 
\end{lemma}

\medskip
\begin{proof}
We apply Corollary \ref{cor:tigthness_criterion_cadlag_unbound}.
By estimates  \eqref{eq:H_estimate_truncated_p} and \eqref{eq:V_estimate_truncated}, conditions \bf (a) \rm  and \bf (b) \rm   of Corollary \ref{cor:tigthness_criterion_cadlag_unbound} are satisfied. Thus, it is sufficient to  prove that the sequence $(\Xn {)}_{n \in \mathbb{N} }$ satisfies the Aldous condition \textbf{[A]}. 
By Lemma \ref{lem:Aldous_criterion} it is sufficient to prove that $(\Xn {)}_{n \in \mathbb{N} }$ satisfies condition \textbf{[A']}.

\medskip  \noindent
Let ${(\taun )}_{n \in \mathbb{N}} $ be a sequence of stopping times taking values in $[0,T]$.
By Remark \ref{rem:truncated_U'} (ii), we have
\begin{equation*}
\begin{split}
\Xn (t) 
\; &= \; \Pn {\X }_{0}  - \int_{0}^{t} \Pn ' \mathcal{A}  \Xn (s) \, ds 
    - \int_{0}^{t} \Pn '\MHD \bigl( \Xn (s) \bigr) \, ds
  - \int_{0}^{t} \Pn ' \tHall   (\Xn (s)) \, ds  
  \\  
& \qquad   + \int_{0}^{t} \int_{Y}  F_n (s,\Xn ({s}^{-});y )  \, \tilde{\eta }(ds,dy)
   \\
\; & =: \;   \Jn{1} + \Jn{2}(t) + \Jn{3}(t) + \Jn{4}(t) + \Jn{5}(t) , \qquad t \in [0,T].
\end{split}  
\end{equation*}
Let us choose  $\theta  >0 $.  It is sufficient to show that each sequence $\Jn{i}$ of processes, $i=1,\cdots, 5$, satisfies the sufficient condition \textbf{[A']} from Lemma \ref{lem:Aldous_criterion}.

\medskip  \noindent
Obviously the term $\Jn{1}$ which is  constant in time,  satisfies this condition.
 In fact, we will check that the term $\Jn{2}$ satisfies condition
 \textbf{[A']} from Lemma \ref{lem:Aldous_criterion} in the space $E={{\Vmath } }^{\prime }$, the terms $\Jn{3}, \Jn{4}$ satisfy this condition in  $E=\Vastprime $ and the term $\Jn{5}$ satisfies this condition in $E=\Hmath $. Since the embeddings $\Hmath \subset {\Umath }^{\prime }$,
 ${{\Vmath } }^{\prime } \subset {\Umath }^{\prime } $ and
$\Vastprime \subset {\Umath }^{\prime }$   are continuous, we infer that
\textbf{[A']} holds in the space $E={\Umath }^{\prime }$, as well.

\medskip \noindent
\textbf{ Ad } ${\Jn{2}}$. \rm  By Remark \ref{rem:Acal-term_properties},
 the linear operator $\mathcal{A} :{\Vmath }  \to {{\Vmath } }^{\prime }$ is bounded, 
and by Corollary \ref{cor:P_n-pointwise_conv},  $\sup_{n\in \mathbb{N} } \nnorm{\Pn }{\mathcal{L} (\Vmath ,\Vmath )}{} < \infty  $. Using the H\"older inequality and \eqref{eq:V_estimate_truncated}, we obtain
\begin{equation}
\begin{split}
& \mathbb{E}\,\bigl[ \nnorm{ \Jn{2} (\taun + \theta ) - \Jn{2}(\taun )  }{{{\Vmath } }^{\prime }}{}  \bigr]
\; \le \; \nnorm{\Pn '}{\mathcal{L} (\Vprime ,\Vprime )}{} \, \mathbb{E}\,\Bigl[  \int_{\taun }^{\taun + \theta }
\nnorm{\mathcal{A}  \Xn (s) }{{\Vmath }^{\prime } }{}  \, ds \biggr]
 \\
& \; \le \;  \nnorm{\Pn }{\mathcal{L} (\Vmath ,\Vmath )}{} \, {\theta }^{\frac{1}{2}} \Bigl( \mathbb{E}\,\Bigl[  \int_{0 }^{T }
 \norm{  \Xn (s) }{}{2}  \, ds \Bigr] {\Bigr) }^{\frac{1}{2}}
\; \le \;   \nnorm{\Pn }{\mathcal{L} (\Vmath ,\Vmath )}{} \, {C}_{2}^{\frac{1}{2}} \cdot {\theta }^{\frac{1}{2}}
\; = \; c_2 \cdot {\theta }^{\frac{1}{2}} ,  
\end{split} \label{eqn-Jn2} 
\end{equation}
where $c_2 = {C}_{2}^{\frac{1}{2}} \cdot \sup_{n\in \mathbb{N} } \nnorm{\Pn }{\mathcal{L} (\Vmath ,\Vmath )}{} <\infty  $.

\medskip 
\noindent
\textbf{Ad } ${\Jn{3}}$. \rm 
By \eqref{eq:MHD-map_est-H-H} in Lemma \ref{lem:MHD-term_properties},
 $\MHD : {\Hmath } \times {\Hmath } \to {\Vmath }_{\ast }^{\prime } $ is bilinear and continuous (and hence bounded so that   the norm $\| \MHD \| $ of $\MHD : {\Hmath } \times {\Hmath } \to {{\Vmath }}_{\ast  }^{\prime }$ is finite), and  by Corollary \ref{cor:P_n-pointwise_conv},  $\sup_{n\in \mathbb{N} } \nnorm{\Pn }{\mathcal{L} (\Vast ,\Vast )}{} < \infty  $.  Then by \eqref{eq:H_estimate_truncated_p} we have the following estimates
\begin{equation}
\begin{split}
&\mathbb{E}\,\bigl[ \nnorm{ \Jn{3}(\taun +\theta ) - \Jn{3}(\taun)}{{{\Vmath }}_{\ast }^{\prime }}{}  \bigr]
\; = \; \nnorm{\Pn '}{\mathcal{L} (\Vastprime ,\Vastprime )}{} \, \mathbb{E}\,\Bigl[  \Nnorm{ \int_{\taun }^{\taun + \theta }
\MHD \bigl(\Xn (r) \bigr) \, dr }{{{\Vmath }}_{\ast }^{\prime }}{} \Bigr]   \\
\; & \le \;  \nnorm{\Pn }{\mathcal{L} (\Vast ,\Vast )}{} \, \mathbb{E}\,\Bigl[  \int_{\taun }^{\taun + \theta }
\nnorm{ \MHD ( \Xn (r)  )  }{{{\Vmath }}_{\ast }^{\prime }}{} \, dr \Bigr]
\; \le \;  \nnorm{\Pn }{\mathcal{L} (\Vast ,\Vast )}{} \, \| \MHD \| \, \mathbb{E}\,\biggl[  \int_{\taun }^{\taun + \theta }  \nnorm{\Xn (r)}{{\Hmath }}{2}   \, dr \biggr]    \\
\; & \le \;  \nnorm{\Pn }{\mathcal{L} (\Vast ,\Vast )}{} \, \| \MHD \|  \cdot  \mathbb{E}\,\bigl[ \sup_{r \in [0,T]} \nnorm{\Xn (r)}{{\Hmath }}{2}\bigr] \cdot \theta
\; \le \;  \nnorm{\Pn }{\mathcal{L} (\Vast ,\Vast )}{} \,  \| \MHD \| \,  {C}_{1}(2)    \cdot \theta \; = \; c_3 \, \theta ,   
\end{split}  \label{eq:Jn3}
\end{equation}
where $c_3=  \sup_{n\in \mathbb{N} }\nnorm{\Pn }{\mathcal{L} (\Vast ,\Vast )}{} \,  \| \MHD \| \,  {C}_{1}(2) <\infty  $.

\medskip  \noindent
\textbf{Ad } ${\Jn{4}}$. \rm 
By \eqref{eq:tHall-map_est-H-V} in Lemma \ref{lem:tHall-term_properties} $\tHall : {\Hmath } \times {\Vmath } \to \Vastprime $ is bilinear and continuous (and hence bounded so that   the norm $\| \tHall \| $ of $\tHall : {\Hmath } \times {\Vmath } \to \Vastprime $ is finite) and by Corollary \ref{cor:P_n-pointwise_conv},  $\sup_{n\in \mathbb{N} } \nnorm{\Pn }{\mathcal{L} (\Vast ,\Vast )}{} < \infty  $.  Then by \eqref{eq:H_estimate_truncated_p} and \eqref{eq:V_estimate_truncated} we have the following estimates
\begin{equation}
\begin{split}
&  \mathbb{E}\,\bigl[ {| \Jn{4} (\taun + \theta ) - \Jn{3}(\taun) |}_{\Vastprime }  \bigr]
\; \le \;  \nnorm{\Pn '}{\mathcal{L} (\Vastprime ,\Vastprime )}{} \, \mathbb{E}\,\Bigl[  \int_{\taun }^{\taun + \theta } \nnorm{\tHall \bigl( \Xn (r)  \bigr) }{\Vastprime }{} \, dr \Bigr]   \\
\; & \le \;  \nnorm{\Pn }{\mathcal{L} (\Vast ,\Vast )}{} \norm{\tHall }{}{} \, \mathbb{E}\Bigl[ \int_{\taun }^{\taun + \theta }   \nnorm{ \Xn (r)}{\Hmath }{} \, \norm{  \Xn (r)}{}{}  \, dr \Bigr]  \\
\; &\le \;  \nnorm{\Pn }{\mathcal{L} (\Vast ,\Vast )}{} \norm{\tHall }{}{} \,  \biggl( \mathbb{E} \Bigl[   \sup_{r \in [\taun,\taun + \theta]}  \nnorm{\Xn (r)}{\Hmath }{2} \Bigr] {\biggr) }^{\frac{1}{2}}
\biggl( \mathbb{E} \Bigl[ \int_{\taun }^{\taun + \theta }
 \norm{\Xn (r)}{}{2} \, dr \Bigr] {\biggr) }^{\frac{1}{2}}  {\theta }^{\frac{1}{2}}   \\
\; & \le \;  \nnorm{\Pn }{\mathcal{L} (\Vast ,\Vast )}{} \norm{\tHall }{}{} \,   \biggl( \mathbb{E} \Bigl[   \sup_{r \in [0,T]}  \nnorm{\Xn (r)}{{\Hmath } }{2} \Bigr]  {\biggr) }^{\frac{1}{2}}
\biggl( \mathbb{E} \Bigl[ \int_{0}^{T}
       \norm{  \Xn (r)}{}{2} \, dr \Bigr] {\biggr) }^{\frac{1}{2}} {\theta }^{\frac{1}{2}}
 \\
\; & \le \;  \nnorm{\Pn }{\mathcal{L} (\Vast ,\Vast )}{} \norm{\tHall }{}{} \,  
 [{C}_{1}(2)]^{\frac{1}{2}} [{C}_{2}]^{\frac{1}{2}} {\theta }^{\frac{1}{2}} 
\; = \; c_4  {\theta }^{\frac{1}{2}}  , 
\end{split}
 \label{eqn-Jn4}
\end{equation}
where $c_4 =  \sup_{n\in \mathbb{N} }\nnorm{\Pn }{\mathcal{L} (\Vast ,\Vast )}{} \norm{\tHall }{}{} \,   [{C}_{1}(2)]^{\frac{1}{2}} [{C}_{2}]^{\frac{1}{2}} < \infty $.

\medskip \noindent
\textbf{Ad } ${\Jn{5}}$. \rm 
Let us consider the noise term ${\Jn{5}}$.  
By \eqref{eq:isometry}, the fact that $\Pn : \Hmath \to \Hn $ is the $\ilsk{\cdot }{\cdot }{\Hmath }$-orthogonal projection,
inequality \eqref{eq:F_linear_growth} with $p=2$  and by \eqref{eq:H_estimate_truncated_p}, we obtain 
\begin{equation}
\begin{split}
& \e \bigl[ \nnorm{\Jn{5} (\taun + \theta ) - \Jn{5}(\taun) }{\Hmath }{2}  \bigr] 
\; = \;  \e \Bigl[ \Nnorm{ \int_{\taun }^{\taun + \theta } \int_{Y} F_n(s, \Xn  ({s}^{-});y ) \, \tilde{\eta}(ds,dy) }{\Hmath }{2} 
\Bigr]  \\
\; &= \; \e \Bigl[  \int_{\taun }^{\taun + \theta } 
  \int_{Y} \nnorm{F_n(s, \Xn  (s);y ) }{\Hmath }{2} \, \mu (dy)ds   \Bigr]  
\; \le \;  {K}_{2}  \e \Bigl[  \int_{\taun }^{\taun + \theta } (1+\nnorm{\Xn (s)}{\Hmath }{2}) \, ds  \Bigr]  
 \\
\; &\le \; {K}_{2} \cdot  \theta \cdot \Bigl( 1+
  \e \Bigl[ \sup_{s \in [0,T]} \nnorm{\Xn (s)}{\Hmath }{2} \Bigr] \Bigr)
\; \le \; {K}_{2} \cdot (1+ {C}_{1}(2)) \cdot \theta \; =: \; {c}_{5} \cdot \theta , 
\end{split}
\label{eq:Jn5}
\end{equation}
where ${c}_{5} = {K}_{2} \cdot (1+ {C}_{1}(2))$.
Thus ${\Jn{5}}$ satisfies condition \eqref{eq:Aldous_est}. 

\medskip  \noindent
Thus the proof of Lemma \ref{lem:X_n-tightness} is complete.
\end{proof}

\medskip
\subsection{Proof of Theorem \ref{th:mart-sol_existence}} \label{sec:main_th-proof}

\medskip  
\noindent
By Lemma \ref{lem:X_n-tightness} the set of measures $\bigl\{ \Law (\Xn  ) , n \in \mathbb{N}  \bigr\} $ is tight on the space $(\zcal , \sigma (\tcal ))$ defined by \eqref{eq:Z_cadlag}. 
Let ${\eta }_{n}:= \eta $, $n \in \mathbb{N} $. The set of measures $\bigl\{ \Law ({\eta }_{n} ) , n \in \mathbb{N}  \bigr\} $ is tight on the space  ${M}_{\bar{\mathbb{N} }}([0,T]\times Y)$.
Thus the set $ \bigl\{ \Law ({\eta }_{n}, \Xn ) , n \in \mathbb{N}  \bigr\}$ is tight on  $ {M}_{\bar{\mathbb{N} }}([0,T]\times Y)  \times \zcal $.
By Corollary \ref{cor:Skorokhod_J,B,H} and Remark \ref{rem:separating_maps}, see Appendix \ref{sec:Skorokhod_th}, there exists a subsequence $({n}_{k}{)}_{k\in \mathbb{N} }$, a probability space 
$\bigl( \bar{\Omega }, \bar{\mathcal{F} },\bar{\p }  \bigr) $ and, on this space, 
$ {M}_{\bar{\mathbb{N} }}([0,T]\times Y) \times \zcal  $-valued random variables $({\eta }_{*}, \Xast )$, 
$(  {\bar{\eta }}_{k}, {\bX}_{k})$, $k \in \mathbb{N} $ such that
\begin{itemize}
\item[(i) ] $\Law \bigl( ( {\bar{\eta }}_{k}, {\bX}_{k}) \bigr) = \Law \bigl( ( {\eta }_{{n}_{k}}, {\X }_{{n}_{k}}) \bigr) $ for all $ k \in \mathbb{N} $;
\item[(ii) ] $({ {\bar{\eta }}_{k}, \bX }_{k}) \to ({\eta }_{*}, \Xast )$ in $ {M}_{\bar{\mathbb{N} }}([0,T]\times Y) \times \zcal  $ with probability $1$ on $\bigl( \bar{\Omega }, \bar{\mathcal{F} },\bar{\p }  \bigr) $ as $k \to \infty $;
\item[(iii) ] $ {\bar{\eta }}_{k} (\bar{\omega }) ) = {\eta }_{*}(\bar{\omega }))$ for all $\bar{\omega } \in \bar{\Omega }$.
\end{itemize}
We will denote these sequences again by $\bigl(  ( {\eta }_{n}, \Xn )  {\bigr) }_{ n\in \mathbb{N} }$ and 
$\bigl(  ( {\bar{\eta }}_{n}, \bXn  )  {\bigr) }_{ n\in \mathbb{N} } $.
Moreover, ${\bar{\eta }}_{n}$, $n \in \mathbb{N} $, and ${\eta }_{*}$ are time homogeneous Poisson random measures on $(Y, \ycal )$ with the intensity measure $\mu $, see \cite[Section 9]{Brze+Haus+Raza'18}.
Using the definition of the space $\zcal $, see \eqref{eq:Z_cadlag}, we have
\begin{equation} \label{eq:conv-as_truncated} 
 \bXn \to \Xast  \quad \mbox{ in } \quad  {L}_{w}^{2}(0,T;\Vmath ) \cap 
{L}^{2}(0,T;{\Hmath }_{loc})  \cap \Dmath ([0,T];{\Umath }^{\prime})
  \cap \Dmath ([0,T];{\Hmath }_{w}) \quad \bar{\p } \mbox{-a.s.}
\end{equation}
In particular,
\begin{equation}
\Xast \; \in \; {L}^{2}(0,T;\Vmath ) \cap  \Dmath ([0,T];{\Hmath }_{w}) \cap  \Dmath ([0,T];{\Umath }^{\prime }) .
\label{eq:Xast_in_Z}
\end{equation}
Since the random variables $\bXn $ and $\Xn $ are identically distributed, using Lemma \ref{lem:Hall-MHD_truncated_estimates} we infer that $\bXn $ satisfy the following inequalities.
For every $q\in [2, \infty ) $  
\begin{equation} \label{eq:H_estimate_bar_Xn}
\sup_{n \ge 1 } \bar{\e } \Bigl[ \sup_{0 \le s \le T } \nnorm{\bXn (s)}{\Hmath }{q} \Bigr] 
\; \le \; {C}_{1}(q) .
\end{equation}
and
\begin{equation} \label{eq:V_estimate_bar_Xn}
\sup_{n \ge 1 }  \bar{\e } \Bigl[ \int_{0}^{T} \norm{ \bXn (s)}{\Vmath }{2} \, ds \Bigr] \; \le \; {C}_{2}.
\end{equation}
Let us fix $p \in [2,\infty )$.
We will show that for every $q \in [2,p]$
\begin{equation}
\bar{\e } \Bigl[ \sup_{0 \le s \le T } \nnorm{\Xast (s)}{\Hmath }{q} \Bigr] \; < \; \infty  .
\label{eq:H_estimate_Xast}
\end{equation}
By inequality \eqref{eq:H_estimate_bar_Xn} with $q:=p $ we can choose a further subsequence of
$(\bXn )$ convergent weak star in the space ${L}^{p }(\bar{\Omega };{L}^{\infty }(0,T;\Hmath ))$, and using \eqref{eq:conv-as_truncated}, deduce that 
\begin{equation*} 
\e \Bigl[ \sup_{t \in [0,T]} \nnorm{\Xast (t)}{\Hmath }{p } \Bigr] 
\; \le \; {C}_{1}(p )  .
\end{equation*} 
Let us fix $q \in [1,p )$. Notice that for every $t \in [0,T]$
\[
\nnorm{\Xast  (t)}{\Hmath }{q} \; = \;  {(\nnorm{\Xast  (t)}{\Hmath }{p  })}^{q/{p }}
\; \le \; {\Bigl( \sup_{t\in [0,T]} \nnorm{\Xast  (t)}{\Hmath }{p } \Bigr) }^{q/{p }}  .
\]
Thus, $ \sup_{t\in [0,T]}  \nnorm{\Xast  (t)}{\Hmath }{q}
\le {\Bigl( \sup_{t\in [0,T]} \nnorm{\Xast  (t)}{\Hmath }{p } \Bigr) }^{q/p }$,
and so by the H\"{o}lder inequality    
\begin{equation*}
\begin{split}
& \bar{\mathbb{E}} \Bigl[ \sup_{t\in [0,T]}  \nnorm{\Xast (t)}{\Hmath }{q}  \Bigr]
 \; \le \;  \bar{\mathbb{E}} \Bigl[ {\Bigl( \sup_{t\in [0,T]} \nnorm{\Xast (t)}{\Hmath }{p } \Bigr) }^{q/p }  \Bigr] \\
& \qquad  \le \;  {\biggl( \bar{\mathbb{E}} \Bigl[  \sup_{t\in [0,T]} \nnorm{\Xast (t)}{\Hmath }{p }   \Bigr]  \biggr) }^{q/p }
\;  \le \;  {C}_{1}(p,q),
 \end{split}
\end{equation*}   
where
${C}_{1}(p ,q):={\bigl( {C}_{1}(p ) \bigr) }^{q/p }$.
Thus \eqref{eq:H_estimate_Xast} holds.

\medskip  \noindent
By inequality \eqref{eq:V_estimate_bar_Xn}, we infer the sequence  $(\bXn )$ contains further subsequence, denoted again by $(\bXn )$, convergent weakly in the space ${L}^{2}([0,T]\times \bar{\Omega };\Vmath )$. Since by 
\eqref{eq:conv-as_truncated}, $\bXn \to \Xast $ in $\zcal $, we infer that 
$\Xast  \in {L}^{2}([0,T]\times \bar{\Omega };\Vmath )$, i.e. 
\begin{equation} \label{eq:V_estimate_Xast}
\bar{\e } \Bigl[ \int_{0}^{T} \norm{\Xast }{\Vmath }{2} \, ds  \Bigr] \;  < \; \infty .
\end{equation}  

\medskip  
\noindent
We will prove that the system $(\bar{\Omega },\bar{\mathcal{F}}, \bar{\mathbb{F}}, \bar{\p }, {\eta }_{\ast },\Xast )$ is a martingale solution  of problem \eqref{eq:Hall-MHD_functional}.

\medskip  
\noindent
\bf Step 1. \rm
Let us   fix $\varphi \in \Vast $. Analogously to  \cite{EM'14}, let us denote
\begin{equation}
\begin{split}
& {\Lambda }_{n}( \bXn ,{\bar{\eta }}_{n}, \varphi) (t)
\; := \;  \ilsk{ \bXn (0)}{\Pn \varphi}{\Hmath  }
  - \int_{0}^{t} \dual{  \mathcal{A} \bXn (s)}{\Pn \varphi}{}  ds
  \\
& - \; \int_{0}^{t} \dual{\MHD (\bXn (s))}{\Pn \varphi}{}  ds 
- \; \int_{0}^{t} \dual{\tHall (\bXn (s))}{\Pn \varphi}{}  ds 
\\
 & + \; \int_{0}^{t} \int_{Y} \ilsk{F(s,\bXn ({s}^{-});y) }{\Pn \varphi }{\Hmath } 
\,  {\tilde{\bar{\eta }}}_{n}(ds, dy ) , \quad t \in [0,T] ,
\end{split} \label{eq:Lambda_n_un_truncated}
\end{equation}
and 
\begin{equation}
\begin{split}
&  \Lambda  (\Xast  , \bar{\eta }, \varphi) (t)
\; := \;   \ilsk{\Xast (0)}{\varphi}{{\Hmath }}
    - \int_{0}^{t} \dual{ \mathcal{A} \Xast (s)}{\varphi}{}  ds 
\\ 
&- \int_{0}^{t} \dual{ \MHD (\Xast (s))}{\varphi}{}  ds 
 - \int_{0}^{t} \dual{ \tHall (\Xast (s))}{\varphi}{}  ds 
 \\
&  +  \; \int_{0}^{t} \int_{Y} \ilsk{F(s,\Xast  ({s}^{-});y) }{ \varphi }{\Hmath } 
\,  {\tilde{\eta }}_{\ast }(ds, dy ) ,
\quad t \in [0,T] .
\end{split} \label{eq:Lambda_Hall-MHD}
\end{equation}  

\medskip  \noindent
In the following  lemma we will prove that each term on the r.h.s. of \eqref{eq:Lambda_n_un_truncated} is convergent to the corresponding term on the r.h.s of \eqref{eq:Lambda_Hall-MHD}.

\medskip
\begin{lemma}  \label{lem:convergence_existence_truncated}
We have
\begin{description}
\item[(a)] for every $\varphi \in \Hmath $
$$\lim_{n\to \infty } \bar{\e } \Bigl[ \int_{0}^{T} 
{|\ilsk{\bXn (t)}{\Pn \varphi}{{\Hmath } } - \ilsk{\Xast (t)}{\varphi}{\Hmath  }|}^{2} \, dt \Bigr] =0 ,$$
\item[(b)] for every $\varphi \in \Hmath  $ 
$$\lim_{n\to \infty } \bar{\e } \bigl[ {|\ilsk{\bXn (0)}{\Pn \varphi}{\Hmath } 
- \ilsk{\Xast (0)}{\varphi}{\Hmath }|}^{2}  \bigr] \; = \; 0 ,$$
\item[(c)] for every $\varphi \in \Vmath  $ 
$$ \lim_{n\to \infty } \bar{\e } \Bigl[ \int_{0}^{T}
 \Bigl| \int_{0}^{t} (\dual{\mathcal{A} \bXn (s)}{\Pn \varphi}{}
- \dual{ \mathcal{A} \Xast (s)}{\varphi}{} ) \,ds \Bigr| \, dt \Bigr] \; = \; 0 ,$$
\item[(d)] for every $\varphi \in \Vast $
 $$\lim_{n\to \infty } \bar{\e } \Bigl[ \int_{0}^{T}
 \Bigl| \int_{0}^{t} (\dual{ \MHD (\bXn (s))}{\Pn \varphi}{}
- \dual{ \MHD (\Xast (s))}{\varphi}{}) \,ds \Bigr| \, dt \Bigr] \; = \; 0 ,$$
\item[(e)] for every $\varphi \in \Vast $
$$ \lim_{n\to \infty } \bar{\e } \Bigl[ \int_{0}^{T}
 \Bigl| \int_{0}^{t} (\dual{ \tHall (\bXn (s))}{\Pn \varphi}{}
- \dual{ \tHall (\Xast (s))}{\varphi}{} ) \,ds \Bigr| \, dt \Bigr] \; = \; 0 ,$$
\item[(f)] for every $\varphi \in \Hmath $
\[
\begin{split} 
&\lim_{n\to \infty } \bar{\e } \Bigl[ \int_{0}^{T} 
\Bigl|  \int_{0}^{t}  \int_{Y}  
\dual{ \Pn F(s, \bXn ({s}^{-});y)}{\varphi }{}  \, {\tilde{\bar{{\eta }}}}_{n} (ds,dy)
\\
&\qquad \qquad   - \int_{0}^{t} \int_{Y}  
\dual{ F(s,{\X }_{\ast }({s}^{-});y)}{\varphi}{}  \, {\tilde{{\eta }}}_{\ast } (ds,dy) {\Bigr| }^{2}   
 \, dt \Bigr] \; = \; 0 .
\end{split}
\]  
\end{description}
\end{lemma}

\medskip
\begin{proof}
\bf Ad. (a). \rm Let us fix $\varphi \in \Hmath  $. Since 
\[
\ilsk{\bXn (t)}{\Pn \varphi}{\Hmath } = \ilsk{\Pn \bXn (t)}{ \varphi}{\Hmath }  =\ilsk{\bXn (t)}{ \varphi}{\Hmath } ,
\]
and by \eqref{eq:conv-as_truncated} 
$ \bXn \to \Xast  $ in $\Dmath ([0,T];{\Hmath }_{w})$, $\bar{\p}$-a.s., 
we infer that 
$\ilsk{\bXn (\cdot )}{\varphi}{\Hmath } \to \ilsk{\Xast  (\cdot )}{\varphi}{\Hmath } $
in $\Dmath ([0,T];\mathbb{R} )$, $\bar{\p }$-a.s.. Hence, in particular, for almost all $t\in [0,T]$
\[
\lim_{n\to \infty } \ilsk{\bXn (t)}{ \varphi}{\Hmath } \; = \; \ilsk{\Xast (t)}{\varphi}{\Hmath  } ,
\qquad \bar{\p} -a.s..
\]
Since by \eqref{eq:H_estimate_bar_Xn} and \eqref{eq:H_estimate_Xast}, 
 $\sup_{t\in [0,T]} \nnorm{\bXn (t)}{\Hmath }{2} < \infty $ and $\sup_{t\in [0,T]} \nnorm{\Xast (t)}{\Hmath }{2} < \infty $, $\bar{\p }$-a.s.,
using the dominated convergence theorem we infer that 
\begin{equation} \label{eq:Vitali_{u}_{n}_conv}
\lim_{n \to \infty } \int_{0}^{T} {| \ilsk{\bXn (t) -\Xast (t)}{\varphi}{\Hmath } |}^{2} \, dt \; = \; 0  \qquad  \mbox{ $\bar{\p }$-a.s. }.
\end{equation}
Moreover, by the H\"{o}lder inequality, 
for every $n \in \mathbb{N}$ and every $r\in \bigl( 1,  \frac{p}{2}\bigr] $
\begin{equation*}
\begin{split}
& \bar{\e }\Bigl[ \Bigl|  \int_{0}^{T} \nnorm{ \bXn (t) -\Xast (t) }{\Hmath }{2} \, dt  {\Bigr| }^{r}\Bigr] 
\; \le \;  {T}^{r-1} \, {2}^{2r-1}
 \bar{\e }\Bigl[  \int_{0}^{T} \bigl( \nnorm{\bXn (t)}{\Hmath }{2r}
+ \nnorm{\Xast (t)}{\Hmath }{2r} \bigr) \, dt \Bigr]  \\
&  \le \; {T}^{r-1} \, {2}^{2r-1} \, T \, 
\bar{\e }\Bigl[ \sup_{t\in [0,T]} \nnorm{ \bXn (t)}{\Hmath }{2r} + \sup_{t\in [0,T]} \nnorm{\Xast (t)}{\Hmath }{2r} \Bigr] .
\end{split}  
\end{equation*} 
Thus by \eqref{eq:H_estimate_bar_Xn} and \eqref{eq:H_estimate_Xast}, we infer that for every  
$r\in \bigl( 1,  \frac{p}{2}\bigr] $
\begin{equation}
\begin{split}
& \sup_{n \in \mathbb{N} }\bar{\e }\Bigl[ \Bigl|  \int_{0}^{T} \nnorm{ \bXn (t) -\Xast (t) }{\Hmath }{2} \, dt  {\Bigr| }^{r}\Bigr]  \; < \; \infty .
\end{split}
 \label{eq:Vitali_{u}_{n}_est}
\end{equation} 
Finally, by \eqref{eq:Vitali_{u}_{n}_conv}, \eqref{eq:Vitali_{u}_{n}_est} and the Vitali theorem we infer that
\begin{equation*}
\lim_{n\to \infty } \bar{\e } 
\Bigl[ \int_{0}^{T} {| \ilsk{\bXn (t) -\Xast (t)}{\varphi }{\Hmath } | }^{2} \, dt \Bigr]  \; = \; 0.
\end{equation*}
The proof of assertion (a) is thus complete.

\medskip  \noindent
\bf Ad. (b). \rm Let us fix $\varphi \in \Hmath $.
Since by \eqref{eq:conv-as_truncated} $\bXn  \to \Xast $ in $\Dmath (0,T;{\Hmath }_{w})$ $\bar{\p }$-a.s. and $\Xast $ is right-continuous at $t=0$, we infer that
$ \ilsk{\bXn (0)}{\varphi}{\Hmath } \to \ilsk{\Xast (0)}{\varphi}{\Hmath } $, $\bar{\p }$-a.s.
By \eqref{eq:H_estimate_bar_Xn}, \eqref{eq:H_estimate_Xast}  and the Vitali theorem, we have
\[
\lim_{n\to \infty } \bar{\e } \bigl[ {|\ilsk{\bXn (0)-\Xast (0)}{\varphi}{\Hmath } |}^{2} \bigr] \; = \; 0,
\]
which completes the proof of (b).

\medskip
\noindent
\bf Ad. (c). \rm  Let us fix $\varphi \in \Vmath $.
By \eqref{eq:A_acal_rel}, we have
\[
\int_{0}^{t} \dual{ \mathcal{A} \bXn (\sigma )}{\Pn \varphi}{} \, d\sigma  
\; = \;  \int_{0}^{t} \dirilsk{\bXn (\sigma )}{\Pn \varphi}{} \, d\sigma  .
\]
By \eqref{eq:conv-as_truncated}, $\bXn  \to \Xast $ in ${L}_{w}^{2}(0,T;\Vmath )$, 
$\bar{\p }$-a.s.
Moreover, since $\varphi \in \Vmath  $, we infer that $\Pn \varphi \to \varphi $ in $\Vmath $, 
(see  Corollary \ref{cor:P_n-pointwise_conv}).
Thus 
\begin{equation} 
\begin{split}
& \lim_{n \to \infty } \int_{0}^{t} \dirilsk{\bXn (\sigma )}{\Pn \varphi}{} \, d\sigma   
\; = \; \int_{0}^{t} \dirilsk{ \Xast (\sigma )}{\varphi}{} \, d\sigma  
\; = \; \int_{0}^{t} \dual{\mathcal{A} \Xast (\sigma )}{\varphi}{} \, d\sigma , \qquad \mbox{$\bar{\p }$-a.s.} 
\end{split} \label{eq:Vitali_{A}_{n}_conv}
\end{equation}
Using \eqref{eq:A_acal_rel}, the H\"{o}lder inequality  we obtain the following inequality
 for all $t \in [0,T]$ and  $n \in \mathbb{N} $
\begin{equation*}
\begin{split} 
& \bar{\e }\Bigl[ \Bigl | \int_{0}^{t}  \dual{  \mathcal{A} \bXn (s)}{\Pn \varphi}{} \, ds {\Bigr| }^{2} \Bigr]
\\
&
 \le \;  \norm{\Pn \varphi }{}{2} \cdot \bar{\e } \Bigl[ \Bigl( \int_{0}^{t} \norm{\bXn }{}{}\, ds {\Bigr) }^{2} \Bigr] 
\; \le \;  \nnorm{\Pn }{\mathcal{L} (\Vmath ,\Vmath )}{2} \, \norm{ \varphi}{}{2}\, T\cdot  \bar{\e } 
 \Bigl[  \int_{0}^{T} \norm{ \bXn (s)}{\Vmath }{2} \, ds  \Bigr] .
\end{split} 
\end{equation*}
Thus by  \eqref{eq:V_estimate_bar_Xn} we infer that for all $t\in [0,T]$
\begin{equation}
\begin{split} 
& \sup_{n \in \mathbb{N} } \, \bar{\e }\Bigl[ \Bigl | \int_{0}^{t}  \dual{  \mathcal{A} \bXn (s)}{\Pn \varphi}{} \, ds {\Bigr| }^{2} \Bigr]
\; \le \;  T \, \sup_{n\in \mathbb{N} } \nnorm{\Pn }{\mathcal{L} (\Vmath ,\Vmath )}{2} \, \norm{ \varphi}{\Vmath }{2} \cdot {C}_{2} 
\; < \;  \infty .
\end{split} \label{eq:Vitali_{A}_{n}_est}
\end{equation}
Therefore by \eqref{eq:Vitali_{A}_{n}_conv}, \eqref{eq:Vitali_{A}_{n}_est}, \eqref{eq:V_estimate_Xast} and the Vitali theorem we infer that for all $t \in [0,T]$
\[
\lim_{n\to \infty } \bar{\e } \Bigl[ \Bigl| \int_{0}^{t} [
\dual{ \mathcal{A} \bXn (s) }{\Pn \varphi}{} -  \dual{\mathcal{A} \Xast (s)}{\varphi}{}] \, ds 
 \Bigr|  \Bigr] \; = \; 0.
\]
Hence by  the dominated convergence theorem and the Fubini theorem
\begin{equation} \label{eq:{A}_{n}_conv}
\lim_{n\to \infty } \bar{\e } \Bigl[ \int_{0}^{T} \Bigl| 
\int_{0}^{t} (\dual{\mathcal{A} \bXn (s)}{\Pn \varphi}{}
- \dual{ \mathcal{A} \Xast (s)}{\varphi}{} ) \,ds \Bigr| \, dt \Bigr] \; = \;0 .
\end{equation}
The proof of assertion (c) is thus complete.

\medskip  \noindent
\bf Ad. (d). \rm  
Let us move to the $\MHD $-term. Let us fix $\varphi \in \Vast $.
For every $s\in [0,T]$ we have
\[
\begin{split}
& \dual{ \MHD (\bXn (s))}{\Pn \varphi}{} - \dual{ \MHD (\Xast (s))}{\varphi}{} \\
\; &= \; \dual{ \MHD (\bXn (s))}{\Pn \varphi - \varphi }{} + \dual{ \MHD (\bXn (s)) - \MHD(\Xast (s))}{ \varphi}{}.
\end{split}
\]
By \eqref{eq:conv-as_truncated}, the sequence $(\bXn )$ is $\bar{\p }$-a.s.,  convergent  to $\Xast $ in ${L}^{2}_{w}(0,T;\Vmath )$. Thus  $(\bXn )$ is bounded in ${L}^{2}(0,T;\Vmath )$,  and in particular, by \eqref{eq:Hall-MHD_V-norm}, it is bounded in ${L}^{2}(0,T;\Hmath )$, as well.
Moreover, by \eqref{eq:conv-as_truncated}, $\bXn  \to \Xast $ in 
${L}^{2}(0,T;{\Hmath }_{loc}) $, $\bar{\p }$-a.s.
By Corollary \ref{cor:MHD-map_conv-aux} we infer that $\bar{\p }$-a.s.
for all $t \in [0,T]$
\[
\lim_{n\to \infty } \int_{0}^{t} \dual{ \MHD (\bXn (s)) - \MHD (\Xast (s))}{\varphi}{} \, ds \; = \; 0 .
\]
By \eqref{eq:MHD-map_est-H-H},
we have
\[
\begin{split}
&\Bigl| \int_{0}^{t} \ddual{\Vastprime }{ \MHD (\bXn (s))}{\Pn \varphi - \varphi }{\Vast } \, ds \Bigr|
\; \le \;  \int_{0}^{t} |\ddual{\Vastprime }{ \MHD (\bXn (s))}{\Pn \varphi - \varphi }{\Vast } | \, ds \\
\; &\le \; \norm{\MHD }{}{} \cdot \int_{0}^{t} \nnorm{\bXn (s)}{\Hmath }{2}  \, ds 
\cdot \norm{\Pn \varphi - \varphi }{\Vast }{} 
\; \le \; \norm{\MHD }{}{} \cdot \norm{\bXn }{{L}^{2}(0,T;\Hmath )}{2} \cdot \norm{\Pn \varphi - \varphi }{\Vast }{} .
\end{split}
\]
Since $\varphi \in \Vast $ then by Corollary \ref{cor:P_n-pointwise_conv} (ii), $\Pn \varphi \to \varphi $ in $\Vast$, and by the boundedness of the sequence $(\bXn )$ in ${L}^{2}(0,T;\Hmath )$,
we infer that $\bar{\p }$-a.s.
for all $t \in [0,T]$
\[
\lim_{n\to \infty } \int_{0}^{t} \ddual{\Vastprime }{ \MHD (\bXn (s))}{\Pn \varphi - \varphi }{\Vast } \, ds 
\; = \; 0.
\]
Thus for  all $t \in [0,T]$
\begin{equation} 
\label{eq:Vitali_{B}_{n}_conv}
\begin{split}
& \lim_{n\to \infty } \int_{0}^{t} (\dual{ \MHD( \bXn (s))}{\Pn \varphi}{} - \dual{\MHD (\Xast (s))}{\varphi }{} ) \, ds  \; = \; 0 ,
\qquad \mbox{$\bar{\p }$-a.s.} 
\end{split}
\end{equation}
By \eqref{eq:MHD-map_est-H-H} we obtain for all $r\in \bigl[1,\frac{p }{2}\bigr] $,
 $t \in [0,T]$   and  $n \in \mathbb{N} $
\begin{equation*} 
\begin{split}
&  \bar{\e } \Bigl[ \Bigl| \int_{0}^{t} \dual{ \MHD (\bXn (s)) }{\Pn \varphi}{} \, ds {\Bigr| }^{r}  \Bigr] 
\; \le \; 
\norm{\MHD }{}{r} \, \norm{\Pn \varphi}{\Vast}{r} \, 
 \bar{\e }  \Bigl[ \Bigl( \int_{0}^{t} \nnorm{\bXn (s)}{\Hmath }{2}  \, ds {\Bigr) }^{r}  \Bigr] 
 \\
\; &\le \; \norm{\MHD }{}{r} \,  \nnorm{\Pn }{\mathcal{L} (\Vast ,\Vast )}{r} \norm{ \varphi}{\Vast}{r} \, \, {T}^{r} \,  \bar{\e } \Bigl[ \sup_{s\in [0,T]} \nnorm{\bXn (s)}{\Hmath }{2r}   \Bigr] 
 .  
\end{split}
\end{equation*}
Thus by \eqref{eq:H_estimate_bar_Xn} we infer that for all $r\in \bigl[1,\frac{p }{2}\bigr] $,
and  $t \in [0,T]$   
\begin{equation} 
\begin{split}
& \sup_{n \in \mathbb{N} } \bar{\e } \Bigl[ \Bigl| \int_{0}^{t} \dual{ \MHD (\bXn (s)) }{\Pn \varphi}{} \, ds {\Bigr| }^{r}  \Bigr] 
\;  \le \; \norm{\MHD }{}{r} \, \sup_{n\in \mathbb{N} } \nnorm{\Pn }{\mathcal{L} (\Vast ,\Vast )}{r} \norm{ \varphi}{\Vast}{r} \, \, {T}^{r} \, {C}_{1}(2r) < \infty .  
\end{split} \label{eq:Vitali_{B}_{n}_est}
\end{equation}
In view of \eqref{eq:Vitali_{B}_{n}_conv} and \eqref{eq:Vitali_{B}_{n}_est} (with $r\in (1,\frac{p}{2}]$), by the Vitali theorem we 
obtain for all $t\in [0,T]$
\begin{equation} \label{eq:Lebesque_{B}_{n}_conv}
\lim_{n \to \infty } \bar{\e } \Bigl[ \Bigl| 
\int_{0}^{t} (\dual{ \MHD( \bXn (s))}{\Pn \varphi}{} - \dual{\MHD (\Xast (s))}{\varphi }{} ) \, ds 
 {\Bigr| }^{} \Bigr] \; = \; 0 .
\end{equation}
By \eqref{eq:Lebesque_{B}_{n}_conv}, \eqref{eq:Vitali_{B}_{n}_est} (with $r=1$), the dominated convergence theorem and the Fubini theorem, we infer that 
\begin{equation} \label{eq:{B}_{n}_conv}
\lim_{n\to \infty } \bar{\e } \Bigl[ \int_{0}^{T}
\Bigl| \int_{0}^{t} (\dual{ \MHD (\bXn (s))}{\Pn \varphi}{}
- \dual{ \MHD (\Xast (s))}{\varphi}{} )\,ds \Bigr| \, dt \Bigr]  \; = \; 0 .
\end{equation}
The proof of (d) is complete.

\medskip  
\noindent
\bf Ad. (e). \rm 
Let us move to the $\tHall $-term. Let us fix $\varphi \in \Vast $.
For every $s\in [0,T]$ we have
\[
\begin{split}
& \dual{ \tHall (\bXn (s))}{\Pn \varphi}{} - \dual{ \tHall (\Xast (s))}{\varphi}{} \\
\; &= \; \dual{ \tHall (\bXn (s))}{\Pn \varphi - \varphi }{} + \dual{ \tHall (\bXn (s)) - \tHall (\Xast (s))}{ \varphi}{}.
\end{split}
\]
By \eqref{eq:conv-as_truncated}, the sequence $(\bXn )$ is $\bar{\p }$-a.s. convergent to $\Xast $ in ${L}_{w}^{2}(0,T;\Vmath ) \cap {L}^{2}(0,T;{\Hmath }_{loc}) $.
By Corollary \ref{cor:tHall-term_conv_general} 
 we infer that $\bar{\p }$-a.s.
for all $t \in [0,T]$
\[
\lim_{n\to \infty } \int_{0}^{t} \dual{ \tHall (\bXn (s)) - \tHall (\Xast (s))}{\varphi}{} \, ds \; = \; 0 .
\]
By \eqref{eq:tHall-map_est-H-V},
we have
\[
\begin{split}
&\Bigl| \int_{0}^{t} \ddual{\Vastprime }{ \tHall (\bXn (s))}{\Pn \varphi - \varphi }{\Vast } \, ds \Bigr|
\; \le \;  \int_{0}^{t} |\ddual{\Vastprime }{ \tHall (\bXn (s))}{\Pn \varphi - \varphi }{\Vast } | \, ds \\
\; &\le \; \norm{\tHall }{}{} \cdot \int_{0}^{t} \nnorm{\bXn (s)}{\Hmath }{} \, \nnorm{\bXn (s)}{\Vmath }{}  \, ds 
\cdot \norm{\Pn \varphi - \varphi }{\Vast }{} \\
\; &\le \; \norm{\tHall }{}{} \cdot \norm{\bXn }{{L}^{2}(0,T;\Hmath )}{} \, \norm{\bXn }{{L}^{2}(0,T;\Vmath )}{}
\cdot \norm{\Pn \varphi - \varphi }{\Vast }{} .
\end{split}
\]
Since $\varphi \in \Vast $ then by Corollary \ref{cor:P_n-pointwise_conv}  (ii), $\Pn \varphi \to \varphi $ in $\Vast$. By the boundedness of the sequence $(\bXn )$ in ${L}^{2}(0,T;\Vmath )$ (and hence in ${L}^{2}(0,T;\Hmath )$, as well),
we infer that $\bar{\p }$-a.s.
for all $t \in [0,T]$
\[
\lim_{n\to \infty } \int_{0}^{t} \ddual{\Vastprime }{ \tHall (\bXn (s))}{\Pn \varphi - \varphi }{\Vast } \, ds 
\; = \; 0.
\]
Thus for  all $t \in [0,T]$
\begin{equation} 
\label{eq:Vitali_{tHall}_{n}_conv}
\begin{split}
& \lim_{n\to \infty } \int_{0}^{t} (\dual{ \tHall (\bXn (s))}{\Pn \varphi}{} - \dual{\tHall (\Xast (s))}{\varphi }{} ) \, ds  \; = \; 0 ,
\qquad \mbox{$\bar{\p }$-a.s.} 
\end{split}
\end{equation}
We will show that for $r \in [1,\frac{2p}{p +2}] $
\begin{equation}
\begin{split}
& \sup_{n \in \mathbb{N} } \sup_{t\in [0,T]}
 \bar{\e } \Bigl[ \Bigl| \int_{0}^{t}\ddual{\Vastprime }{\tHall (\bXn (s))}{\Pn \varphi }{\Vast } \, ds {\Bigr| }^{r} \Bigr] \; < \; \infty .
\end{split}
\label{eq:Vitali_{tHall}_{n}_est}
\end{equation}
Indeed, by \eqref{eq:tHall-map_est-H-V}, 
we have for each $n \in \mathbb{N} $ and $t \in [0,T]$
\[
\begin{split}
&\Bigl| \int_{0}^{t}\ddual{\Vastprime }{\tHall (\bXn (s))}{\Pn \varphi }{\Vast } \, ds \Bigr|
\; \le \; \int_{0}^{t}|\ddual{\Vastprime }{\tHall (\bXn (s))}{\Pn \varphi }{\Vast } | \, ds
\\
\; &\le \; \norm{\tHall }{}{} \, \sup_{n\in \mathbb{N} } \nnorm{\Pn }{\mathcal{L} (\Vast ,\Vast )}{} \, \norm{\varphi }{\Vast }{} \int_{0}^{t}\nnorm{\bXn (s)}{\Hmath }{} \norm{\bXn (s)}{\Vmath }{} \, ds.
\end{split}
\]
Let $r \in [1,\frac{2p }{p+2}] $.
By the H\"{o}lder inequality
and the Young inequality (for numbers with $\frac{1}{\alpha }=  1-\frac{r}{2}$ and $\frac{1}{\beta } = \frac{r}{2}$), and estimates \eqref{eq:H_estimate_bar_Xn}  and \eqref{eq:V_estimate_bar_Xn} 
we have for each $n \in \mathbb{N} $ and $t \in [0,T]$
\[
\begin{split}
\bar{\e } \Bigl[ \Bigl( \int_{0}^{t} \nnorm{\bXn (s)}{\Hmath }{} \norm{\bXn (s)}{\Vmath }{} \, ds {\Bigr) }^{r} \Bigr]
\; &\le \; 
\frac{1}{\alpha } \, {T}^{r} \,  \bar{\e } \Bigl[ \sup_{s\in [0,T]}\nnorm{\bXn (s)}{\Hmath }{\frac{2r}{2-r}} \Bigr]
+ \frac{1}{\beta } \, {T}^{r-1} \, \bar{\e } \Bigl[ \int_{0}^{T} \norm{\bXn (s)}{\Vmath }{2} \, ds \Bigr]
\\
\; &\le \; \frac{1}{\alpha } \, {T}^{r} \, C_1 \Bigl(\frac{2r}{2-r} \Bigr)
+ \frac{1}{\beta } \, {T}^{r-1} \, {C}_{2}.
\end{split}
\]
Since moreover by Corollary \ref{cor:P_n-pointwise_conv}  (ii), $\sup_{n\in \mathbb{N} } \nnorm{\Pn }{\mathcal{L} (\Vast ,\Vast )}{} <\infty $,
 the proof of \eqref{eq:Vitali_{tHall}_{n}_est} is complete.

\medskip
\noindent
In view of \eqref{eq:Vitali_{tHall}_{n}_conv} and \eqref{eq:Vitali_{tHall}_{n}_est} 
(with $r\in (1,2)$), by the Vitali theorem we 
obtain for all $t\in [0,T]$
\begin{equation} \label{eq:Lebesque_{tHall}_{n}_conv}
\lim_{n \to \infty } \bar{\e } \Bigl[ \Bigl| 
\int_{0}^{t} (\dual{ \tHall ( \bXn (s))}{\Pn \varphi}{} - \dual{\tHall (\Xast (s))}{\varphi }{} ) \, ds 
 {\Bigr| }^{} \Bigr] \; = \; 0 .
\end{equation}
By \eqref{eq:Lebesque_{tHall}_{n}_conv}, \eqref{eq:Vitali_{tHall}_{n}_est} (for $r=1$) the dominated convergence theorem and the Fubini theorem, we infer that 
\begin{equation} \label{eq:{tHall}_{n}_conv}
\lim_{n\to \infty } \bar{\e } \Bigl[ \int_{0}^{T}
\Bigl| \int_{0}^{t} (\dual{ \tHall (\bXn (s))}{\Pn \varphi}{}
- \dual{ \tHall (\Xast (s))}{\varphi}{} )\,ds \Bigr| \, dt \Bigr]  \; = \; 0 .
\end{equation}
The proof of (e) is complete.

\medskip
\noindent
\bf Ad. (f). \rm 
Let us fix $\varphi \in \Hmath $.
By \eqref{eq:F_n},  inequality \eqref{eq:F_linear_growth} in Remark \ref{rem:F_properties}, 
\eqref{eq:H_estimate_bar_Xn} and \eqref{eq:H_estimate_Xast}
for every $t \in [0,T]$  and  $n \in \mathbb{N} $ 
\begin{equation*}
\begin{split}
&\bar{\e } \Bigl[ \int_{0}^{t} \int_{Y} 
{|\ilsk{{F}_{n}(s, \bXn ({s}^{-});y)}{\varphi}{\Hmath }|}^{2} \, d\mu (y)ds  \Bigr]
\; \le \; \nnorm{\varphi }{\Hmath }{2} \bar{\e } \Bigl[ \int_{0}^{t} \int_{Y} 
\nnorm{F(s, \bXn ({s}^{-});y)}{\Hmath }{2} \, d\mu (y)ds  \Bigr] \\
\; &\le \; \nnorm{\varphi }{\Hmath }{2}{K}_{2} \bar{\e } \Bigl[   \int_{0}^{t} 
\bigl\{ 1+ \nnorm{\bXn (s)}{\Hmath }{2} \bigr\}  \, ds  \Bigr] 
\; \le \; \nnorm{\varphi }{\Hmath }{2}T{K}_{2} \bigl( 1+ \bar{\e } \bigl[ \sup_{s \in [0,T]} \nnorm{\bXn (s)}{\Hmath }{2}\bigr] \bigr) 
  < \infty ,
\end{split}
\end{equation*}
and
\begin{equation*}
\begin{split}
\bar{\e } \Bigl[  \int_{0}^{t} \int_{Y} 
{|\ilsk{F(s, \Xast  ({s}^{-});y)}{\varphi}{\Hmath }|}^{2} \, d\mu (y)ds  \Bigr] 
\; &\le \; \nnorm{\varphi }{\Hmath }{2} T{K}_{2} \bigl( 1+ \bar{\e } \bigl[ \sup_{s \in [0,T]} \nnorm{\Xast  (s)}{\Hmath }{2}\bigr] \bigr)  < \infty .
\end{split}
\end{equation*}
By the isometry formula \eqref{eq:isometry}
and the fact that ${\bar{\eta }}_{n} = {\eta }_{\ast }$, we have
\begin{equation}
\begin{split}
&\bar{\e } \Bigl[ \Bigl|  \int_{0}^{t}  \int_{Y}  
\dual{ F_n(s, \bXn ({s}^{-});y)}{\varphi }{}  \, {\tilde{\bar{{\eta }}}}_{n} (ds,dy)
- \int_{0}^{t} \int_{Y}  
\dual{ F(s,\Xast ({s}^{-});y)}{\varphi}{}  \, {\tilde{{\eta }}}_{\ast } (ds,dy) {\Bigr| }^{2}  \Bigr] 
\\
\; &= \; \bar{\e } \Bigl[ \Bigl|  \int_{0}^{t} \int_{Y}  
\dual{ F_n(s, \bXn ({s}^{-});y)- F(s,\Xast ({s}^{-});y)}{\varphi}{}  \, {\tilde{{\eta }}}_{\ast } (ds,dy) {\Bigr| }^{2}  \Bigr] 
\\
\; &= \; \bar{\e } \Bigl[ \int_{0}^{t} \int_{Y} 
\Bigl| \dual{F_n (s,\bXn({s}^{-});y) -F(s,\Xast ({s}^{-});y)}{\varphi}{} {\Bigr| }^{2} \, d\mu (y)  ds  \Bigr] .
\end{split}
\label{eq:Lebesgue_{noise}_{n}_equal_H} 
\end{equation}

\medskip  \noindent
We will prove that  for every $t \in [0,T]$ and $\varphi \in \Hmath $
\begin{equation}
\lim_{n\to \infty } \bar{\e } \Bigl[ \int_{0}^{t} \int_{Y} 
\Bigl| \dual{F_n(s,\bXn({s}^{-});y) -F(s,{\X }_{\ast }({s}^{-});y)}{\varphi}{} {\Bigr| }^{2} \, d\mu (y) ds  \Bigr] \; = \; 0 .
\label{eq:Lebesgue_{noise}_{n}_conv_H} 
\end{equation}

\medskip 
\noindent
Indeed, for all $t \in [0,T]$ we have
\begin{equation*}
\begin{split}
&  \int_{0}^{t} \int_{Y} {| \dual{F(s, \bXn ({s}^{-});y)- F(s,\Xast({s}^{-});y)}{\varphi}{} |}^{2} \, d\mu (y)ds   \\
\; & \le \; \int_{0}^{T} \int_{Y} {|\ilsk{F(s, \bXn ({s}^{-});y)- F(s,\Xast({s}^{-});y)}{\varphi}{\Hmath} | }^{2} \, d\mu (y)ds   \\
\; & = \; \int_{0}^{T} \int_{Y} {| {\tilde{F}}_{\varphi}(\bXn )(s,y) - {\tilde{F}}_{\varphi}(\Xast )(s,y)| }^{2} \, d\mu (y)ds   \\
\; & = \; \norm{{\tilde{F}}_{\varphi}(\bXn ) - {\tilde{F}}_{\varphi}(\Xast )}{{L}^{2}([0,T]\times Y;\mathbb{R} )}{2}, 
\end{split}
\end{equation*}
where ${\tilde{F}}_{\varphi}$ is the mapping defined by \eqref{eq:F**}.  
Since $\bar{\p }$-a.s., $\bXn \in {L}^{2}(0,T;\Hmath )$ for $n \in \mathbb{N} $  and,  by \eqref{eq:conv-as_truncated},
$\bXn  \to \Xast $ in ${L}^{2}(0,T;{\Hmath }_{loc}) $, 
using  Remark \ref{rem:F_properties} (ii) 
we infer that for all $t\in [0,T]$ 
\begin{equation}
\lim_{n \to \infty } \int_{0}^{t} \int_{Y} 
{| \ilsk{F(s, \bXn ({s}^{-});y)- F(s,\Xast ({s}^{-});y)}{\varphi}{\Hmath } |}^{2} 
\, d\mu (y)ds  \; = \; 0. 
\label{eq:Vitali_{noise}_{n}_conv}
\end{equation}
Moreover, by inequality \eqref{eq:F_linear_growth} in Remark \ref{rem:F_properties} and by \eqref{eq:H_estimate_bar_Xn},
for every $t \in [0,T]$ every   $r \in [ 1, \frac{p }{2} ] $ 
and every $n \in \mathbb{N} $ the following inequalities hold
\begin{align} 
&  \bar{\e } \Bigl[ \Bigl| \int_{0}^{t} \int_{Y} {| 
\ilsk{F(s, \bXn ({s}^{-});y)- F(s,{\X }_{\ast }({s}^{-});y)}{\varphi}{\Hmath }|}^{2} 
 \, d\mu (y)ds {\Bigr| }^{r} \Bigr] \nonumber \\
\; &\le \; {2}^{r} \nnorm{\varphi}{\Hmath }{2r}\bar{\e } \Bigl[ \Bigl| \int_{0}^{t} \int_{Y} \bigl\{ 
 \nnorm{F(s, \bXn ({s}^{-});y) }{\Hmath }{2}
 +\nnorm{F(s,{\X }_{\ast }({s}^{-});y)}{\Hmath }{2} \bigr\} \, d\mu (y)ds {\Bigr| }^{r} \Bigr] \nonumber \\
\; &\le \; {2}^{r} {K}_{2}^{r} \nnorm{\varphi }{\Hmath }{2r}\, \bar{\e } \Bigl[ \Bigl|  \int_{0}^{t} \bigl\{ 2+ \nnorm{\bXn (s)}{\Hmath }{2} + \nnorm{{\X }_{\ast }(s)}{\Hmath }{2}
\bigr\}  \, ds {\Bigr| }^{r} \Bigr] 
\; \le \; c \bigl( 1+ \bar{\e } \bigl[ \sup_{s \in [0,T]}  \nnorm{\bXn (s)}{\Hmath }{2r}\bigr] \bigr)
 \nonumber \\
\; &\le \; c (1+ {C}_{1}(2r))  
\label{eq:Vitali_{noise}_{n}_est}
\end{align} 
for some constant $c>0$.
Thus by \eqref{eq:Vitali_{noise}_{n}_conv}, \eqref{eq:Vitali_{noise}_{n}_est} 
(for $r \in ( 1, \frac{p }{2} ] $)
 and the Vitali theorem for all $t\in [0,T]$
\begin{equation}
\lim_{n\to \infty } \bar{\e } \Bigl[  \int_{0}^{t} \int_{Y} 
{| \dual{F(s, \bXn ({s}^{-});y)- F(s,\Xast ({s}^{-});y)}{\varphi}{}|}^{2} \, d\mu (y)ds  \Bigr] 
\; = \; 0,
\quad  \varphi \in \Hmath  . 
\label{eq:Lebesgue_{noise}_{n}_conv_vcal}
\end{equation} 
Moreover, by \eqref{eq:F_n} and the fact that $\Pn :\Hmath \to \Hn $  is the $\ilsk{\cdot }{\cdot }{\Hmath }$-projection onto $\Hn$, see Section \ref{sec:truncated_eq}, we infer that also
\begin{equation*}
\lim_{n\to \infty } \bar{\e } \Bigl[  \int_{0}^{t} \int_{Y} 
{| \dual{F_n(s, \bXn ({s}^{-});y)- F(s,\Xast ({s}^{-});y)}{\varphi}{} |}^{2} \, d\mu (y)ds  \Bigr]
 \; = \; 0,
\quad  \varphi \in \Hmath  . 
\end{equation*}
This completes the proof of \eqref{eq:Lebesgue_{noise}_{n}_conv_H}.

\medskip  \noindent
By  \eqref{eq:Lebesgue_{noise}_{n}_conv_H}, \eqref{eq:Vitali_{noise}_{n}_est} (with $r=1$)
 and the dominated convergence theorem, we have for all $\varphi \in \Hmath $
\begin{equation}
\lim_{n\to \infty } \int_{0}^{T}
 \bar{\e } \Bigl[  \int_{0}^{t} \int_{Y} 
{| \dual{F_n(s, \bXn ({s}^{-});y)- F(s,\Xast ({s}^{-});y)}{\varphi}{} |}^{2} \, d\mu (y)ds  \Bigr] \, dt
 \; = \; 0 .  
\label{eq:{noise}_{n}_conv}
\end{equation} 
By \eqref{eq:Lebesgue_{noise}_{n}_equal_H}, \eqref{eq:{noise}_{n}_conv} and the Fubini theorem
\begin{equation}
\begin{split}
&\lim_{n\to \infty } \bar{\e } \Bigl[ \int_{0}^{T} 
\Bigl|  \int_{0}^{t}  \int_{Y}  
\dual{ F_n(s, \bXn ({s}^{-});y)}{\varphi }{}  \, {\tilde{\bar{{\eta }}}}_{n} (ds,dy)
\\
& \qquad \qquad - \int_{0}^{t} \int_{Y}  
\dual{ F(s,\Xast ({s}^{-});y)}{\varphi}{}  \, {\tilde{{\eta }}}_{\ast } (ds,dy) {\Bigr| }^{2}   
 \, dt \Bigr] \; = \; 0 .
\end{split}
\label{eq:{noise}_{n}_conv_final}
\end{equation} 
This concludes the proof of (f) and of Lemma \ref{lem:convergence_existence_truncated}.
\end{proof}

\medskip
\noindent
Directly from  Lemma \ref{lem:convergence_existence_truncated} we get the following corollary

\medskip
\begin{cor}
For every $\varphi \in \Vast $,
\begin{equation} \label{eq:un_convergence_truncated}
\lim_{n \to \infty } \norm{\ilsk{\bXn(\cdot )}{\varphi}{{\Hmath } }
- \ilsk{\Xast (\cdot )}{\varphi}{{\Hmath } } }{{L}^{2}([0,T]\times \bar{\Omega })}{} \; = \; 0
\end{equation}
and
\begin{equation} \label{eq:Lambda_n_un_convergence_truncated}
\lim_{n \to \infty } \norm{  {\Lambda  }_{n}(\bXn , {\bar{\eta }}_{n},\varphi)
- \Lambda  (\Xast ,{\eta }_{\ast }, \varphi) }{{L}^{1}([0,T]\times \bar{\Omega })}{} \; = \; 0 .
\end{equation}
\end{cor}

\begin{proof}
Assertion  \eqref{eq:un_convergence_truncated} follows from the  equality
\[
\norm{\ilsk{\bXn (\cdot )}{\varphi}{{\Hmath } }
- \ilsk{\Xast  (\cdot )}{\varphi}{{\Hmath } } }{{L}^{2}([0,T]\times \bar{\Omega })}{2} \\
\; = \;  \bar{\e } \Bigl[ \int_{0}^{T}
{| \ilsk{ \bXn (t) -\Xast (t)}{\varphi}{{\Hmath } } |}^{2} \, dt \Bigr]
\]
and Lemma \ref{lem:convergence_existence_truncated}  (a).
To prove  \eqref{eq:Lambda_n_un_convergence_truncated} let us note
that by the Fubini theorem, we have
\begin{align*}
& \norm{  {\Lambda  }_{n}( \bXn , {\bar{\eta }}_{n},\varphi)
- \Lambda  (\Xast ,{\eta }_{\ast }, \varphi) }{{L}^{1}([0,T]\times \bar{\Omega })}{} \\
\; &= \;  \int_{0}^{T} \bar{\e } \bigl[ {| {\Lambda  }_{n}( \bXn, {\bar{\eta }}_{n},\varphi)(t)
- \Lambda  (\Xast , {\eta }_{\ast } ,\varphi )(t) |}^{}\, \bigr] dt .
 \end{align*}
To complete the proof of \eqref{eq:Lambda_n_un_convergence_truncated} it is sufficient to note that
by Lemma \ref{lem:convergence_existence_truncated} (b)-(f),  each  term on the right hand side of
\eqref{eq:Lambda_n_un_truncated} tends at least in ${L}^{1}([0,T]$ $\times \bar{\Omega })$ to the corresponding term in \eqref{eq:Lambda_Hall-MHD}.
\end{proof}

\medskip  \noindent
\bf Step 2. \rm Since $\Xn $ is a solution of the truncated equation, for all $t\in [0,T]$ and $\varphi \in \Vast $
\[
\ilsk{\Xn (t)}{\varphi}{{\Hmath } } \; = \; {\Lambda  }_{n} (\Xn ,{\eta }_{},\varphi )(t) , \qquad \p \mbox{-a.s.}
\]
In particular,
\[
\int_{0}^{T} \e \bigl[ {|  \ilsk{\Xn (t)}{\varphi}{{\Hmath } }
 - {\Lambda  }_{n} (\Xn ,{\eta }_{},\varphi)(t) |}^{} \, \bigr] \, dt  \; = \; 0.
\]
Since $\Law (\Xn ,{\eta }_{}) = \Law (\bXn ,{\bar{\eta }}_{n})$,
using \eqref{eq:un_convergence_truncated} and \eqref{eq:Lambda_n_un_convergence_truncated} we infer that
\[
\int_{0}^{T} \bar{\e } \bigl[ {|
\ilsk{\Xast (t)}{\varphi}{{\Hmath } } - {\Lambda  }_{} (\Xast ,{\eta }_{\ast },\varphi)(t) |}^{} \, \bigr] \, dt  =0.
\]
Hence for $\ell $-almost all $t \in [0,T]$ and $\bar{\p }$-almost all $\omega \in \bar{\Omega }$
\begin{equation}
\ilsk{\Xast (t)}{\varphi}{{\Hmath } } - {\Lambda  }_{} (\Xast ,{\eta }_{\ast },\varphi )(t) \; = \; 0, \label{eq:solution_Hall-MHD}
\end{equation}
Since $\Xast $ is ${\zcal }_{T}$-valued  random variable, in particular $\Xast \in \Dmath  ([0,T];{{\Hmath } }_{w})$, 
we infer that
equality \eqref{eq:solution_Hall-MHD} holds  for all $t\in [0,T]$ and all $\varphi \in \Vast   $.
Since $\Vast   $ is dense in $\Vtest  $,   equality \eqref{eq:solution_Hall-MHD} holds for all $\varphi \in \Vtest   $, as well.
Putting $\bar{\mathfrak{A}}:=(\bar{\Omega }, \bar{\mathcal{F} },\bar{\p }, \bar{\Fmath })$,
$\bar{\eta} :=  {\eta }_{\ast }$ and $\bX := \Xast  $,
by \eqref{eq:solution_Hall-MHD}, \eqref{eq:Lambda_Hall-MHD}, \eqref{eq:H_estimate_Xast} and \eqref{eq:V_estimate_Xast}
 we infer that the system
$(\bar{\mathfrak{A}}, {\eta }_{\ast }, \Xast )$ is a martingale solution of  equation \eqref{eq:Hall-MHD_functional} satisfying inequality \eqref{eq:mart-sol_energy-ineq_exist}.
The proof of Theorem \ref{th:mart-sol_existence} is thus complete. \qed


\medskip
\section{Some generalizations} \label{sec:general}

\medskip  \noindent
Let us consider the  stochastic Hall-magnetohydrodynamics system on $[0,T] \times {\mathbb{R} }^{3}$ with the L\'{e}vy  noise terms 
\begin{align}
 d\ubold  +& \Bigl[ (\ubold \cdot \nabla  ) \ubold + \nabla p 
-  \,(\Bbold \cdot \nabla ) \Bbold + \, \nabla \Bigl( \frac{{|\Bbold |}^{2}}{2} \Bigr) 
- {\nu }_{1} \, \Delta \ubold \Bigr] \, dt \notag \\
\; &= \; {f}_1 (t) + \int_{Y_1}{F}_{1}(t,\ubold ({t}^{-});y) \, d{\tilde{\eta }}_{1}(dt,dy)
 + {\Gbold }_{1}(t,\ubold ) \, dW_1(t) ,
\label{eq:Hall-MHD_u_Levy} \\
  d \Bbold  +& \Bigl[ (\ubold \cdot \nabla ) \Bbold - (\Bbold \cdot \nabla ) \ubold 
+ \curl [(\curl \Bbold ) \times \Bbold ]  
- {\nu }_{2} \, \Delta \Bbold \Bigr] \, dt  \notag \\
\; &= \; {f}_2(t)  + \int_{Y_2}{F}_{2}(t,\Bbold ({t}^{-});y) \, d{\tilde{\eta }}_{2}(dt,dy) 
+{\Gbold }_{2}(t,\Bbold ) \, dW_2(t)
\label{eq:Hall-MHD_B_Levy} 
\end{align}
with the incompressibility conditions \eqref{eq:Hall-MHD_incompressibility}, i.e.
\begin{equation*}
\begin{split}
\diver \ubold \; &= \; 0 \quad \mbox{ and } \quad \diver \Bbold \; = \; 0
\end{split}
\end{equation*}
and the initial conditions \eqref{eq:Hall-MHD_ini-cond}, i.e.
\begin{equation*}
\ubold (0) \; = \; {\ubold }_{0} \quad \mbox{ and } \quad \Bbold (0) \; = \; {\Bbold }_{0}.
\end{equation*}

\medskip  \noindent
\begin{assumption}  \label{ass:data+noise_Wiener} \
\begin{description}
\item[(G.1)] ${\Kmath }_{1}, {\Kmath }_{2} $  are separable Hilbert spaces, and
\[
G_i: [0,T] \times V  \to \lhs ({\Kmath }_{i}, H), \qquad i=1,2,
\]
are two measurable map which are Lipschitz continuous, i.e.
there exist constants ${L}_{i}$, $i=1,2$,
such that 
\begin{equation}
\norm{G_i(t,{\phi }_{1}) - G_i(t,{\phi }_{2})}{\lhs ({\Kmath }_{i}, H)}{2} 
\; \le \; {L}_{i}\, \norm{{\phi }_1-{\phi }_{2}}{V}{2}  ,
   \qquad {\phi }_{1}, {\phi }_{2} \in V  , \, \,  t \in [0,T] .
\label{eq:G_i_Lipschitz}   
\end{equation} 
In addition, there exist ${\lambda }_{i}$, ${\varrho }_{i} \in \mathbb{R} $ and 
$a \in (0,2]$  such that
\begin{equation} \label{eq:G_i}
\norm{G_i(t,\phi  )}{\lhs ({\Kmath }_{i},H)}{2}
\; \le \; {\nu }_{i}(2- a ) \nnorm{\nabla \phi }{L^2}{2} +{\lambda }_{i} \nnorm{\phi }{H}{2} +{\varrho }_{i}  , \quad (t,\phi ) \in [0,T] \times V  .
\end{equation}
\item[(G.2)] The maps $G_i$, $i=1,2$, can be extended to  measurable maps
\[ 
g_i :[0,T] \times H \to \mathcal{L}  ( {\Kmath }_{i},  {V}^{\prime }) 
\]
such that for some $C_i>0$
\begin{equation} \label{eq:G_i*}
\sup_{\psi \in V ,\norm{\psi }{V}{} \le 1 } \sup_{y\in {\Kmath }_{i},\norm{y}{{\Kmath }_{i}}{} \le 1}
{|\ddual{V'}{g(t,\phi )(y)}{\psi }{V}|}^{2} 
 \le C_i (1 + \nnorm{\phi }{H}{2}) , \qquad (t,\phi ) \in [0,T] \times H  . 
\end{equation}
\item[(G.3)]
Moreover, for every $\psi  \in V$ the maps ${\tilde{g_i}}_{\psi }$ defined for $\phi \in {L}^{2}(0,T;H) $  by
\begin{equation} \label{eq:G_i**}
 \bigl( { \tilde{g_i} }_{\psi }(\phi )\bigr)  (t) := 
\{ {\Kmath }_{i} \ni y \mapsto \ddual{V'}{g_i(t,\phi (t))(y)}{\psi }{V} \in \mathbb{R} \} \in  \lhs ({\Kmath }_{i},\mathbb{R} ) ,
\quad t \in [0,T],
\end{equation}
are continuous maps from ${L}^{2}(0,T;{H}_{loc}) $ into $ {L}^{2}(0,T;\lhs ({\Kmath }_{i},\mathbb{R} ) ) $.
\end{description}
\end{assumption}

\noindent
For any Hilbert spaces $\Kmath $ and $Y$ by $\lhs (\Kmath;Y)$ we denote the space of Hilbert-Schmidt operators from $\Kmath$ into  $Y$. 

\medskip
\begin{assumption} \label{assumption-data}  \rm 
We assume also that 
\begin{description}
\item[(H.1)]  $p$ is a real number such that 
\begin{equation} \label{eqn-p_cond}
p \in [2,2+\gamma )  ,
\end{equation}
where
\begin{equation}
\gamma := \begin{cases}
\frac{a}{2-a} , \quad &\mbox{ if }  a \in [0,2) , \\
\infty , \quad &\mbox{ if }  a =2,
\end{cases}
\label{eq:gamma}
\end{equation}
and $a$ is the parameter from inequality \eqref{eq:G_i}. 
\item[(H.2)] ${\X }_{0} := ({\ubold }_{0}, {\Bbold }_{0}) \in H \times H $ and $f:= (f_1,f_2)$, where $f_i \in {L}^{p}(0,T;{V}^{\prime })$ for $i=1,2$.
\item[(H.3)] 
$\mathfrak{A}:=(\Omega , \mathcal{F} , \Fmath , \p )$ is a filtered probability space with a filtration
 $\Fmath ={({\mathcal{F} }_{t})}_{t \ge 0}$ satisfying usual hypotheses and 
 $W_i(t)$  are two cylindrical  Wiener processes in a separable Hilbert space ${\Kmath }_{i}$, $i=1,2$, defined on the stochastic basis $\mathfrak{A}$. 
\end{description}
\end{assumption}

\medskip \noindent
Let $W(t):= (W_1(t),W_2(t))$. Then $W(t)$ is a cylindrical Wiener process on $\Kmath :={\Kmath }_{1} \times {\Kmath }_{2}$,  on the stochastic basis $\mathfrak{A}$.
Using the maps ${G}_{1}$ and ${G}_{2}$  given in Assumption \ref{ass:data+noise_Wiener} we 
 define the  map $G$ by
\begin{equation}
G(t,\Phi )(y) \; := \; (G_1(t,\ubold )(y_1),G_2(t, \Bbold )(y_2)), 
\label{eq:G_map}
\end{equation}
where $ t \in [0,T], \;  \Phi := (\ubold ,\Bbold ) \in \Vmath ,  \; y=(y_1,y_2) \in \Kmath  $.

\medskip  
\begin{remark} \label{rem:G_properties} \
\begin{itemize} 
\item[\rm (i) ] Then
\[
G : [0,T] \times \Vmath \; \to \; \lhs (\Kmath ,\Hmath ). 
\]
The map $G$ satisfies the Lipschitz condition, i.e there exists a constant $L>0$ such that
\begin{equation}
\norm{G(t,{\Phi }_{1}) - G(t,{\Phi }_{2})}{\lhs (\Kmath,\Hmath )}{2} 
\; \le \; L \norm{{\Phi }_{1}- {\Phi }_2}{\Vmath }{2} ,
   \qquad {\Phi }_{1}, {\Phi }_{2} \in \Vmath  , \, \,  t \in [0,T] .
\label{eq:G_Lipschitz}   
\end{equation} 
\item[\rm (ii) ] The map $G$ satisfies the following inequality
\begin{equation}
\begin{split}
\norm{G(s,\Phi )}{\lhs ({\Kmath },{\Hmath })}{2}
\; & \le \; (2-a ) \norm{\Phi }{}{2} + \lambda \nnorm{\Phi }{\Hmath }{2} + \varrho  ,  \quad (s,\Phi  ) \in [0,T] \times \Vmath ,
\end{split}
\label{eq:G}
\end{equation}
where $\lambda := {\lambda }_{1}+{\lambda }_{2} $ and $\varrho :={\varrho }_{1} + {\varrho }_{2}$.
\item[\rm (iii) ] Let ${g}_{1}$ and ${g}_{2}$ be the maps from Assumption \ref{ass:data+noise_Wiener} and let us define
\begin{equation}
g(t,\Phi )(y) \; := \; (g_1(t,\ubold )(y_1),g_2(t, \Bbold )(y_2)), 
\label{eq:g_map}
\end{equation}
where $ t \in [0,T], \;  \Phi := (\ubold ,\Bbold ) \in \Vmath ,  \; y=(y_1,y_2) \in \Kmath  $. 
Then the map $g$ is an extension of the map $G$ to a measurable map
\[ 
g :[0,T] \times \Hmath  \to \mathcal{L}  ( \Kmath,  {\Vmath }^{\prime }) 
\]
and by \eqref{eq:G_i*} we obtain
\begin{equation} \label{eq:G*}
\sup_{\Psi \in \Vmath  ,\norm{\Psi }{\Vmath }{} \le 1 } \sup_{y\in \Kmath ,\norm{y}{\Kmath }{} \le 1}
{|\ddual{\Vmath '}{g(t,\Phi )(y)}{\Psi }{\Vmath }|}^{2} 
 \le C (1 + \nnorm{\Phi }{\Hmath }{2}) , \qquad (t,\Phi ) \in [0,T] \times \Hmath   . 
\end{equation}
Moreover, for every $\Psi  \in \Vmath $ the map ${\tilde{g}}_{\Psi }$ defined for $\Phi \in {L}^{2}(0,T;\Hmath ) $  by
\begin{equation} \label{eq:G**}
 \bigl( { \tilde{g} }_{\Psi }(\Phi )\bigr)  (t) \; := \; 
\{ \Kmath \ni y \mapsto \ddual{\Vmath '}{g(t,\Phi (t))(y)}{\Psi }{\Vmath } \in \mathbb{R} \} \in  \lhs (\Kmath ,\mathbb{R} ) ,
\quad t \in [0,T],
\end{equation}
is a continuous map from ${L}^{2}(0,T;{\Hmath }_{loc}) $ into $ {L}^{2}(0,T;\lhs (\Kmath,\mathbb{R} ) ) $.
\end{itemize}
\end{remark}

\medskip  \noindent
Using the maps  $\mathcal{A} $, $\MHD $, $\tHall $, $F$ and $G$ defined respectively  by \eqref{eq:A_acal_rel},  \eqref{eq:MHD-map}, \eqref{eq:tHall_map}, \eqref{eq:F_map} and \eqref{eq:G_map},
equations \eqref{eq:Hall-MHD_u_Levy}-\eqref{eq:Hall-MHD_u_Levy}
with the incompressibility conditions \eqref{eq:Hall-MHD_incompressibility}
and the initial conditions \eqref{eq:Hall-MHD_ini-cond}
can be rewritten as the following abstract SPDE
\begin{equation}
\begin{split} 
&  \X (t)  +  \int_{0}^{t} \bigl[  \mathcal{A} \X (s)  + \MHD  (\X (s)) +\tHall (\X (s)) \bigr] \, ds  
   \; = \; {\X }_{0} + \int_{0}^{t}  f (s) \, ds  \\
& + \int_{0}^{t}\int_{Y}  F(s,\X ({s}^{-});y ) \, \tilde{\eta }(ds,dy)
 + \int_{0}^{t} G (s,\X (s) ) \, dW(s)  , \quad t \in [0,T] .
\end{split}  \label{eq:Hall-MHD_functional_Levy}
\end{equation} 

\medskip
\begin{definition}  \rm  \label{def:mart-sol_gen}
Let Assumption \ref{ass:data+noise_Poisson}, \ref{ass:data+noise_Wiener}  and \ref{assumption-data} be satisfied. 
We say that there exists a \bf martingale solution \rm of  problem \eqref{eq:Hall-MHD_functional_Levy} 
iff there exist
\begin{itemize}
\item[$\bullet $] a stochastic basis $\bar{\mathfrak{A}}:= \bigl( \bar{\Omega }, \bar{\mathcal{F} },  \bar{\Fmath } ,\bar{\p }  \bigr) $ with a  filtration $\bar{\Fmath } = \{ {\bar{\mathcal{F} }_{t}}{\} }_{t \in [0,T]} $ satisfying the usual conditions, 
\item[$\bullet $]  a time homogeneous Poisson random measure $\bar{\eta }$ on $(Y, \ycal )$ over
$\bar{\mathfrak{A}} $ with the intensity measure $\mu $,
\item[$\bullet $] a $\Kmath $-cylindrical Wiener process $\bar{W}$  over
$\bar{\mathfrak{A}} $,
\item[$\bullet $] and an $\bar{\Fmath }$- progressively measurable process $\bX : [0,T] \times \Omega \to \Hmath $ with  paths satisfying
\begin{equation}
\bX (\cdot , \omega ) \; \in \; \Dmath ([0,T]; {\Hmath }_{w}) \cap {L}^{2}(0,T;\Vmath ) ,
\label{eq:mart-sol_regularity_Levy}
\end{equation}
for $\bar{\p } $-a.e. $\omega \in \bar{\Omega }$,
and such that 
for all $ t \in [0,T] $ and  $\phi \in \Vtest  $ the following identity holds $\bar{\p }$ - a.s.
\begin{equation} 
\begin{split}
&\ilsk{\bX (t)}{\phi}{\Hmath }  + \int_{0}^{t} \dual{\mathcal{A} \bX (s)}{\phi}{}  ds
+ \int_{0}^{t} \dual{\MHD (\bX (s))}{\phi }{}  ds 
+ \int_{0}^{t} \dual{\tHall (\bX (s))}{\phi }{}  ds
    \\
\; & = \; \ilsk{{\X }_{0}}{\phi}{\Hmath } + \int_{0}^{t} \dual{f(s)}{\varphi  }{} ds
 + \int_{0}^{t} \int_{Y}  \ilsk{F(s,\bX ({s}^{-});y ) }{\phi }{\Hmath } \, \tilde{\bar{\eta }}(ds,dy)
 \\
& \qquad   + \Dual{\int_{0}^{t}G(s,\bX (s))\,  d\bar{W}(s)}{\varphi  }{} .
\end{split}  \label{eq:mart-sol_int-identity_gen}
\end{equation}
\end{itemize}
If all the above conditions are satisfied, then the system
$ \bigl( \bar{\mathfrak{A}}, \bar{\eta } , \bar{W} ,\bX \bigr)  $
is called a \bf martingale solution \rm of problem \eqref{eq:Hall-MHD_functional_Levy}.
\end{definition}

\medskip  \noindent
Using the ideas from \cite{EM'14} and \cite{EM'21_ArX_Hall-MHD_Gauss} we can prove the following generalization of Theorem \ref{th:mart-sol_existence}.

\medskip
\begin{theorem} \label{th:mart-sol_existence_Levy}
Let Assumptions \ref{ass:data+noise_Poisson}, \ref{ass:data+noise_Wiener} and \ref{assumption-data}  be satisfied.
In particular, we assume that
 $p$ satisfies  \eqref{eqn-p_cond}, i.e.
\[
p \in [2,2+\gamma ),
\] 
where $\gamma $ is given by \eqref{eq:gamma}.
Then there exists a  martingale solution  of  problem \eqref{eq:Hall-MHD_functional_Levy} such that
\begin{description}
\item[(i)] for every $q\in [1,p]$ there exists a constant $C_1(p,q)$ such that
\begin{equation}
\bar{\e } \Bigl[ \sup_{t \in [0,T]} \nnorm{\bX (t)}{\Hmath }{q} \Bigr] \; \le \;  C_1(p,q),
\label{eq:H_estimate_Hall-MHD_q_Levy}
\end{equation}
\item[(ii)] there exists a constant $C_2(p)$ such that
\begin{equation}
\bar{E} \biggl[ \int_{0}^{T} \norm{\bX (t)}{\Vmath }{2} \, dt \biggr] \; \le \; C_2(p) .
\label{eq:V_estimate_Hall-MHD_Levy}
\end{equation}
\end{description}
\end{theorem}

\pagebreak

\medskip

\appendix

\section{The space of c\`{a}dl\`{a}g functions} \label{sec:Cadlag_functions}

\medskip \noindent
Let $(\Smath ,\varrho )$ be a separable and complete metric space.
Let $\Dmath ([0,T];\Smath )$  the space of all $\Smath $-valued \it c\`{a}dl\`{a}g \rm functions defined on $[0,T]$, i.e. the functions  which are right continuous and have left limits at every $t\in [0,T]$ . 
The space $\Dmath ([0,T];\Smath )$ is endowed with the Skorokhod topology.

\medskip  \noindent
\begin{remark} \label{rem:cadlag_conv}
A sequence $(\un ) \subset \Dmath ([0,T];\Smath )$ converges to $u \in \Dmath ([0,T];\Smath )$ iff there exists a sequence $({\lambda }_{n})$ of homeomorphisms of  $[0,T]$ such that ${\lambda }_{n} $ tends to the identity uniformly on $[0,T]$ and $\un  \circ {\lambda }_{n} $ tends to $u$ uniformly on $[0,T]$.
\end{remark} 

\medskip \noindent
This topology is  metrizable by the following metric ${\delta }_{T}$
\[  \label{eq:II_1.1.1_[Metivier_88]}
{\delta }_{T}(u,v) \; := \; \inf_{\lambda \in {\Lambda }_{T}} \Bigl[ \sup_{t\in [0,T]}
  \varrho  \bigl( u(t), v \circ \lambda (t)\bigr)  
  + \sup_{t\in [0,T]} |t-\lambda (t)|
 + \sup_{s\ne t} \Bigl| \log \frac{\lambda (t)-\lambda (s)}{t-s} \Bigr| \Bigr]  ,
\]
where ${\Lambda }_{T}$ is the set of increasing homeomorphisms of $[0,T]$.
Moreover, 

\noindent
$\bigl(  \Dmath ([0,T];\Smath ),{\delta }_{T}\bigr) $ is a complete metric space,
see \cite{Joffe+Metivier'86} and \cite{Jacod+Shiryaev'2003}.

\medskip  \noindent
Let us recall the notion of a \it modulus \rm of the function. It plays analogous role in the space $\Dmath ([0,T];\Smath )$ as \it the modulus of continuity \rm in the space of continuous functions $\Cmath ([0,T];\Smath )$.

\medskip  \noindent
\begin{definition}  \rm (see \cite{Metivier'88})
Let $u \in \Dmath ([0,T];\Smath )$ and let $ \delta >0 $ be given. A \bf modulus \rm of $u$ is defined by 
\begin{equation}  \label{eq:modulus_cadlag}
{w}_{[0,T],\Smath }(u,\delta ) \; : = \; \inf_{{\Pi }_{\delta }} \, \max_{{t}_{i} \in \bar{\omega }} \,
  \sup_{{t}_{i} \le s <t < {t}_{i+1} \le T } \varrho  \bigl( u(t), u(s) \bigr) ,
\end{equation}
where ${\Pi }_{\delta }$ is the set of all increasing sequences 
$
\bar{\omega } \; = \; \{ 0= {t}_{0} < {t}_{1} < ... < {t}_{n} =T  \}
$
with the following property
\[
{t}_{i+1} - {t}_{i}  \; \ge \; \delta , \qquad i=0,1,...,n-1.
\]
If no confusion seems likely, we will denote the modulus by ${w}_{[0,T]}(u,\delta )$.

\end{definition}

\medskip  \noindent
We have the following criterion for relative compactness of a subset of the space $\Dmath ([0,T];\Smath )$,
see \cite{Joffe+Metivier'86},\cite[Ch.II]{Metivier'88} and \cite[Ch.3]{Billingsley},
analogous to the Arzel\`{a}-Ascoli theorem for the space of continuous functions.

\medskip
\begin{theorem} \label{th:cadlag_compactness} \it 
A set $A \subset \Dmath ([0,T];\Smath ) $ has compact closure iff it satisfies the following two conditions:
\begin{itemize}
\item[(a) ] there exists a dense subset $J \subset [0,T]$ such that 
for every $ t \in J  $  the set $\{ u(t), \, \, u \in A  \} $ has compact closure in $\Smath $.
\item[(b) ] $ \lim_{\delta \to 0} \, \sup_{u \in A}  \, {w}_{[0,T]}(u,\delta ) =0 $.
\end{itemize}
\end{theorem}

\medskip 

\section{A version of the Skorohod embedding theorem} \label{sec:Skorokhod_th}

\medskip 
\noindent
In the proof of Theorem \ref{th:mart-sol_existence}  we use the following  version of the Skorohod embedding theorem following from the version due to Jakubowski \cite{{Jakubowski'98}} and the version  due to Brze\'{z}niak,  Hausenblas and Razafimandimby  \cite[Theorem C.1]{Brze+Haus+Raza'18}. 

\medskip 
\begin{cor} \label{cor:Skorokhod_J,B,H} \rm (Corollary 2 in \cite{EM'13}) \it 
 Let ${\xcal }_{1}$ be a separable complete metric space and let ${\xcal }_{2}$ be a topological space such that
there exists a sequence  $\{ {f}_{\iota }{ \} }_{\iota \in \mathbb{N} } $ of continuous functions ${f}_{\iota }:{\xcal }_{2} \to \mathbb{R} $  separating points of ${\xcal }_{2}$.
Let $\xcal := {\xcal }_{1}\times {\xcal }_{2}$ with the Tykhonoff topology induced by the projections 
$$
{\pi }_{i}: {\xcal }_{1}\times {\xcal }_{2} \to {\xcal }_{i} , \qquad  i=1,2.
$$
Let $(\Omega ,\mathcal{F} ,\p )$ be a probability space and let 
${\chi }_{n}:\Omega \to {\xcal }_{1}\times {\xcal }_{2}$, $n\in \mathbb{N} $, be a family of random variables such that the sequence $\{ \mathcal{L} aw({\chi }_{n}), n \in \mathbb{N} \} $ is tight on ${\xcal }_{1}\times {\xcal }_{2}$.
Finally let us assume that there exists a random variable $\rho :\Omega \to {\xcal }_{1}$ such that 
$\mathcal{L} aw({\pi }_{1}\circ {\chi }_{n}) = \mathcal{L} aw(\rho )$ for  all $ n \in \mathbb{N} $.

\medskip \noindent
Then there exists a subsequence $\bigl( {\chi }_{{n}_{k}} {\bigr) }_{k \in \mathbb{N} } $,
a probability space $(\bar{\Omega }, \bar{\mathcal{F} }, \bar{\p })$, a family of ${\xcal }_{1}\times {\xcal }_{2}$-valued random variables $\{ {\bar{\chi }}_{k}, \, k \in \mathbb{N}  \} $ on 
$(\bar{\Omega }, \bar{\mathcal{F} }, \bar{\p })$ and a random variable ${\chi }_{*}: \bar{\Omega } \to 
{\xcal }_{1}\times {\xcal }_{2}$ such that 
\begin{itemize}
\item[(i) ] $\mathcal{L} aw ({\bar{\chi }}_{k}) = \mathcal{L} aw ({\chi }_{{n}_{k}})$ for all $ k \in \mathbb{N} $;
\item[(ii) ] ${\bar{\chi }}_{k} \to {\chi }_{*}$ in ${\xcal }_{1}\times {\xcal }_{2}$ a.s. as $k \to \infty $;
\item[(iii) ] ${\pi }_{1} \circ {\bar{\chi }}_{k} (\bar{\omega }) = {\pi }_{1} \circ {\chi }_{*}(\bar{\omega })$ for all $\bar{\omega } \in \bar{\Omega }$.
\end{itemize}
\end{cor}
\noindent
In topological spaces we consider the $\sigma $-fields generated by the topology, i.e. the smallest $\sigma $-field containing the topology.

\medskip 
\noindent
In the proof of Theorem \ref{th:mart-sol_existence} we use Corollary \ref{cor:Skorokhod_J,B,H} for the space
$ {\xcal }_{2} :=  {\zcal }_{} $ defined by \eqref{eq:Z_cadlag}, i.e.,
\[
{\xcal }_{2} \; := \; {\zcal }_{}
\; = \;  {L}_{w}^{2}(0,T;\Vmath ) \, \cap \,  {L}^{2}(0,T;{\Hmath }_{loc})
\,\cap \, \Dmath ([0,T];\Uprime)
\, \cap \,  \Dmath ([0,T];{\Hmath }_{w}) .
\]

\medskip  
\begin{remark} \ \label{rem:separating_maps}   Proceeding similarly as in  \cite[Remark 2]{EM'13}, we infer that the space $(\zcal , \sigma (\zcal ))$ defined in Definition \ref{def:space_Z_cadlag} satisfies assumptions of Corollary \ref{cor:Skorokhod_J,B,H}.
\end{remark}

\medskip
\section{The spaces $\Ldn $ and the cut-off operators $\Sn $}
\label{sec:Fourier_truncation}

\medskip  \noindent
We recall crucial facts related to the Fourier truncation method, called also the Friedrichs method, see \cite[Section 4, p.174]{Bahouri+Chemin+Danchin'11}. This approach is also used, e.g., in \cite{Brze+Dha'20} and \cite{Feff+McCorm+Rob+Rod'2014}.
For details see \cite[Appendix A]{EM'21_ArX_Hall-MHD_Gauss} and references given there.

\medskip  \noindent
The Fourier transform of a rapidly decreasing function $\psi \in \mathcal{S} (\rd )$ is defined by (see \cite{Rudin}, \cite{Taylor_I'11})
\[
\widehat{\psi } (\xi ) \; := \; {(2\pi )}^{-\frac{d}{2}} \int_{\rd } {e}^{-i\xi \cdot x } \psi (x) \, dx , \qquad \xi \in \rd ,
\]
and the Fourier transform $\widehat{f}$ of a tempered distribution is defined via duality.

\medskip  \noindent
For $s\ge 0 $ the Sobolev space is defined by
\[
 {H}^{s}(\rd ) \; := \; \{ u \in {L}^{2}(\rd ): \; \; {(1+{|\xi |}^{2})}^{\frac{s}{2}} \widehat{u} \in {L}^{2}(\rd ) \}  
\]
and
\[
\norm{u}{{H}^{s}}{} \; := \; \nnorm{{(1+{|\xi |}^{2})}^{\frac{s}{2}} \widehat{u}}{{L}^{2}(\rd )}{}
\; = \; \biggl( \int_{\rd } {(1+{|\xi |}^{2})}^{s} {|\widehat{u}(\xi )|}^{2} \, d \xi {\biggr) }^{\frac{1}{2}}.
\]
(See \cite{Taylor_I'11}, \cite{Stein'1970}.) The spaces ${H}^{s}(\rd )$ are also called Lebesgue spaces and denoted by ${L}_{s}^{2}$.

\medskip  \noindent
Let
\[
{\bar{B}}_{n} \; := \; \{ \xi \in \rd : \; \; |\xi | \le n \} , \qquad n \in \mathbb{N} , 
\]
and let
\begin{equation}
\Ldn  \; := \; \{ u \in {L}^{2}(\rd ) : \; \; \supp \widehat{u} \subset {\bar{B}}_{n} \} .
\label{eq:Ld_n}
\end{equation}
On the subspace $\Ldn $ we consider the norm inherited from ${L}^{2}(\rd )$.

\medskip  \noindent
The \bf cut-off operator \rm  $\Sn $ is defined by
\begin{equation}
\Sn (u)  \; := \; {\mathcal{F} }^{-1} (\ind{{\bar{B}}_{n}} \widehat{u}) , \qquad u \in {L}^{2}(\rd ),
\label{eq:S_n}
\end{equation}
where ${\mathcal{F} }^{-1}$ denotes the inverse Fourier transform.
(See  \cite[Section 4, p.174]{Bahouri+Chemin+Danchin'11}.)

\medskip
\begin{remark} \label{rem:S_n-projection}
\rm (See \cite[Section 4, p.174]{Bahouri+Chemin+Danchin'11} and \cite{Brze+Dha'20}.)  \it 
The map
\[
\Sn  : {L}^{2}(\rd ) \; \to \; \Ldn  
\] 
is an  $\ilsk{\cdot }{\cdot }{{L}^{2}}$-orthogonal projection.
\end{remark}

\medskip  \noindent
Basic properties  of the  operators $\Sn $ are stated in the following two lemmas, see  \cite[Appendix A]{EM'21_ArX_Hall-MHD_Gauss} and references given there. 

\medskip
\begin{lemma} \label{lem:S_n-pointwise_conv}
(See \cite[Lemma A.2]{EM'21_ArX_Hall-MHD_Gauss}.)
Let  $s\ge 0 $ be fixed. 
Then for all $n \in \mathbb{N} $:
\[
\Sn  : {H}^{s}(\rd ) \; \to \; {H}^{s}(\rd )
\]
is well defined linear and bounded. Moreover, for every $ u \in {H}^{s}(\rd ) $
\begin{equation}
\norm{\Sn u}{{H}^{s}}{} \; \le \; \norm{u}{{H}^{s}}{} 
\label{eq:S_n:H^s-H^s}
\end{equation}
and
\begin{align}
\lim_{n \to \infty } \norm{\Sn u-u}{{H}^{s}}{}  \; = \; 0 . 
\label{eq:S_n_pointwise_conv}
\end{align} 
\end{lemma}

\medskip
\begin{lemma} \label{lem:S_n-norm_conv}
(See \cite[Lemma A.3]{EM'21_ArX_Hall-MHD_Gauss}.)
If $s\ge 0 $ and $k>0 $, then 
\begin{equation*}
\Sn :{H}^{s+k}(\rd ) \to {H}^{s}(\rd )
\end{equation*} 
is well defined  and bounded and 
$
\nnorm{\Sn }{\mathcal{L} ({H}^{s+k},{H}^{s})}{} \; \le  \; 1 .
$
Moreover, for every $u \in {H}^{s+k}(\rd )$: 
\begin{equation}
\norm{\Sn u-u}{{H}^{s}}{2} \; \le \; \frac{1}{{(1+{n}^{2})}^{k}} \, \norm{u}{{H}^{s+k}}{2}.
\label{eq:S_n:H^s+k-H^s_est}
\end{equation}
Thus
\begin{equation}
\lim_{n\to \infty } \nnorm{\Sn -I}{\mathcal{L} ({H}^{s+k},{H}^{s})}{} \; = \; 0, 
\label{eq:S_n:H^s+k-H^s_norn-conv}
\end{equation}
where $I$ stands for the identity operator.  
\end{lemma}

\medskip  \noindent
Let us also recall also the relation between the spaces $\Ldn $ and ${H}^{s}(\rd )$ for $s\ge 0$.
By definition, on the spaces $\Ldn $ we consider the norms inherited from the space ${L}^{2}(\rd )$, see \eqref{eq:Ld_n}.

\medskip
\begin{lemma}  \label{lem:Ld_n-H^s-relation}
(See \cite[Lemma A.4]{EM'21_ArX_Hall-MHD_Gauss}.)
For each $n \in \mathbb{N} $
\begin{equation*}
\Ldn \; {\hookrightarrow } \; {H}^{s}(\rd ) \quad \mbox{ for all } s \ge 0  
\end{equation*}
and  for every $ s \ge 0 $ and  $ u \in \Ldn :$
\begin{equation}
\norm{u}{{H}^{s}}{2} \; \le \;  {(1+{n}^{2})}^{s} \, \nnorm{u}{\Ldn }{2} .  
\label{eq:Ld_n-H^s-ineq} 
\end{equation}
\end{lemma}
\noindent
In particular,  the norm of the embedding $\Ldn \subset {H}^{s}(\rd )$ depends on $n$ and $s$.

\medskip
\begin{cor}  \label{cor:Ld_n-H^s-norm_equiv}
(See \cite[Corollary A.5]{EM'21_ArX_Hall-MHD_Gauss}.)
On the subspace $\Ldn $ the norms $\nnorm{\cdot}{\Ldn }{}$ and $\norm{\cdot }{H^s}{}$, for $s>0$, are equivalent (with appropriate constants depending on $s$ and $n$).   
\end{cor}


\medskip

\begin{thebibliography}{10}

\bibitem{Acheritogaray+Degond'11} M. Acheritogaray, P. Degond, A. Frouvelle and J-G. Liu, \emph{Kinetic formulation and global existence for the Hall-Magneto-hydrodynamics system}, Kinet. Relat. Models, \bf 4\rm , 901--918  (2011).


\bibitem{Albeverio+Brzezniak+Wu'2010}  S. Albeverio, Z. Brze\'{z}niak, J-L. Wu, \emph{Existence of global solutions and invariant measures for stochastic differential equations driven by Poisson type noise with non-Lipschitz coefficients}, J. Math. Anal. Appl. 371, 309--322, (2010). 

\bibitem{Applebaum'2009} D. Applebaum, \emph{L\'{e}vy Processes and Stochastic Calculus},
 Cambridge University Press, 2009. 
 
\bibitem{Bahouri+Chemin+Danchin'11} H. Bahouri, J.-Y. Chemin, R. Danchin,
\emph{\sc Fourier analysis and nonlinear partial differential equations}, Springer-Verlag, 2011.

\bibitem{Billingsley} P. Billingsley,  \emph{\sc Convergence of probability measures}, Wiley, New York, 1969. 

\bibitem{Brezis'2011} H. Brezis, \emph{\sc Functional analysis, Sobolev spaces and partial differential equations}, Springer - New York Dordrecht Heidelberg London, 2011.   

\bibitem{Brze+Dha'20}  Z. Brze\'{z}niak, G. Dhariwal, \emph{Stochastic tamed Navier-Stokes equations on ${\mathbb{R} }^{3}$: the existence and the uniqueness of solutions and the existence of an invariant measure}, J. Math. Fluid Mech. 22:23, 1--54,  (2020).

     
\bibitem{Brze+Hausenblas'2009} Z.~Brze\'{z}niak, E.~Hausenblas, \emph{Maximal regularity for stochastic convolution driven by L\'{e}vy processes},  Probab. Theory Relat. Fields, 145:615--637 (2009). 

\bibitem{Brze+Haus+Raza'18} Z.~Brze\'{z}niak, E.~Hausenblas, P.A.~Razafimandimby, \emph{Stochastic reaction-diffusion equations driven by jump processes}, Potential Anal  49:131--201, (2018).


\bibitem{Brze+EM'13}  Z. Brze\'{z}niak, E. Motyl,  \emph{Existence of a martingale  solution to the stochastic Navier-Stokes equations in unbounded 2D and 3D domains},  \rm J. Differential Equations,  254, 1627--1685 (2013).

\bibitem{Chae+Degond+Liu'14} D. Chae, P. Degond and J.-G. Liu, \emph{Well-posedness for Hall-magnetohydrodynamics}, Ann. Inst. H. Poincar\'{e} Anal. Non Lin\'{e}aire, \bf 31\rm , 555--565  (2014).

\bibitem{Chae+Lee'14}  D. Chae and J. Lee, \emph{On the blow-up criterion and small data global existence for the Hall-magnetohydrodynamics}, J. Differential Equations, \bf 256\rm , 3835--3858  (2014).

\bibitem{Chae+Schonbek'13} D. Chae and M. Schonbek, \emph{On the temporal decay for the Hall-magnetohydrodynamic equations}, J. Differential Equations, \bf 255\rm , 3971--3982  (2013).

\bibitem{Chae+Wan+Wu'15}  D. Chae, R. Wan and J. Wu, \emph{Local well-posedness for the Hall-MHD equations with fractional magnetic diffusion}, J. Math. Fluid Mech., \bf 17\rm , 627--638  (2015).

\bibitem{Chueshov+Millet'10} I. Chueshov, A. Millet, \emph{Stochastic 2D hydrodynamical type systems: well posedeness and large deviations}, Appl. Math. Optim. \bf 61\rm (3),  379--420 (2010).


\bibitem{Feff+McCorm+Rob+Rod'2014} C.L. Fefferman, D.S. McCormick, J.C. Robinson, J.L. Rodrigo, \emph{Higher order commutator estimates and local existence for the non-resistive MHD equations and related models}, Journal of Functional Analysis \bf 267\rm , 1035--1056 (2014).


\bibitem{Gyongy+Wu'21} I. Gy\"{o}ngy, S. Wu, \emph{It\^{o}'s formula for jump processes in ${L}^{p}$-spaces},
Stochastic  Process.    Appl. \bf 131\rm , 523--552 (2021). 

\bibitem{Holly+Wiciak'1995} K. Holly, M. Wiciak, \emph{Compactness method applied to an abstract nonlinear parabolic equation}, Selected problems of Mathematics, Cracow University of Technology,  95--160 (1995). 


\bibitem{Ikeda+Watanabe'81} N.~Ikeda, S.~Watanabe, \emph{\sc Stochastic Differential Equations and Diffusion Processes}, North-Holland Publishing Company, Amsterdam, 1981. 

\bibitem{Jacod+Shiryaev'2003}  J.~Jacod and A.~Shiryaev, \emph{\sc Limit theorems for stochastic processes}, volume 288 of Grundlehren der Mathematischen Wissenschaften [Fundamental Principles of
 Mathematical Sciences]. Springer-Verlag, Berlin, second edition, 2003.  
          
\bibitem{Jakubowski'98} A.~Jakubowski, \emph{The almost sure Skorohod representation for subsequences in nonmetric spaces}, Teor. Veroyatnost. i Primenen. \bf 42\rm , no. 1, 209-216 (1997); translation in Theory Probab.  Appl. \bf 42 \rm no.1, 167--174 (1998).  

\bibitem{Joffe+Metivier'86} A.~Joffe, M.~M\'{e}tivier, \emph{Weak convergence of sequences of semimartingales with applications to multitype  branching processes}, Adv. Appl. Prob. \bf 18\rm ,  20--65  (1986).

\bibitem{Manna+Mohan'2013} U. Manna and M. T. Mohan, \emph{Two-dimensional magneto-hydrodynamic system with jump processes: well posedness and invariant measures}, Commun. Stoch. Anal., \bf 7\rm , 153--178  (2013).

\bibitem{Manna+Mohan+Srith'2017}  U. Manna, M. T. Mohan and S. S. Sritharan, \emph{Stochastic non-resistive magnetohydrodynamic system with L\'{e}vy noise}, Random Oper. Stoch. Equ., \bf 25\rm , 155--194  (2017).


\bibitem{Metivier'88} M.~M\'{e}tivier, \emph{Stochastic partial differential equations in infinite dimensional spaces}, Scuola Normale Superiore, Pisa, 1988. 

\bibitem{Mohan+Sritharan'16} M. T. Mohan, S. S. Sritharan, \emph{Stochastic Euler equations of fluid dynamics with L\'{e}vy noise}, Asymptot. Anal., \bf 99\rm  , 67--103 (2016).


\bibitem{EM'13}  E. Motyl, \it Stochastic Navier-Stokes equations driven by L\'{e}vy noise in unbounded 3D domains, \rm  Potential Anal., \bf 38\rm , 863--912 (2013).

\bibitem{EM'14}  E. Motyl, \emph{Stochastic hydrodynamic-type evolution equations driven by L\'{e}vy noise in 3D unbounded domains - abstract framework and applications}, Stoch. Processes and their Appl. \bf 124\rm , 2052--2097 (2014).

\bibitem{EM'21_ArX_Hall-MHD_Gauss}  E. Motyl, \emph{Martingale solutions of the stochastic Hall-magnetohydrodynamics  equations on ${\mathbb{R} }^{3}$}, arXiv: 2109.06297v1 (2021). 

\bibitem{Peszat+Zabczyk'2007} S.~Peszat, J.~Zabczyk, \emph{\sc Stochastic Partial Differential  Equations with L\'{e}vy Noise}, Cambridge University Press, 2007. 

\bibitem{Revuz+Yor'99} D. Revuz, M. Yor,  \emph{\sc Continuous martingales and Brownian motion}, Springer-Verlag, 1999.

\bibitem{Rudin} W. Rudin, \emph{\sc Functional analysis}, McGraw-Hill Book  Company, New York, 1973.  

\bibitem{Sango'10} M. Sango, \emph{Magnetohydrodynamic turbulent flows: Existence results}, Physica D, Vol. 239, 12, 912--923 (2010).                     

\bibitem{Sermange+Temam'83} M. Sermange, R. Temam, \emph{Some mathematical questions related to the M.H.D. equations}, Comm. Pure Appl. Math., 36, pp. 634--664 (1983).

\bibitem{Stein'1970} E. Stein, \emph{\sc Singular Integrals and Differentiability Properties of Functions}, Princeton Univ. Press, 1970.   

\bibitem{Strauss'66} W. A. Strauss,  \emph{On continuity of functions with values in various Banach spaces}, Pacific J. Math., \bf 19\rm , 3, 543--555, (1966).   

\bibitem{Taylor_I'11} M.E. Taylor, \emph{\sc Partial Differential Equations I}, Springer, 2011. 

\bibitem{Temam'79} R. Temam, \emph{\sc Navier-Stokes equations. Theory and numerical analysis}, North Holland Publishing Company, Amsterdam - New York - Oxford, 1979.

\bibitem{Vishik+Fursikov'88} M.J. Vishik, A.V. Fursikov, \emph{\sc Mathematical Problems of Statistical Hydromechanics}, Kluwer Academic Publishers, Dordrecht, 1988. 

\bibitem{Yamazaki'17} K. Yamazaki, \emph{Stochastic Hall-magneto-hydrodynamics system in three and two and a half dimensions}, J. Stat. Phys., \bf 166\rm , 368--397 (2017).

\bibitem{Yamazaki'19a} K. Yamazaki, \emph{Remarks on the three and two and a half dimensional Hall-magnetohydrodynamics system: deterministic and stochastic cases},  Complex Analysis and its Synergies, \bf 5\rm , doi.org/10.1007/s40627-019-0033-5 (2019). 

\bibitem{Yamazaki'19}  K. Yamazaki, \emph{Ergodicity of a Galerkin approximation of three-dimensional magnetohydrodynamics system forced by a degenerate noise}, Stochastics, Vol. \bf 91\rm , No. 1, 114--142 (2019).

\bibitem{Yamazaki+Mohan'19} K. Yamazaki,  M.T. Mohan,  \emph{Well-posedness of Hall-magnetohydrodynamics system forced by L\'{e}vy noise}, Stoch PDE: Anal Comp  \bf 7\rm , pp. 331--378 (2019).


\bibitem{Zhu+Brze+Haus'17} J. Zhu, Z. Brze\'{z}niak, E. Hausenblas, \emph{Maximal inequalities for stochastic convolutions driven by compensated Poisson random measures in Banach spaces}, Ann. l'Inst. Henri Poincar\'{e} - Probab. et Stat. \bf 53\,\rm (2)  937--956 (2017).
\end{thebibliography}
\end{document}